\documentclass[12pt]{article}
\usepackage{amsthm,amsfonts,amssymb,amsmath,comment}

\usepackage{stmaryrd}
\usepackage{pifont}
\usepackage{cite,hyperref}
\usepackage{epsfig}
\usepackage{url}
\usepackage{xcolor,tikz}
\usepackage{tabularray}
\usepackage{lscape}
\usepackage{longtable}
\usetikzlibrary{matrix}
\usetikzlibrary{calc}
\usepackage{multicol}
\usepackage{fullpage}
\usepackage[shortlabels]{enumitem}
\numberwithin{equation}{section}
\setlist{leftmargin=3\parindent,labelindent=3\parindent}
\setlist[enumerate]{%
  leftmargin=3\parindent,%
  align=left,%
  labelwidth=3\parindent,%
  labelsep=0pt%
}
\setlist[enumerate,1]{%
  label={\normalfont (\thesection.\arabic{equation})}, ref={\normalfont \thesection.\arabic{equation}},
  resume%
}
\allowdisplaybreaks

\newtheorem{thm}[equation]{Theorem}
\newtheorem{cor}[equation]{Corollary}
\newtheorem{lem}[equation]{Lemma}

\newtheorem{claim}[equation]{Claim}

\newtheorem{conj}[equation]{Conjecture}

\newtheorem{obs}[equation]{Observation}
\theoremstyle{definition}
\newtheorem{defn}[equation]{Definition}
\newtheorem{ex}[equation]{Example}

\newtheorem{const}[equation]{Construction}

\theoremstyle{remark}

\DeclareMathOperator{\Hom}{Hom}

\DeclareMathOperator{\rad}{rad}
\DeclareMathOperator{\rng}{rng}
\DeclareMathOperator{\tr}{tr}
\DeclareMathOperator{\as}{as}
\DeclareMathOperator{\st}{st}

\DeclareTextCompositeCommand{\v}{OT1}{l}{l\nobreak\hspace{-.1em}'}

\title{Sidorenko-Type Inequalities for Pairs of Trees}
\author{Natalie Behague\thanks{Research supported by a PIMS Postdoctoral Fellowship. E-mail: \texttt{Natalie.Behague@warwick.ac.uk}. Current Affiliation: Mathematics Institute, University of Warwick, Coventry, UK.}${\text{ }}^{\text{,}}$\footnotemark[6]
\and
Gabriel Crudele\thanks{E-mail: {\tt gabrielcrudele1@gmail.com}. Current Affiliation: Department of Mathematics and Statistics, McGill University, Montr\'eal, Canada.}${\text{ }}^{\text{,}}$\footnotemark[6]
\and
Jonathan A. Noel\thanks{Research supported by NSERC Discovery Grant RGPIN-2021-02460, NSERC Early Career Supplement DGECR-2021-00024 and a Start-Up Grant from the University of Victoria. E-mail: {\tt noelj@uvic.ca}.}${\text{ }}^{\text{,}}$\thanks{Current Affiliation: Department of Mathematics and Statistics, University of Victoria, Victoria, Canada.} 
\and 
Lina M. Simbaqueba\thanks{Research supported by a Mitacs Globalink Research Internship. E-mail: {\tt lmsimbaquebam@uvic.ca}.}${\text{ }}^{\text{,}}$\footnotemark[4]${\text{ }}^{\text{,}}$\thanks{This research was completed while the first and second authors were affiliated with the Department of Mathematics and Statistics, University of Victoria, Victoria, Canada and the fourth author was affiliated with Departmento de Matem\'aticas, Universidad Nacional de Colombia, Bogot\'a, Colombia. } }

\begin{document}

\maketitle

\begin{abstract}
Given two non-empty graphs $H$ and $T$, write $H\succcurlyeq T$ to mean that $t(H,G)^{|E(T)|}\geq t(T,G)^{|E(H)|}$ for every graph $G$, where $t(\cdot,\cdot)$ is the homomorphism density function. We obtain various necessary and sufficient conditions for two trees $H$ and $T$ to satisfy $H\succcurlyeq T$ and determine all such pairs on at most 8 vertices. This extends results of Leontovich and Sidorenko from the 1980s and 90s. Our approach applies an information-theoretic technique to reduce the problem of showing that $H\succcurlyeq T$ for two forests $H$ and $T$ to solving a linear program of Kopparty and Rossman~\cite{KoppartyRossman11}. We also characterize trees $H$ which satisfy $H\succcurlyeq S_k$ or $H\succcurlyeq P_4$, where $S_k$ is the $k$-vertex star and $P_4$ is the $4$-vertex path and resolve a problem of Csikv\'ari and Lin~\cite{CsikvariLin15}. 
\end{abstract}

\section{Introduction}

A \emph{homomorphism} from a (simple undirected) graph $H$ to a graph $G$ is a function $\varphi:V(H)\to V(G)$ such that, for every $uv\in E(H)$, we have $\varphi(u)\varphi(v)\in E(G)$. We write $\varphi:H\to G$ to mean that $\varphi$ is a homomorphism from $H$ to $G$. Given a homomorphism $\varphi:H\to G$ and an edge $e=uv\in E(H)$, we let $\varphi(e)$ denote the edge $\varphi(u)\varphi(v)$ of $G$. Define $\Hom(H,G)$ to be the set of all homomorphisms from $H$ to $G$ and $\hom(H,G)$ to be the number of such homomorphisms. The \emph{homomorphism density} of $H$ in $G$ is defined to be
\[t(H,G):=\frac{\hom(H,G)}{v(G)^{v(H)}}\]
where, for any graph $H$, we let $v(H):=|V(H)|$ and $e(H):=|E(H)|$. It can be equivalently interpreted as the probability that a random function from $V(H)$ to $V(G)$ is a homomorphism. The homomorphism density function is of fundamental importance throughout extremal combinatorics and, especially, in the theory of graph limits~\cite{Lovasz12}. 

We study a binary relation defined in terms of the homomorphism density function. Given two non-empty\footnote{Throughout the paper, a \emph{non-empty} graph is a graph whose edge set is non-empty. Our only reason for restricting our attention to non-empty graphs is to avoid expressions that may evaluate to $0^0$.} graphs $H$ and $T$, we write $H\succcurlyeq T$ to mean that $t(H,G)^{e(T)}\geq t(T,G)^{e(H)}$ for every graph $G$. Note that, whenever such an inequality holds, it is tight. For example, for any fixed $p\in [0,1]$ and any graph $H$, we have $\lim_{n\to\infty}t(H,G(n,p)) = p^{e(H)}$ with probability one, where $G(n,p)$ is the standard Erd\H{o}s--R\'enyi random graph. Technically, the relation $\succcurlyeq$ fails to be antisymmetric on the set of all non-empty graphs since $t(H\sqcup H,G) = t(H,G)^2$ for any graphs $H$ and $G$, where $H\sqcup H$ denotes the disjoint union of two copies of $H$. However, we show that $\succcurlyeq$ is a partial order on the set of non-empty connected graphs. This is derived from a classical result on counts of homomorphisms; it is likely that this theorem has been observed earlier, but we have been unable to find a reference for it.

\begin{thm}
\label{th:poset}
The relation $\succcurlyeq$ is a partial order on the set of non-empty connected graphs. 
\end{thm}

Many important results and open problems in extremal combinatorics can be phrased in terms of the relation $\succcurlyeq$. For example, the famous Sidorenko Conjecture~\cite{Sidorenko89} asserts that $H\succcurlyeq K_2$ for every non-empty bipartite graph $H$. Most of the known cases of Sidorenko's Conjecture are proved by deriving inequalities of the form $H\succcurlyeq T$ where $T$ is known to satisfy the conjecture either by induction or by an earlier result, and then using transitivity of $\succcurlyeq$; see, e.g.~\cite{ConlonLee21,Szegedy15,ConlonKimLeeLee18}. One of the important precursors to Sidorenko's Conjecture is the well-known Blakley--Roy Inequality, which implies Sidorenko's Conjecture for paths~\cite{MulhollandSmith59,AtkinsonWattersonMoran60,BlakleyRoy65}. An early result on Sidorenko's Conjecture is that it is true for forests~\cite{Sidorenko91}.

\begin{thm}[Sidorenko~\cite{Sidorenko91}]
\label{th:Sid}
If $H$ is a non-empty forest, then $H\succcurlyeq K_2$. 
\end{thm}

Erd\H{o}s and Simonovits~\cite{ErdosSimonovits82} conjectured that $P_{n+2}\succcurlyeq P_{n}$ for all $n\geq2$, where $P_n$ denotes the path with $n$ vertices. When $n$ is odd, this follows from the following more general result of  Godsil (see~\cite{ErdosSimonovits82}).

\begin{thm}[Godsil]
\label{th:Godsil}
If $n$ is odd, then $P_n\succcurlyeq P_m$ for all $3\leq m\leq n$. 
\end{thm}

The remaining cases of the conjecture of Erd\H{o}s and Simonovits~\cite{ErdosSimonovits82} were settled by Sa\u{g}lam~\cite{Saglam18}; for an alternative proof using ideas that are closely related to the approach of this paper,  see~\cite{BlekhermanRaymond20}.

\begin{thm}[Sa\u{g}lam~\cite{Saglam18}]
\label{th:Saglam}
$P_{n+2}\succcurlyeq P_n$ for all $n\geq2$. 
\end{thm}

Kopparty and Rossman~\cite{KoppartyRossman11} introduced the ``homomorphism domination exponent'' of two graphs $H$ and $T$ to be the maximum $c>0$ such that $\hom(H,G)\geq \hom(T,G)^c$ for all graphs $G$. In~\cite{BlekhermanRaymond22}, Blekherman and Raymond introduce a  ``tropicalization'' technique and apply it to prove various inequalities of a similar flavour. Stoner~\cite{Stoner22+} investigated the related problem of determining the largest exponent $c$ such that $t(H,G)\geq t(T,G)^c$ for every graph $G$. Clearly, $H\succcurlyeq T$ if and only if this maximum exponent is $e(H)/e(T)$. Very recently, Conlon and Lee~\cite{ConlonLee23+} initiated the study of graphs $H$ such that $H\succcurlyeq T$ for every subgraph $T$ of $H$, motivated by connections to the notion of weakly norming graphs~\cite{ConlonLee17,Hatami10,Sidorenko20,GarbeHladkyLee22,Kral+19}.  In particular, they showed that every perfect $d$-regular tree has this property.

Our main focus in this paper is on the restriction of $\succcurlyeq$ to the set of all trees. Hoffman~\cite{Hoffman67} proved that $S_k\succcurlyeq P_k$ for all $k\geq2$, where $S_k$ is the $k$-vertex star. A generalization of this, due to Sidorenko~\cite{Sidorenko85}, determines the relation $\succcurlyeq$ on pairs of trees $H$ and $T$ with $v(H)=v(T)$ obtained from paths by adding pendant vertices at the end. The problem of determining $\succcurlyeq$ for general trees on the same number of vertices was considered by Leontovich~\cite{Leontovich89} back in 1989, who proved several general results and determined the poset restricted to $6$-vertex trees. Further results of a similar flavour were obtained by Sidorenko~\cite{Sidorenko94}; together with the earlier theorems, this determined the poset for $7$-vertex trees apart from six pairs $(H,T)$ for which it remained unknown  whether $H\succcurlyeq T$. We obtain the full structure of this poset on the set of all trees on at most 8 vertices. We note that it would not be hard, in principle, to use the ideas in this paper to compute the relation $\succcurlyeq$ for all trees on at most a slightly larger number of vertices (say, 9 or 10). However, the sheer number of pairs of trees would make it challenging to present such a result in writing.

\begin{thm}
The restriction of $\succcurlyeq$ to the set of trees on at most 8 vertices is fully described by the tables in Section~\ref{sec:smallTrees}. 
\end{thm}

We also classify the trees $H$ which satisfy $H\succcurlyeq S_k$ for $k\geq3$ or $H\succcurlyeq P_4$. For a bipartite graph $H$, let $\sigma(H)$ be the minimum of $|A|$ over all bipartitions $(A,B)$ of $H$. 

\begin{thm}
\label{th:star}
Let $k\geq3$ and let $H$ be a non-empty tree. Then $H\succcurlyeq S_k$ if and only if $e(H)\geq (k-1)\sigma(H)$.
\end{thm}

\begin{thm}
\label{th:P4}
Let $H$ be a tree. Then $H\succcurlyeq P_4$ if and only if $H$ has at least four vertices.
\end{thm}

Intuitively, the poset given by $\succcurlyeq$ on the set of all $k$-vertex trees for small $k$ tends to have ``star-like'' trees near the top and ``path-like'' trees near the bottom. In fact, as proven by Sidorenko~\cite[Theorem~1.2]{Sidorenko94}, for all $k$, the unique maximal element is $S_k$. This result has found applications in studying noise sensitivity of boolean functions~\cite{GopalanServedioWigderson16} and analyzing clustering coefficients in complex networks~\cite{Kurauskas22};  alternative proofs of Sidorenko's result appear in~\cite{CsikvariLin14,LevinPeres17,LuchtrathMonch24+}. Right below $S_k$ is the tree $S_k'$ obtained from $S_{k-1}$ by subdividing one edge; more precisely, every $k$-vertex tree $H$ apart from $S_k$ satisfies $H\preccurlyeq S_k'\preccurlyeq S_k$~\cite{Sidorenko94}. For $k\leq 6$, the unique minimal element is $P_k$, as is discussed in~\cite{Leontovich89,Sidorenko94}. However, Leontovich~\cite{Leontovich89} showed that there is a $7$-vertex tree $H$ with $H\not\succcurlyeq P_7$ and so $P_7$ is not the unique minimal element among $7$-vertex trees; see also~\cite[Remark~4.13]{CsikvariLin14}. Our next theorem says that, somewhat surprisingly, there are some very ``star-like'' trees $H$ which satisfy $H\not\succcurlyeq P_k$. We say that a tree $H$ is a \emph{near-star} if it can be obtained from a star by subdividing edges in such a way that each edge of the star is subdivided at most once.   

\begin{thm}
\label{th:pathsNotUnique}
Let $H$ be a $k$-vertex near-star with $\ell$ leaves. If $\frac{k+1}{2}\leq \ell\leq k-3$, then $H\not\succcurlyeq P_k$.
\end{thm}

Csikv\'ari and Lin~\cite[Problem~5.2]{CsikvariLin15} asked for the smallest $f(k)\geq 2$ such that every tree $H$ with $k$ vertices and at least $f(k)$ leaves satisfies $H\succcurlyeq P_k$ and also asked whether $f(k)\leq 4$ holds for all $k$. We can now give an exact formula for $f(k)$ for all $k\geq2$, thereby resolving both questions. If $2\leq k\leq 6$, then $H\succcurlyeq P_k$ for all $k$-vertex trees, as was shown in~\cite{Leontovich89}, and so $f(k)=2$. If $k\geq7$, then Theorem~\ref{th:pathsNotUnique} implies that $f(k)\geq k-2$. The only $k$-vertex trees with at least $k-2$ leaves are $S_k$ and $S_k'$; the facts that $S_k\succcurlyeq P_k$ and $S_k'\succcurlyeq P_k$ seem to have been first proven in~\cite{Hoffman67} and~\cite{Sidorenko89}, respectively. Thus, $f(k)=k-2$ for $k\geq7$. 

The rest of the paper is organized as follows. We start in  Section~\ref{sec:basics} by proving Theorem~\ref{th:poset} and describing the linear programs, which are based on more general linear programs from~\cite{KoppartyRossman11}, that are key to proving most of the results in this paper. In the same section, we also prove that a dual-feasible point for which the dual objective function is less than $e(T)$ can be used to certify that $H\not\succcurlyeq T$ for graphs $H$ and $T$. In Section~\ref{sec:entropy}, we use standard facts from information theory to show that, if $H$ and $T$ are forests such that the value of primal linear program from the previous section is $e(T)$, then $H\succcurlyeq T$. We note that the proofs of these results on the linear program and its dual are included for completeness, as they can be derived directly from the results of~\cite{KoppartyRossman11}. As a basic application of the linear programming approach, in Section~\ref{sec:star}, we prove a few basic necessary and sufficient conditions for $H\succcurlyeq T$ to hold and use them to prove Theorem~\ref{th:star}. We then build upon these ideas to prove Theorem~\ref{th:P4} in Section~\ref{sec:P4}. In Section~\ref{sec:nec}, we provide additional necessary conditions (some old and some new) and prove Theorem~\ref{th:pathsNotUnique}. 
Finally, in Section~\ref{sec:smallTrees}, we use the results built up in the paper, together with some ad hoc constructions, to determine the full structure of the poset $\succcurlyeq$ restricted to trees on at most 8 vertices.  The data required to verify the ad hoc constructions is given in Appendices~\ref{app:Primalcertificates} and~\ref{app:Dualcertificates}.  A list of all such trees is provided in Appendix~\ref{app:smallTrees} for reference.
We conclude the paper in Section~\ref{sec:conclusion} with a few closing remarks and open problems.

\section{Basic Properties and Linear Programs}
\label{sec:basics}

We begin with the proof of Theorem~\ref{th:poset}; i.e. we show that $\succcurlyeq$ is a partial order on the set of all non-empty connected graphs. The proof that $\succcurlyeq$ is antisymmetric will apply the following lemma, which we have borrowed from~\cite[Corollary~5.45]{Lovasz12}. The ideas behind its proof go back to the work of Erd\H{o}s, Lov\'asz and Spencer~\cite{ErdosLovaszSpencer79}.

\begin{lem}[See~{\cite[Corollary~5.45]{Lovasz12}}]
\label{lem:linIndep}
The functions $\hom(F,\cdot)$ are linearly independent, where $F$ ranges over all graphs up to isomorphism. 
\end{lem}

Using Lemma~\ref{lem:linIndep}, we prove Theorem~\ref{th:poset}.

\begin{proof}[Proof of Theorem~\ref{th:poset}]
It is clear that $\preccurlyeq$ is reflexive. For transitivity, suppose that $F\succcurlyeq H$ and $H\succcurlyeq T$. Then, for every graph $G$,
\[t(F,G)\geq t(H,G)^{\frac{e(F)}{e(H)}} \geq t(T,G)^{\frac{e(H)}{e(T)}\cdot \frac{e(F)}{e(H)}} = t(T,G)^{\frac{e(F)}{e(T)}}\]
and so $F\succcurlyeq T$. Finally, suppose that $H$ and $T$ are non-empty connected graphs such that $H\succcurlyeq T$ and $H\preccurlyeq T$. Then we have
\[t(H,G)^{e(T)}=t(T,G)^{e(H)}\]
for every graph $G$. Using the definition of homomorphism density and the fact that $\hom(K_1,G)=v(G)$ for any graph $G$, this is equivalent to
\[\hom(H,G)^{e(T)}\hom(K_1,G)^{v(T)e(H)} = \hom(T,G)^{e(H)}\hom(K_1,G)^{v(H)e(T)}\]
for every graph $G$. In other words,
\[\hom((e(T)\cdot H)\sqcup(v(T)e(H)\cdot K_1),G)=\hom((e(H)\cdot T)\sqcup(v(H)e(T)\cdot K_1),G)\]
where, for any non-negative integer $k$ and graph $F$, we let $k\cdot F$ denote the disjoint union of $k$ copies of $F$. By Lemma~\ref{lem:linIndep}, this implies that $(e(T)\cdot H)\sqcup(v(T)e(H)\cdot K_1)$ and $(e(H)\cdot T)\sqcup(v(H)e(T)\cdot K_1)$ are isomorphic. So, since $H$ and $T$ are connected, $H$ and $T$ must be isomorphic. This completes the proof. 
\end{proof}

Next, we present a linear program which plays a key role in this paper. We should make it clear that this linear program can be obtained by specializing a more general linear program of Kopparty and Rossman~\cite{KoppartyRossman11} to the case of forests; therefore, in particular, this linear program is not new. However, rather than deriving the results about this linear program by specializing the (more complicated) linear program in~\cite{KoppartyRossman11}, we have decided to prove them from scratch in this simpler setting for the sake of completeness and because we believe that the ideas in the proof will be enlightening to any reader who is not already familiar with the argument in~\cite{KoppartyRossman11}. We remark that the linear program from~\cite{KoppartyRossman11} has been applied in several other recent papers on homomorphism density inequalities, such as~\cite{BlekhermanRaymond20,BlekhermanRaymond22,DascaluRaymond24}, and similar entropy-theoretic ideas are at the core of much of the recent progress on Sidorenko's Conjecture and related problems; see, e.g.,~\cite{ConlonKimLeeLee18,Szegedy15,ConlonLee17,Behague+22+,GrzesikLeeLidickyVolec22,Lee21}.

Let $H$ and $T$ be graphs. For each homomorphism $\varphi:H\to T$, let $\mu_\varphi:V(T)\cup E(T)\to \{0,1,2,\dots,\}$ be the function defined by $\mu_\varphi(v):=|\varphi^{-1}(v)|$ for each $v\in V(T)$ and $\mu_\varphi(e):=|\varphi^{-1}(e)|$ for each $e\in E(T)$. Let $w:\Hom(H,T)\to \mathbb{R}$ be a weighting on the homomorphisms from $H$ to $T$; we view the values $w(\varphi)$ for $\varphi\in\Hom(H,T)$ as variables.  We let $LP(H,T)$ be the following linear program:
\begin{align}
\label{eq:OBJ}\text{maximize } & \sum_{e\in E(T)}\sum_{\varphi:H\to T}\mu_\varphi(e)\cdot w(\varphi) & \\
\label{eq:Econstraint}\text{subject to } & \sum_{\varphi:H\to T}\mu_\varphi(e)\cdot w(\varphi)\leq 1 & \forall e\in E(T),\\
\label{eq:Vconstraint} & \sum_{\varphi:H\to T}\mu_\varphi(v)\cdot w(\varphi)\leq 1 & \forall v\in V(T),\\
\label{eq:+constraint} & w(\varphi)\geq0 & \forall \varphi\in \Hom(H,T).
\end{align}
Summing the constraint \eqref{eq:Econstraint} over all edges of $T$ tells us that the value of $LP(H,T)$ is at most $e(T)$. The next lemma says that, if $H$ and $T$ are forests such that $LP(H,T)$ attains the value $e(T)$, then $H\succcurlyeq T$. The proof of this lemma will require a fair bit of background from information theory (specifically, entropy), and so we postpone it until the next section. We note that this lemma can be derived from results in~\cite{KoppartyRossman11} but we will give a self-contained proof for expository purposes. 

\begin{lem}[See~{\cite[Theorem~3.3]{KoppartyRossman11}}]
\label{lem:LP}
If $H$ and $T$ are forests such that the value of $LP(H,T)$ is equal to $e(T)$, then $H\succcurlyeq T$.
\end{lem}

Before moving on, let us illustrate the applicability of Lemma~\ref{lem:LP} with a quick example that is inspired by the argument in~\cite{BlekhermanRaymond20}. 

\begin{ex}
Let $H=P_6$ and $T=P_4$. Denote the vertices of $P_6$ by $v_1,\dots,v_6$ and the vertices of $P_4$ by $u_1,\dots,u_4$, labelled in the order that they appear on the path. We let $\varphi_1,\varphi_2$ and $\varphi_3$ be homomorphisms defined by
\[(\varphi_1(v_1),\varphi_1(v_2),\varphi_1(v_3),\varphi_1(v_4),\varphi_1(v_5),\varphi_1(v_6)) = (u_1,u_2,u_1,u_2,u_3,u_4),\]
\[(\varphi_2(v_1),\varphi_2(v_2),\varphi_2(v_3),\varphi_2(v_4),\varphi_2(v_5),\varphi_2(v_6))  = (u_1,u_2,u_3,u_2,u_3,u_4)\]
and
\[(\varphi_3(v_1),\varphi_3(v_2),\varphi_3(v_3),\varphi_3(v_4),\varphi_3(v_5),\varphi_3(v_6))  = (u_1,u_2,u_3,u_4,u_3,u_4).\]
These homomorphisms are depicted in Figure~\ref{fig:P4P6}. 

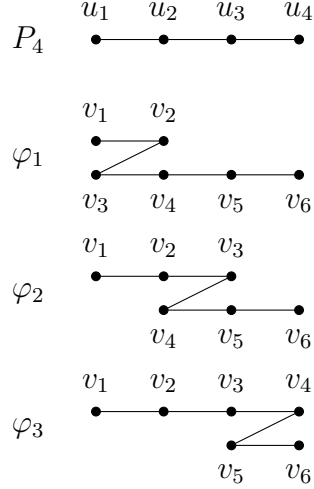
\begin{figure}[htbp]
\begin{center}
\begin{tikzpicture}[scale=0.9]
  \node (P4) at (-1,0) {$P_4$};
  \node[fill=black,circle,minimum size=0.05cm,scale=0.3,draw,label={[label distance=0.05cm]90:$u_1$}] (u1) at (0,0) {};
  \node[fill=black,circle,minimum size=0.05cm,scale=0.3,draw,label={[label distance=0.05cm]90:$u_2$}] (u2) at (1,0) {};
  \node[fill=black,circle,minimum size=0.05cm,scale=0.3,draw,label={[label distance=0.05cm]90:$u_3$}] (u3) at (2,0) {};
  \node[fill=black,circle,minimum size=0.05cm,scale=0.3,draw,label={[label distance=0.05cm]90:$u_4$}] (u4) at (3,0) {};
  \draw (u1) -- (u2) -- (u3) -- (u4);
  
  \node (phi1) at (-1,-1.75) {$\varphi_1$};
  \node[fill=black,circle,minimum size=0.05cm,scale=0.3,draw,label={[label distance=0.05cm]90:$v_1$}] (v11) at (0,-1.5) {};
  \node[fill=black,circle,minimum size=0.05cm,scale=0.3,draw,label={[label distance=0.05cm]90:$v_2$}] (v21) at (1,-1.5) {};
  \node[fill=black,circle,minimum size=0.05cm,scale=0.3,draw,label={[label distance=0.05cm]270:$v_3$}] (v31) at (0,-2) {};
  \node[fill=black,circle,minimum size=0.05cm,scale=0.3,draw,label={[label distance=0.05cm]270:$v_4$}] (v41) at (1,-2) {};
  \node[fill=black,circle,minimum size=0.05cm,scale=0.3,draw,label={[label distance=0.05cm]270:$v_5$}] (v51) at (2,-2) {};
  \node[fill=black,circle,minimum size=0.05cm,scale=0.3,draw,label={[label distance=0.05cm]270:$v_6$}] (v61) at (3,-2) {};
  \draw (v11) -- (v21) -- (v31) -- (v41) -- (v51) -- (v61);
  
  \node (phi2) at (-1,-3.75) {$\varphi_2$};
  \node[fill=black,circle,minimum size=0.05cm,scale=0.3,draw,label={[label distance=0.05cm]90:$v_1$}] (v12) at (0,-3.5) {};
  \node[fill=black,circle,minimum size=0.05cm,scale=0.3,draw,label={[label distance=0.05cm]90:$v_2$}] (v22) at (1,-3.5) {};
  \node[fill=black,circle,minimum size=0.05cm,scale=0.3,draw,label={[label distance=0.05cm]90:$v_3$}] (v32) at (2,-3.5) {};
  \node[fill=black,circle,minimum size=0.05cm,scale=0.3,draw,label={[label distance=0.05cm]270:$v_4$}] (v42) at (1,-4) {};
  \node[fill=black,circle,minimum size=0.05cm,scale=0.3,draw,label={[label distance=0.05cm]270:$v_5$}] (v52) at (2,-4) {};
  \node[fill=black,circle,minimum size=0.05cm,scale=0.3,draw,label={[label distance=0.05cm]270:$v_6$}] (v62) at (3,-4) {};
  \draw (v12) -- (v22) -- (v32) -- (v42) -- (v52) -- (v62);
  
  \node (phi3) at (-1,-5.75) {$\varphi_3$};
  \node[fill=black,circle,minimum size=0.05cm,scale=0.3,draw,label={[label distance=0.05cm]90:$v_1$}] (v13) at (0,-5.5) {};
  \node[fill=black,circle,minimum size=0.05cm,scale=0.3,draw,label={[label distance=0.05cm]90:$v_2$}] (v23) at (1,-5.5) {};
  \node[fill=black,circle,minimum size=0.05cm,scale=0.3,draw,label={[label distance=0.05cm]90:$v_3$}] (v33) at (2,-5.5) {};
  \node[fill=black,circle,minimum size=0.05cm,scale=0.3,draw,label={[label distance=0.05cm]90:$v_4$}] (v43) at (3,-5.5) {};
  \node[fill=black,circle,minimum size=0.05cm,scale=0.3,draw,label={[label distance=0.05cm]270:$v_5$}] (v53) at (2,-6) {};
  \node[fill=black,circle,minimum size=0.05cm,scale=0.3,draw,label={[label distance=0.05cm]270:$v_6$}] (v63) at (3,-6) {};
  \draw (v13) -- (v23) -- (v33) -- (v43) -- (v53) -- (v63);
\end{tikzpicture}
\end{center}
\vspace{-1em}
\caption{A depiction of the homomorphisms $\varphi_1,\varphi_2$ and $\varphi_3$ from $P_6$ to $P_4$. Each vertex of $P_6$ is drawn directly below the vertex of $P_4$ that it is mapped to.}
\label{fig:P4P6}
\end{figure}

Set $w(\varphi_1)=w(\varphi_2)=w(\varphi_3)=1/5$ and $w(\varphi)=0$ for every other homomorphism $\varphi$  from $P_6$ to $P_4$. Then, by construction, for each edge $e\in P_4$, we have
\[\sum_{\varphi:P_6\to P_4}\mu_\varphi(e)\cdot w(\varphi) = 1.\]
Also, we have
\[\sum_{\varphi:P_6\to P_4}\mu_\varphi(u_1)\cdot w(\varphi) = \sum_{\varphi:P_6\to P_4}\mu_\varphi(u_4)\cdot w(\varphi)=4/5\]
and
\[\sum_{\varphi:P_6\to P_4}\mu_\varphi(u_2)\cdot w(\varphi) = \sum_{\varphi:P_6\to P_4}\mu_\varphi(u_3)\cdot w(\varphi)=1.\]
Therefore, all of the constraints of $LP(P_6,P_4)$ are satisfied. Since \eqref{eq:Econstraint} is tight for every edge of $P_4$, the value of $LP(P_6,P_4)$ is precisely equal to $e(P_4)=3$ and so, by Lemma~\ref{lem:LP}, the inequality $P_6\succcurlyeq P_4$ holds. In~\cite{BlekhermanRaymond20}, this construction is generalized to give a beautiful alternative proof of Theorem~\ref{th:Saglam}.
\end{ex}

Let us now discuss the dual of $LP(H,T)$, which we denote by $DLP(H,T)$. We let $y:V(T)\cup E(T)\to \mathbb{R}$ be a weighting of the vertices and edges of $T$, and view $y(v)$ and $y(e)$ as variables for all $v\in V(T)$ and $e\in E(T)$. Then $DLP(H,T)$ is as follows:
\begin{align}
\label{eq:DOBJ}\text{minimize } & \sum_{v\in V(T)}y(v) + \sum_{e\in E(T)}y(e) & \\
\label{eq:DhomConstraint}\text{subject to } & \sum_{v\in V(T)}\mu_\varphi(v)\cdot y(v) + \sum_{e\in E(T)}\mu_\varphi(e)\cdot y(e)\geq e(H) & \forall \varphi\in\Hom(H,T),\\
\label{eq:Dv+} & y(v)\geq 0 & \forall v\in V(T),\\
\label{eq:De+} & y(e)\geq0 & \forall e\in E(T).
\end{align}

By setting $y(v)=0$ for all $v\in V(T)$ and $y(e)=1$ for all $e\in E(T)$, we see that all constraints of $DLP(H,T)$ are satisfied and, for this choice of variables, its objective function is equal to $e(T)$. The following lemma shows that, if the value of $DLP(H,T)$ is less than $e(T)$, then $H\not\succcurlyeq T$. Note that this lemma holds for all graphs $H$ and $T$; i.e. unlike Lemma~\ref{lem:LP}, it does not require $H$ and $T$ to be forests. The proof of this lemma does not require any significant background and so we can present it immediately. Again, this lemma can be derived from results in~\cite{KoppartyRossman11}. 

\begin{lem}[See~\cite{KoppartyRossman11}]
\label{lem:DLP}
If $H$ and $T$ are non-empty graphs such that the value of $DLP(H,T)$ is less than $e(T)$, then $H\not\succcurlyeq T$.
\end{lem}

\begin{proof}
Let $y:V(T)\cup E(T)\to \mathbb{R}$ be a certificate that the value of $DLP(H,T)$ is less than $e(T)$. Since all coefficients in the objective function and constraints of $DLP(H,T)$ are integral, we can assume that $y(v)$ and $y(e)$ are rational for all $v\in V(T)$ and $e\in E(T)$. Let $q$ be a positive integer so that $y(v)\cdot q$ and $y(e)\cdot q$ are integers for all $v\in V(T)$ and $e\in E(T)$. Let $\xi:=\sum_{v\in V(T)}y(v)+\sum_{e\in E(T)}y(e)$, which is less than $e(T)$ by our choice of $y$. We let $m$ be an integer chosen large enough that $m>q\cdot\xi$.

For each integer $N$, let $G_N$ be a (random) graph constructed as follows. For each $v\in V(T)$, let $S_v$ be a set of $N^{m-y(v)\cdot q}$ vertices. Note that, by our choice of $m$ and $q$, we have that $N^{m-y(v)\cdot q}$ is a positive integer, and so this choice is well-defined. Also, let $S_0$ be a set of precisely
\[v(T)\cdot N^m - \sum_{v\in V(T)} |S_v|\] 
vertices. The above expression evaluates to a non-negative integer and so $S_0$ is well-defined. We assume that all of the sets $S_v$ for $v\in V(T)\cup\{0\}$ are pairwise disjoint. Note that the sole purpose of $S_0$ is to ensure that $G_N$ has precisely $v(T)\cdot N^m$ vertices, which will be useful for simplifying some of our calculations. For each $uv=e\in E(T)$, between $S_u$ and $S_v$, we place the edges of a random bipartite graph with bipartition $(S_u,S_v)$ in which each edge is present with probability $N^{-y(e)\cdot q}$ independent of all other edges. Our goal is to prove that, with probability one, there exists $N$ such that
\[t(H,G_N)^{e(T)} < t(T,G_N)^{e(H)}\]
which will imply that $H\not\succcurlyeq T$ and complete the proof.

Given a function $\psi:V(T)\to V(G_N)$ with the property that $\psi(v)\in S_v$ for all $v\in V(T)$, the probability that $\psi$ is a homomorphism from $T$ to $G$ is precisely equal to $\prod_{e\in E(T)}N^{-y(e)\cdot q}$. Also, the number of such functions is $\prod_{v\in V(T)}N^{m-y(v)\cdot q}$. Therefore, 
\[\mathbb{E}(\hom(T,G_N)) \geq \left(\prod_{v\in V(T)}N^{m-y(v)\cdot q}\right)\cdot \left(\prod_{e\in E(T)}N^{-y(e)\cdot q}\right) = N^{v(T)\cdot m - \xi\cdot q}.\]
Moreover, by standard concentration results, $\hom(T,G_N)$ is at least half of its expectation with probability $1-o(1)$, where the asymptotics are with respect to $N$ tending to infinity. Therefore, with probability close to $1$, we have
\begin{equation}\label{eq:Tlb}t(T,G_N) \geq\frac{1}{2}\cdot \frac{N^{v(T)\cdot m - \xi\cdot q}}{v(G_N)^{v(T)}} = \frac{1}{2}\cdot \frac{N^{v(T)\cdot m - \xi\cdot q}}{v(T)^{v(T)} N^{v(T)\cdot m}}  = \frac{ N^{-\xi \cdot q}}{2v(T)^{v(T)}}.\end{equation}

Now, let $\gamma:V(G_N)\setminus S_0\to V(T)$ be the function such that, for each $v\in V(T)$ and $x\in S_v$, we have $\gamma(x)=v$. For each homomorphism $\varphi:H\to T$, let $\Hom_\varphi(H,G_N)$ be the set of all homomorphisms $\psi:H\to G_N$ such that $\gamma\circ \psi = \varphi$ and let $\hom_\varphi(H,T)$ be its cardinality. Note that $\Hom(H,G_N) = \bigsqcup_{\varphi:H\to T}\Hom_\varphi(H,G_N)$ by the construction of $G_N$. Thus,
\[\hom(H,G_N) \leq \hom(H,T) \cdot \max\{\hom_\varphi(H,G_N): \varphi\in \Hom(H,T)\}.\]
Now, for fixed $\varphi\in \Hom(H,T)$ and a function $\psi:V(H)\to V(G_N)$ such that $\gamma\circ \psi=\varphi$, the probability that $\psi$ is a homomorphism is equal to $\prod_{e\in E(T)}N^{-\mu_\varphi(e)\cdot y(e)\cdot q}$. Therefore,
\[\mathbb{E}(\hom_\varphi(H,G_N))= \left(\prod_{v\in V(T)}\left(N^{ m-y(v)\cdot q}\right)^{\mu_\varphi(v)}\right)\cdot \left(\prod_{e\in E(T)}N^{-\mu_\varphi(e)\cdot y(e)\cdot q}\right).\]
By \eqref{eq:DhomConstraint}, this is at most
\[N^{v(H)\cdot m - e(H)\cdot q}.\]
Once again, by standard concentration inequalities, $\hom_\varphi(H,G_N)$ is at most twice its expectation with probability $1-o(1)$. Thus, putting all of this together, we have, with probability close to one,
\begin{equation}\label{eq:Hub}t(H,G_N)\leq \frac{2\cdot\hom(H,T)\cdot N^{v(H)\cdot m-e(H)\cdot q}}{v(G_N)^{v(H)}}\leq \frac{2\cdot v(T)^{v(H)}\cdot N^{v(H)\cdot m-e(H)\cdot q}}{v(T)^{v(H)}N^{v(H)\cdot m}}\leq 2\cdot N^{-e(H)\cdot q}.\end{equation}
Since $\xi<e(T)$, the combination of \eqref{eq:Tlb} and \eqref{eq:Hub} implies that $t(H,G_N)^{e(T)}<t(T,G_N)^{e(H)}$ holds with probability $1-o(1)$. This completes the proof.
\end{proof}

Putting Lemmas~\ref{lem:LP} and~\ref{lem:DLP} together, we get that the linear program $LP(H,T)$ exactly captures the property $H\succcurlyeq T$ in the case that $H$ and $T$ are forests.

\begin{thm}[See~{\cite[Theorem~3.3]{KoppartyRossman11}}]
If $H$ and $T$ are forests, then $H\succcurlyeq T$ if and only if the value of $LP(H,T)$ is equal to $e(T)$.
\end{thm}

\begin{proof}
If $LP(H,T)$ has value $e(T)$, then, by Lemma~\ref{lem:LP}, we get $H\succcurlyeq T$. So, suppose that the value of $LP(H,T)$ is not equal to $e(T)$. By summing \eqref{eq:Econstraint} over all edges of $T$, we see that the value of $LP(H,T)$ is at most $e(T)$. So, if its value is not equal to $e(T)$, then it must be less than $e(T)$. By the Strong Duality Theorem and the fact that $DLP(H,T)$ is the dual of $LP(H,T)$, the value of $DLP(H,T)$ is less than $e(T)$. By Lemma~\ref{lem:DLP}, we get that $H\not\succcurlyeq T$. This completes the proof.
\end{proof}

\section{Homomorphism Density Inequalities Via Entropy}
\label{sec:entropy}

The goal of this section is to prove Lemma~\ref{lem:LP}. As was mentioned earlier, this lemma is implied by results of~\cite{KoppartyRossman11} and so this proof is included for mainly expository purposes. For the proof, we need some basic properties of the entropy of a discrete random variable. We note that all of the properties of entropy that we use here are completely standard and can be found in many standard references on information theory or probabilistic combinatorics; see, e.g.~\cite[Chapter~15]{AlonSpencer16}.

\begin{defn}
Given a discrete random variable $X$, the \emph{range} of $X$ is defined to be
\[\rng(X):=\{x: \mathbb{P}(X=x)>0\}.\]
\end{defn}

\begin{defn}
The \emph{entropy} of a discrete random variable $X$ is defined to be
\[\mathbb{H}(X):=\sum_{x\in \rng(X)}\mathbb{P}(X=x)\log_2\left(\frac{1}{\mathbb{P}(X=x)}\right).\]
\end{defn}

A key fact is that the entropy of a discrete random variable $X$ with a given finite range is maximized by taking $X$ to be a uniformly random element of $\rng(X)$; this follows from a simple application of Jensen's Inequality to the function $x\log_2(x)$.

\begin{lem}[Maximality of the Uniform Distribution]
\label{lem:unif}
For any discrete random variable $X$ with finite range,
\[\mathbb{H}(X) \leq \log_2(|\rng(X)|).\]
Moreover, equality holds if and only if $\mathbb{P}(X=x)=\frac{1}{|\rng(X)|}$ for all $x\in \rng(X)$. 
\end{lem}

We also require the notion of conditional entropy.

\begin{defn}
Let $X$ and $Y$ be discrete random variables and let $y\in \rng(Y)$. Define 
\[\rng(X\mid Y=y) := \{x: \mathbb{P}(X=x\mid Y=y)>0\}.\]
\end{defn}

\begin{defn}
Given two discrete random variables $X$ and $Y$ and $y\in \rng(Y)$, the \emph{entropy of $X$ given $Y=y$} is defined to be
\[\mathbb{H}(X\mid Y=y):=\sum_{x\in \rng(X\mid Y=y)}\mathbb{P}(X=x\mid Y=y)\log_2\left(\frac{1}{\mathbb{P}(X=x\mid Y=y)}\right).\]
\end{defn}

\begin{defn}
Given discrete random variables $X$ and $Y$, the \emph{entropy of $X$ given $Y$} is
\[\mathbb{H}(X\mid Y):=\sum_{y\in \rng(Y)}\mathbb{P}(Y=y)\cdot\mathbb{H}(X\mid Y=y).\]
\end{defn}

An important property of entropy is the so called ``chain rule'' for representing the entropy of a joint random variable in terms of conditional entropy. 

\begin{lem}[Chain Rule of Entropy]
\label{lem:CR}
Let $X_1,\dots,X_m$ be discrete random variables. Then
\[\mathbb{H}(X_1,\dots,X_m) = \mathbb{H}(X_1) + \sum_{i=2}^m\mathbb{H}(X_i\mid X_1,\dots,X_{i-1}).\]
\end{lem}

We also need the notion of conditional independence of random variables. 

\begin{defn}
Given three discrete random variables $X,Y$ and $Z$, we say that $X$ and $Y$ are \emph{conditionally independent given $Z$} if 
\[\mathbb{P}(X=x\mid Z=z, Y=y) = \mathbb{P}(X=x\mid Z=z)\]
and
\[\mathbb{P}(Y=y\mid Z=z, X=x) = \mathbb{P}(Y=y\mid Z=z)\]
for all $z\in \rng(Z)$, $x\in \rng(X\mid Z=z)$ and $y\in \rng(Y\mid Z=z)$. 
\end{defn}

The following lemma describes the main way in which conditional independence will be applied in this paper.

\begin{lem}[Deconditioning]
\label{lem:decon}
For any discrete random variables $X,Y$ and $Z$,
\[\mathbb{H}(X\mid Y,Z) \leq \mathbb{H}(X\mid Z).\]
Moreover, if $X$ and $Y$ are conditionally independent given $Z$, then
\[\mathbb{H}(X\mid Y,Z) = \mathbb{H}(X\mid Z). \]
In particular, if $X$ and $Y$ are independent, then
\[\mathbb{H}(X\mid Y) = \mathbb{H}(X). \]
\end{lem}

We are now able to derive a formula for the entropy of a uniformly random homomorphism from a forest to any non-empty graph. Given a graph $T$ and a vertex $v\in V(T)$, let $d_T(v)$ be the \emph{degree} of $v$ in $T$, i.e., the number of edges of $T$ that are incident to $v$. 

\begin{lem}
\label{lem:unifHom}
Let $T$ be a forest, $G$ be a non-empty graph and $\psi$ be a uniformly random homomorphism from $T$ to $G$. For each $v\in V(T)$ let $X_v=\psi(v)$ and for each $e=uv\in E(T)$ let $X_e=(X_u,X_v)$. Then
\[\mathbb{H}(\psi) = \sum_{e\in E(T)}\mathbb{H}(X_e) - \sum_{v\in V(T)}(d_T(v)-1)\cdot \mathbb{H}(X_v).\]
\end{lem}

\begin{proof}
If $T$ is disconnected, then write $T=T_1\sqcup T_2$ for two forests $T_1$ and $T_2$. Let $\psi_i$ be the restriction of $\psi$ to $T_i$ for $i\in\{1,2\}$. Then $\mathbb{H}(\psi)=\mathbb{H}(\psi_1,\psi_2) = \mathbb{H}(\psi_1)+\mathbb{H}(\psi_2\mid \psi_1)$ by Lemma~\ref{lem:CR}. Clearly, $\psi_1$ and $\psi_2$ are independent and so $\mathbb{H}(\psi)=\mathbb{H}(\psi_1)+\mathbb{H}(\psi_2)$ by Lemma~\ref{lem:decon}. Also, $\psi_1$ and $\psi_2$ are uniformly random homomorphisms from $T_1$ and $T_2$, respectively, to $G$. So, by induction, we may assume that $T$ is connected; i.e. $T$ is a tree.

Let $t=v(T)$ and order the vertices of $T$ by $v_1,\dots,v_t$ so that each vertex $v_i$ for $2\leq i\leq t$ has a unique neighbour in $\{v_1,\dots,v_{i-1}\}$; we denote this neighbour by $v_{p(i)}$. Let $\psi$ be a uniformly random homomorphism from $T$ to $G$. For $1\leq i\leq t$, let $X_i=\psi(v_i)$ and, if $i\geq2$, let $X_{e(i)}=(X_{p(i)},X_i)$. By the Chain Rule (Lemma~\ref{lem:CR}), 
\[\mathbb{H}(\psi)=\mathbb{H}(X_1,\dots,X_t) = \mathbb{H}(X_1) + \sum_{i=2}^t\mathbb{H}(X_i\mid X_1,\dots,X_{i-1}).\]
By Lemma~\ref{lem:decon}, this is at most
\[\mathbb{H}(X_1) + \sum_{i=2}^t \mathbb{H}(X_i\mid X_{p(i)}).\]
Applying the Chain Rule once again, we have $\mathbb{H}(X_i\mid X_{p(i)}) = \mathbb{H}(X_{e(i)}) - \mathbb{H}(X_{p(i)})$. Plugging this in to the above expression and rearranging yields
\[\mathbb{H}(\psi)\leq \sum_{e\in E(T)}\mathbb{H}(X_e) - \sum_{v\in V(T)}(d_T(v)-1)\cdot \mathbb{H}(X_v).\]
Since $T$ is a tree and $\psi$ is uniform, the variable $X_i$ is conditionally independent of $(X_j: j<i, j\neq p(e))$ given $X_{p(i)}$ for all $2\leq i\leq t$. Thus, $\mathbb{H}(X_i\mid X_1,\dots,X_{i-1})$ is equal to $\mathbb{H}(X_i\mid X_{p(i)})$ by Lemma~\ref{lem:decon}, which means that all of the inequalities above are actually equalities, and so the proof is complete. 
\end{proof}

Finally, we present the proof of Lemma~\ref{lem:LP}.

\begin{proof}[Proof of Lemma~\ref{lem:LP}]
Let $H$ and $T$ be non-empty forests and suppose that the value of $LP(H,T)$ is equal to $e(T)$. Let $w:\Hom(H,T)\to\mathbb{R}$ be a function which certifies this; as usual, we may assume that $w(\varphi)$ is rational for all $\varphi\in \Hom(H,T)$. Let $q$ be a positive integer multiple of $e(T)$ such that $w(\varphi)\cdot q$ is also an integer multiple of $e(T)$ for all $\varphi\in \Hom(H,T)$. Note that, since $w$ certifies that $LP(H,T)$ has value $e(T)$, we must have
\[e(T) = \sum_{e\in E(T)}\sum_{\varphi:H\to T}\mu_\varphi(e)\cdot w(\varphi) = \sum_{\varphi:H\to T}w(\varphi) \left(\sum_{e\in E(T)}\mu_\varphi(e)\right) = e(H)\sum_{\varphi:H\to T}w(\varphi)\]
and so
\begin{equation}\label{eq:e(T)/e(H)}\sum_{\varphi:H\to T}w(\varphi) = e(T)/e(H).\end{equation}
Also, summing \eqref{eq:Vconstraint} over all vertices of $T$ gives us
\begin{equation}\label{eq:positive}v(T)\geq \sum_{v\in V(T)}\sum_{\varphi:H\to T}\mu_\varphi(v)\cdot w(\varphi) = \sum_{\varphi:H\to T}w(\varphi)\left(\sum_{v\in V(T)}\mu_\varphi(v)\right) = \frac{v(H)\cdot e(T)}{e(H)}\end{equation}
where the last step applies \eqref{eq:e(T)/e(H)}. We define 
\begin{equation}\label{eq:mDef}m=q\cdot\left(\frac{v(T)\cdot e(H)}{e(T)} - v(H)\right).\end{equation}
Note that, since $q$ is a multiple of $e(T)$, we have that $m$ is an integer. Also, $m$ is non-negative by \eqref{eq:positive}. Define
\[H':= (q\cdot H)\sqcup (m\cdot K_1).\]
Our aim is to show that
\begin{equation}\label{eq:LPH'} \hom(H',G)^{e(T)} \geq \hom(T,G)^{q\cdot e(H)}\end{equation}
for every non-empty graph $G$. To see that \eqref{eq:LPH'} is sufficient for proving $H\succcurlyeq T$, observe that
\[t(H,G)^{e(T)}= \left(\frac{\hom(H,G)}{v(G)^{v(H)}}\right)^{e(T)} =\left(\frac{\hom(H,G)^q\cdot v(G)^{m}}{v(G)^{v(H)\cdot q}\cdot v(G)^{m}}\right)^{e(T)/q}  =\left(\frac{\hom(H',G)}{v(G)^{v(H)\cdot q + m}}\right)^{e(T)/q}.\]
If \eqref{eq:LPH'} holds, then this expression can be bounded from below by
\[\left(\frac{\hom(T,G)^{q\cdot e(H)/e(T)}}{v(G)^{v(H)\cdot q + m}}\right)^{e(T)/q} = \frac{\hom(T,G)^{e(H)}}{\left(v(G)^{v(H)\cdot q + m}\right)^{e(T)/q}} = t(T,G)^{e(H)}.\]
By \eqref{eq:mDef}, this is equal to $t(T,G)^{e(H)}$, as desired.

So, we focus on proving \eqref{eq:LPH'}. To prove it, we construct a distribution on $\Hom(H',G)$ such that, if $\psi'$ is chosen randomly according to this distribution, then 
\begin{equation}\label{eq:compareToUnif}\mathbb{H}(\psi')\geq \frac{q\cdot e(H)}{e(T)}\cdot \mathbb{H}(\psi)\end{equation}
where $\psi$ is a uniformly random element of $\Hom(T,G)$. If this holds, then, by two applications of Lemma~\ref{lem:unif}, it will follow that
\[\log_2(\hom(H',G)) \geq \mathbb{H}(\psi')\geq \frac{q\cdot e(H)}{e(T)}\cdot \mathbb{H}(\psi) = \frac{q\cdot e(H)}{e(T)}\cdot\log_2(\hom(T,G))\]
which immediately implies \eqref{eq:LPH'}. Thus, it suffices to construct a random homomorphism $\psi'$ in such a way that \eqref{eq:compareToUnif} holds. 

As in Lemma~\ref{lem:unifHom}, for each $v\in V(T)$, let $X_v$ be the random variable $X_v=\psi(v)$ for a uniformly random homomorphism $\psi$ from $T$ to $G$. Also, for $e=uv\in E(T)$, let $X_e=(X_u,X_v)$. On the way to constructing the distribution on $\psi'$, let us argue that, for any homomorphism $\varphi:H\to T$, there exists a distribution on homomorphisms from $H$ to $G$ such that, if $\psi_\varphi$ is chosen according to this distribution, then
\[\mathbb{H}(\psi_\varphi) = \sum_{e\in E(H)}\mathbb{H}(X_{\varphi(e)} )- \sum_{v\in V(H)}(d_H(v)-1)\mathbb{H}(X_{\varphi(v)}).\]
This distribution is constructed as follows. In this description, we assume that $H$ is connected; if it is not connected, then one can simply repeat this procedure for every component of $H$, in turn. Let $v_1$ be an arbitrary vertex of $H$ and, for $2\leq i\leq v(H)$, let $v_i$ be a vertex of $H$ with a unique neighbour in $\{v_1,\dots,v_{i-1}\}$; this exists because $H$ is a tree. We denote this unique neighbour by $v_{p(i)}$. We start by choosing $Y_1$ according to the distribution of $X_{\varphi(v_1)}$. Now, for $2\leq i\leq v(H)$, we choose $Y_i$  in such a way that $(Y_{p(i)},Y_i)$ has the same distribution as $(X_{\varphi(v_{p(i)})},X_{\varphi(v_i)})$ and $Y_i$ is conditionally independent of $(Y_j: j<i, j\neq p(i))$ given $Y_{p(i)}$.  In practice, this can be achieved as follows. Sample a pair $(Z_{p(i)},Z_i)$ according to the distribution of $(X_{\varphi(v_{p(i)})}, X_{\varphi(v_i)})$, independently of $Y_1,\dots, Y_{i-1}$. As long as $Z_{p(i)} \neq Y_{p(i)}$, resample $(Z_{p(i)},Z_i)$ in the same way. At the first time that $Z_{p(i)} =Y_{p(i)}$, we stop resampling and set $Y_i:=Z_i$. Since $Y_{p(i)}$ is distributed in the same way as $X_{\varphi(v_{p(i)})}$ by induction, and so is $Z_{p(i)}$ (by construction), the probability that $Z_{p(i)} =Y_{p(i)}$ is non-zero, and so this process will terminate after a finite number of steps with probability one. It is clear that the variable $Y_i$ obtained from this procedure will satisfy the desired properties. Now, if we let $\psi_\varphi:V(H)\to V(G)$ be defined by $\psi_\varphi(v_i) = Y_i$ for all $1\leq i\leq v(H)$, then, by Lemmas~\ref{lem:CR} and~\ref{lem:decon},
\[\mathbb{H}(\psi_\varphi) = \mathbb{H}(Y_1) + \sum_{i=2}^{v(H)}\mathbb{H}(Y_i\mid Y_1,\dots,Y_{i-1})=\mathbb{H}(Y_1) + \sum_{i=2}^{v(H)}\mathbb{H}(Y_i\mid Y_{p(i)})\]
\[=\mathbb{H}(Y_1) + \sum_{i=2}^{v(H)}\left(\mathbb{H}(Y_i,Y_{p(i)}) - \mathbb{H}(Y_{p(i)})\right) = \sum_{e\in E(H)}\mathbb{H}(X_{\varphi(e)}) - \sum_{v\in V(H)}(d_H(v)-1)\mathbb{H}(X_{\varphi(v)})\]
as desired. 

Now, we construct a random homomorphism $\psi'$ from $H'$ to $G$ as follows. For each $\varphi\in \Hom(H,T)$, take $q\cdot w(\varphi)\cdot \frac{e(H)}{e(T)}$ (which is an integer by definition of $q$) of the copies of $H$ from the construction of $H'$ and map them into $G$ according to $q\cdot w(\varphi)\cdot \frac{e(H)}{e(T)}$ independent copies of the random homomorphism $\psi_\varphi$ constructed in the previous paragraph. Note that, by \eqref{eq:e(T)/e(H)}, this determines the mapping all of the $q$ copies of $H$ added during the construction of $H'$. Next, for each $v\in V(T)$, we map exactly
\[m_v:=q\cdot \frac{e(H)}{e(T)}\left(1- \sum_{\varphi:H\to T}\mu_{\varphi}(v)\cdot w(\varphi)\right)\]
of the independent vertices added in the construction of $H'$ into $V(G)$ according to the distribution on $X_v$, independently of everything else that has been done so far. Note that $m_v$ is an integer by definition of $q$ and that $m_v\geq0$ by constraint \eqref{eq:Vconstraint} of the linear program. Also, $\sum_{v\in V(T)}m_v=m$ because of \eqref{eq:e(T)/e(H)}. This concludes the definition of $\psi'$. 

To complete the proof, we need to establish \eqref{eq:compareToUnif}. Since the mappings of different copies of $H$ or independent vertices in $H'$ are chosen independently of one another, we have
\[\mathbb{H}(\psi') = \sum_{\varphi:H\to T}q\cdot w(\varphi)\cdot \frac{e(H)}{e(T)}\cdot \mathbb{H}(\psi_\varphi) + \sum_{v\in V(T)}q\cdot \frac{e(H)}{e(T)}\left(1- \sum_{\varphi:H\to T}\mu_{\varphi}(v)\cdot w(\varphi)\right)\mathbb{H}(X_v)\]
\[=q\cdot \frac{e(H)}{e(T)} \left(\sum_{\varphi:H\to T} w(\varphi)\cdot \mathbb{H}(\psi_\varphi) + \sum_{v\in V(T)}\left(1- \sum_{\varphi:H\to T}\mu_{\varphi}(v)\cdot w(\varphi)\right)\mathbb{H}(X_v)\right).\]
So, it suffices to show that the expression within the parentheses is precisely equal to $\mathbb{H}(\psi)$. For future reference, we label the two summations within the parentheses as follows:
\begin{equation}
\label{eq:firstSum}
\sum_{\varphi:H\to T} w(\varphi)\cdot \mathbb{H}(\psi_\varphi),
\end{equation}
and 
\begin{equation}
\label{eq:secondSum}
\sum_{v\in V(T)}\left(1- \sum_{\varphi:H\to T}\mu_{\varphi}(v)\cdot w(\varphi)\right)\mathbb{H}(X_v).
\end{equation}
Using the expression for $\mathbb{H}(\psi_\varphi)$ that we proved earlier, we can expand the expression in \eqref{eq:firstSum} as follows: 
\[\sum_{\varphi:H\to T} w(\varphi)\cdot \mathbb{H}(\psi_\varphi) = \sum_{\varphi:H\to T} w(\varphi)\cdot \left(\sum_{e\in E(H)}\mathbb{H}(X_{\varphi(e)}) - \sum_{v\in V(H)}(d_H(v)-1)\cdot\mathbb{H}(X_{\varphi(v)})\right)\]
\[=\sum_{e\in E(T)}\left(\sum_{\varphi:H\to T}w(\varphi)\cdot \mu_\varphi(e)\right)\cdot \mathbb{H}(X_e)-\sum_{v\in V(H)}\sum_{\varphi:H\to T}w(\varphi)\cdot (d_H(v)-1)\cdot\mathbb{H}(X_{\varphi(v)}).\]
Since the value of $LP(H,T)$ is $e(T)$, we must have that $\sum_{\varphi:H\to T}w(\varphi)\cdot \mu_\varphi(e)=1$ for all $e\in E(T)$. So, the previous expression reduces to
\[\sum_{e\in E(T)} \mathbb{H}(X_e)-\sum_{v\in V(H)}\sum_{\varphi:H\to T}w(\varphi)\cdot (d_H(v)-1)\cdot\mathbb{H}(X_{\varphi(v)}).\]
The second summation in the above expression can be rearranged as follows:
\[\sum_{v\in V(H)}\sum_{\varphi:H\to T}w(\varphi)\cdot (d_H(v)-1)\cdot\mathbb{H}(X_{\varphi(v)})\]
\[= \sum_{v\in V(H)}\sum_{\varphi:H\to T}w(\varphi)\cdot d_H(v)\cdot\mathbb{H}(X_{\varphi(v)}) - \sum_{v\in V(H)}\sum_{\varphi:H\to T}w(\varphi)\cdot \mathbb{H}(X_{\varphi(v)})\]
\[= \sum_{v\in V(T)}\sum_{\substack{e\in E(T)\\ v\in e}}\sum_{\varphi:H\to T}w(\varphi)\cdot \mu_\varphi(e)\cdot \mathbb{H}(X_v) -\sum_{v\in V(T)}\sum_{\varphi:H\to T}w(\varphi)\cdot \mu_\varphi(v)\cdot \mathbb{H}(X_v).\]
Using the fact that the value of $LP(H,T)$ is $e(T)$ once again, we can simplify the first summation in the above expression as follows:
\[\sum_{v\in V(T)}\sum_{\substack{e\in E(T)\\ v\in e}}\sum_{\varphi:H\to T}w(\varphi)\cdot \mu_\varphi(e)\cdot \mathbb{H}(X_v) =  \sum_{v\in V(T)}\sum_{\substack{e\in E(T)\\ v\in e}} \mathbb{H}(X_v)=\sum_{v\in V(T)}d_T(v)\cdot \mathbb{H}(X_v).\]
Therefore, \eqref{eq:firstSum} can be rewritten as 
\[\sum_{e\in E(T)}\mathbb{H}(X_e) -\sum_{v\in V(T)}d_T(v)\cdot \mathbb{H}(X_v)+ \sum_{v\in V(T)}\sum_{\varphi:H\to T}w(\varphi)\cdot \mu_\varphi(v)\cdot \mathbb{H}(X_v).\]
Adding this to \eqref{eq:secondSum} and doing a bit of cancellation precisely yields the expression for $\mathbb{H}(\psi)$ derived in Lemma~\ref{lem:unifHom}. This completes the proof. 
\end{proof}

\section{Stars}
\label{sec:star}

Our goal in this section is to prove a characterization, for each $k\geq3$, of trees $H$ such that $H\succcurlyeq S_k$. We will derive this from a few general necessary and sufficient conditions for $H\succcurlyeq T$ to hold for a general tree $T$. The key necessary condition roughly says that, if $H$ and $T$ are bipartite graphs such that $H\succcurlyeq T$, then $H$ must be at least as ``unbalanced'' as $T$ is. This extends a result of Sidorenko~\cite[Theorem~3.3]{Sidorenko94} which only applied to pairs of trees on the same number of vertices. 

\begin{lem}
\label{lem:sigma}
Let $H$ and $T$ be bipartite graphs. If $H\succcurlyeq T$, then $e(H)/\sigma(H)\geq e(T)/\sigma(T)$.
\end{lem}

\begin{proof}
Let $(A,B)$ be a bipartition of $T$ such that $|A|=\sigma(T)$, define $y(v)=e(H)/\sigma(H)$ for all $v\in A$, $y(v)=0$ for all $v\in B$ and $y(e)=0$ for all $e\in E(T)$. Any homomorphism $\varphi$ from $H$ to $T$ must map at least $\sigma(H)$ vertices to $A$. Therefore, for any $\varphi\in \Hom(H,T)$, we have
\[\sum_{v\in V(T)}\mu_\varphi(v)\cdot y(v) + \sum_{e\in E(H)}\mu_\varphi(e)\cdot y(e)\geq  \sigma(H)(e(H)/\sigma(H))=e(H).\]
Therefore, \eqref{eq:DhomConstraint} is satisfied. So, by Lemma~\ref{lem:DLP}, in order for $H\succcurlyeq T$ to hold, we must have
\[\sum_{v\in V(T)}y(v) + \sum_{e\in E(T)}y(e)\geq e(T)\]
which is equivalent to
\[\sigma(T)\cdot (e(H)/\sigma(H)) \geq e(T).\]
This completes the proof. 
\end{proof}

Next, we provide a sufficient condition for $H\succcurlyeq T$ to hold for two trees $H$ and $T$. The idea is to restrict our attention to primal-feasible weight functions $w:\Hom(H,T)\to \mathbb{R}$ supported only on homomorphisms that map to single edges of $T$. A \emph{fractional orientation} of a graph $T$ is a function $f:V(T)\times V(T)\to [0,\infty)$ such that $f(u,v)+f(v,u)=1$ for any edge $uv\in E(T)$ and $f(u,v)=0$ if $uv\notin E(T)$. The \emph{out-degree} and \emph{in-degree} of a vertex $v\in V(T)$ in a fractional orientation $f$ of $T$ are $d_{f}^+(v):=\sum_{u\in V(T)}f(v,u)$ and $d_{f}^-(v):=\sum_{u\in V(T)}f(u,v)$, respectively. 

\begin{lem}
\label{lem:orient}
If $H$ and $T$ are trees and $f$ is a fractional orientation of $T$ such that, for all $v\in V(T)$,
\begin{equation}\label{eq:orient}\frac{d_{f}^-(v)\cdot (v(H)-\sigma(H)) + d_{f}^+(v)\cdot\sigma(H) }{e(H)} \leq 1,\end{equation}
then $H\succcurlyeq T$. 
\end{lem}

\begin{proof}
Let $H$ be a tree with bipartition $(A,B)$ where $|A|=\sigma(H)$. For every such ordered pair, we let $\varphi_{u,v}$ be the unique homomorphism from $H$ to $T$ such that $\varphi_{u,v}(A)=\{u\}$ and $\varphi_{u,v}(B)=\{v\}$. For each ordered pair $(u,v)$ with $uv\in E(T)$, we set
\[w(\varphi_{u,v}) = \frac{f(u,v)}{e(H)}.\]
Also, set $w(\varphi)=0$ for all other homomorphisms $\varphi\in \Hom(H,T)$. Then \eqref{eq:Econstraint} holds with equality for every $e\in E(T)$ by the definition of fractional orientation. Regarding \eqref{eq:Vconstraint}, we have, for each $v\in V(T)$,
\[\sum_{\varphi:H\to T}\mu_\varphi(v)\cdot w(\varphi) = \sum_{u\in V(T)}\frac{f(u,v)\cdot (v(H)-\sigma(H))}{e(H)} +\sum_{u\in V(T)}\frac{f(v,u)\cdot \sigma(H)}{e(H)}\]
\[= \frac{d_{f}^-(v)\cdot (v(H)-\sigma(H)) + d_{f}^+(v)\cdot\sigma(H) }{e(H)}\]
which is at most 1 by \eqref{eq:orient}. Thus, the value of $LP(H,T)$ is $e(T)$ and so $H\succcurlyeq T$ by Lemma~\ref{lem:LP}. 
\end{proof}

Using the previous two lemmas, we characterize trees $H$ satisfying $H\succcurlyeq S_k$ for $k\geq3$. 

\begin{proof}[Proof of Theorem~\ref{th:star}]
If $H\succcurlyeq S_k$, then, by Lemma~\ref{lem:sigma}, we have $\frac{e(H)}{\sigma(H)}\geq \frac{e(S_k)}{\sigma(S_k)} =k-1$. This proves the ``only if'' direction.

For the other direction, suppose that $e(H)\geq (k-1)\sigma(H)$. Let $x$ be the unique vertex of degree $k-1$ in $S_k$ and let $f$ be the fractional orientation of $S_k$ such that $f(x,v)=1$ for all $v\in V(S_k)\setminus\{x\}$. We verify that \eqref{eq:orient} holds for all vertices of $S_k$. First, for $x$,
\[\frac{d_f^-(x)\cdot (v(H)-\sigma(H)) + d_f^+(x)\cdot \sigma(H)}{e(H)} = \frac{(k-1)\cdot \sigma(H)}{e(H)}\leq \frac{e(H)}{e(H)}=1\]
by the assumption that $e(H)\geq (k-1)\sigma(H)$. Now, if $v$ is a vertex of degree one, then
\[\frac{d_f^-(v)\cdot (v(H)-\sigma(H)) + d_f^+(v)\cdot \sigma(H)}{e(H)} = \frac{v(H)-\sigma(H)}{e(H)}\leq \frac{v(H)-1}{e(H)}= 1\]
since $H$ is a non-empty tree. Thus, $H\succcurlyeq S_k$ by Lemma~\ref{lem:orient} and we are done. 
\end{proof}

\section{The Four-Vertex Path}
\label{sec:P4}

The goal of this section is to prove Theorem~\ref{th:P4}. For the ``only if'' part of the theorem, we use the following general result which recently appeared as~\cite[Proposition~2.6(b)]{Stoner22+}; we give an alternative proof using Lemma~\ref{lem:DLP}.

\begin{lem}[Stoner~{\cite[Proposition~2.6(b)]{Stoner22+}}]
\label{lem:e/v}
If $H$ and $T$ are non-empty graphs without isolated vertices such that $H\succcurlyeq T$, then $e(H)/v(H)\geq e(T)/v(T)$.
\end{lem}

\begin{proof}
Let $y:V(T)\cup E(T)\to \mathbb{R}$ be the function defined by $y(v)=\frac{e(H)}{v(H)}$ for all $v\in V(T)$ and $y(e)=0$ for all $e\in E(T)$. Then the constraints of $DLP(H,T)$ hold trivially. If $H\succcurlyeq T$, then, by Lemma~\ref{lem:DLP}, the value of $DLP(H,T)$ is at least $e(T)$. Therefore, we must have
\[\sum_{v\in V(T)}y(v) + \sum_{e\in E(T)}y(e) = v(T)(e(H)/v(H))\geq e(T).\]
The result follows by rearranging the above inequality.
\end{proof}

For trees, Lemma~\ref{lem:e/v} simply translates to the following. 

\begin{cor}
\label{cor:biggerTreeIsBigger}
If $H$ and $T$ are trees such that $H\succcurlyeq T$, then $v(H)\geq v(T)$. 
\end{cor}

The following lemma comes from applying Lemma~\ref{lem:orient} to the graph $T=P_k$ for $k\geq2$. 

\begin{lem}
\label{lem:pathBalance}
If $k\geq2$ and $H$ is a tree such that $\sigma(H)\leq \frac{v(H)-k+2}{2}$, then $H\succcurlyeq P_k$. 
\end{lem}

\begin{proof}
Suppose that $\sigma(H)\leq \frac{v(H)-k+2}{2}$. We apply Lemma~\ref{lem:orient}. Let $u_1,\dots,u_k$ be the vertices of $P_k$ in the order that they come on the path. We define a fractional orientation $f$ of $P_k$ such that
\[f(u_i,u_{i+1}) = \frac{i-1}{k-2}\]
for $1\leq i\leq k-1$. We verify that \eqref{eq:orient} holds for every vertex of $P_k$. First,
\[\frac{d_f^-(u_1)\cdot (v(H)-\sigma(H)) + d_f^+(u_1)\cdot \sigma(H)}{e(H)} =\frac{v(H)-\sigma(H)}{e(H)}\leq \frac{v(H)-1}{e(H)}=1\]
and the same holds for $u_k$ by symmetry. Now, for $2\leq i\leq k-1$,
\[d_f^-(u_i) = \frac{i-2}{k-2} + \frac{k-2-i+1}{k-2} = \frac{k-3}{k-2}\]
and
\[d_f^+(u_i) = \frac{k-2-i+2}{k-2} + \frac{i-1}{k-2}=\frac{k-1}{k-2}.\]
So,
\[\frac{d_f^-(u_i)\cdot (v(H)-\sigma(H)) + d_f^+(u_i)\cdot \sigma(H)}{e(H)} = \frac{\left(\frac{k-3}{k-2}\right)\cdot(v(H)-\sigma(H)) + \left(\frac{k-1}{k-2}\right)\cdot \sigma(H)}{e(H)}\]
\[ = \frac{(k-3)\cdot v(H) + 2\sigma(H)}{(k-2)\cdot e(H)} \leq \frac{(k-2)\cdot(v(H)-1)}{(k-2)\cdot e(H)} = 1\]
where the final inequality uses the assumption that $\sigma(H)\leq \frac{v(H)-k+2}{2}$. Thus, $H\succcurlyeq P_k$ by Lemma~\ref{lem:orient}. 
\end{proof}

Without further delay, we prove Theorem~\ref{th:P4}. 

\begin{proof}[Proof of Theorem~\ref{th:P4}]
If $H$ is a tree such that $H\succcurlyeq P_4$, then $v(H)\geq4$ by Corollary~\ref{cor:biggerTreeIsBigger}. This proves one direction of the theorem. So, from here forward, we let $H$ be any tree on at least four vertices and show that $H\succcurlyeq P_4$. Suppose, for the sake of contradiction, that this is not the case. Label the vertices of $P_4$ by $u_1,u_2,u_3,u_4$ in the order that they come on the path. Throughout the rest of the proof, let $(A_1,A_2)$ be the bipartition of $H$ and assume, without loss of generality, that $|A_1|\geq |A_2|$. By Lemma~\ref{lem:pathBalance}, we have $\sigma(H)\geq \frac{v(H)-1}{2}$. The next claim sharpens this bound. 

\begin{claim}
\label{claim:veryBalanced}
$\sigma(H)=\frac{v(H)}{2}$.
\end{claim}

\begin{proof}
Observe that $\sigma(H)=\frac{v(H)}{2}$ if and only if $|A_1|=|A_2|$. So, for the sake of contradiction, suppose that $|A_1|>|A_2|$. Let $z_1$ and $z_2$ be two leaves of $H$. Consider the following two homomorphisms. 
\begin{itemize}
    \item $\varphi_1$ satisfies $\varphi_1(A_1)=\{u_1\}$ and $\varphi_1(A_2)=\{u_2\}$,
    \item $\varphi_2$ satisfies $\varphi_2(A_1)=\{u_4\}$ and $\varphi_2(A_2)=\{u_3\}$.
\end{itemize}
Next, we define two additional homomorphisms $\varphi_3$ and $\varphi_4$. Their precise definitions depend on whether or not $z_1$ and $z_2$ lie on the same side of the bipartition. Let $\{i,j\}=\{1,2\}$ such that $z_1\in A_i$. First, define $\varphi_3$ as follows:
\begin{itemize}
    \item if $z_2\in A_j$, then define $\varphi_3$ to be the homomorphism such that $\varphi_3(z_1)=u_1$, $\varphi_3(A_j\setminus\{z_2\})=\{u_2\}$, $\varphi_3(A_i\setminus\{z_1\})=\{u_3\}$ and $\varphi_3(z_2)=u_4$, 
    \item if $z_2\in A_i$, then define $\varphi_3$ to be the homomorphism such that $\varphi_3(z_1)=\varphi_3(z_2)=u_1$, $\varphi_3(A_j)=\{u_2\}$ and $\varphi_3(A_i\setminus\{z_1,z_2\})=\{u_3\}$.
\end{itemize}
Next, define $\varphi_4$ as follows:
\begin{itemize}
    \item if $z_2\in A_j$, then define $\varphi_4$ to be the homomorphism such that $\varphi_4(z_1)=u_4$, $\varphi_4(A_j\setminus\{z_2\})=\{u_3\}$, $\varphi_4(A_i\setminus\{z_1\})=\{u_2\}$ and $\varphi_4(z_2)=u_1$, 
    \item if $z_2\in A_i$, then define $\varphi_4$ to be the homomorphism such that $\varphi_4(z_1)=\varphi_4(z_2)=u_4$, $\varphi_4(A_j)=\{u_3\}$ and $\varphi_4(A_i\setminus\{z_1,z_2\})=\{u_2\}$.
\end{itemize}
Let $w:\Hom(H,P_4)\to\mathbb{R}$ be the weight function such that
\[w(\varphi_1)=w(\varphi_2)=\frac{e(H)-3}{e(H)\cdot(e(H)-2)},\]
\[w(\varphi_3)=w(\varphi_4)=\frac{1}{2(e(H)-2)}\]
and $w(\varphi)=0$ for every other homomorphism from $H$ to $P_4$. Note that, since $e(H)\geq3$, this weighting is well-defined and $w(\varphi)\geq0$ for every homomorphism $\varphi$. Next, we show that \eqref{eq:Econstraint} holds with equality for every edge of $P_4$. First, for the edge $e=u_2u_3$, 
\[\sum_{\varphi:H\to P_4}\mu_\varphi(e)\cdot w(\varphi) = 2(e(H)-2)\cdot \left(\frac{1}{2(e(H)-2)}\right)=1.\]
Now, for $e=u_1u_2$ or $e=u_3u_4$,
\[\sum_{\varphi:H\to P_4}\mu_\varphi(e)\cdot w(\varphi)= e(H)\cdot\left(\frac{e(H)-3}{e(H)\cdot(e(H)-2)}\right) + 2\cdot\left(\frac{1}{2(e(H)-2)}\right)=1.\]
Finally, we verify the constraint \eqref{eq:Vconstraint}. For $v=u_1$ or $v=u_4$, we have
\[\sum_{\varphi:H\to P_4}\mu_\varphi(v)\cdot w(\varphi) =|A_1|\cdot\left(\frac{e(H)-3}{e(H)\cdot(e(H)-2)}\right) + 2\cdot\left(\frac{1}{2(e(H)-2)}\right)\]
\[=\frac{|A_1|\cdot (e(H)-3) + e(H)}{e(H)(e(H)-2)}.\]
Note that $|A_1|\leq v(H)-1=e(H)$ and so the above expression is at most one. Now, for $v=u_2$ or $v=u_3$, 
\[\sum_{\varphi:H\to P_4}\mu_\varphi(v)\cdot w(\varphi)=|A_2|\cdot\left(\frac{e(H)-3}{e(H)\cdot(e(H)-2)}\right) + (|A_1|+|A_2|-2)\cdot\left(\frac{1}{2(e(H)-2)}\right)\]
\[=\frac{2|A_2|\cdot(e(H)-3) + e(H)\cdot(|A_1|+|A_2|-2)}{2e(H)\cdot(e(H)-2)}=\frac{2|A_2|\cdot(e(H)-3) + e(H)\cdot(e(H)-1)}{2e(H)\cdot(e(H)-2)}.\]
Since $|A_1|>|A_2|$, we have that $|A_2|\leq \frac{v(H)-1}{2}=\frac{e(H)}{2}$. Plugging this into the expression above, we get that it is at most $1$. Thus, the value of $LP(H,P_4)$ is at least $e(P_4)$ and we have $H\succcurlyeq P_4$ by Lemma~\ref{lem:LP}. This contradiction completes the proof of the claim.
\end{proof}

\begin{claim}
\label{claim:big}
$v(H)\geq 9$.
\end{claim}

\begin{proof}
We refer the reader to the tables in Section~\ref{sec:smallTrees} which provide justification that $H\succcurlyeq P_4$ for every tree $H$ with $4\leq v(H)\leq 8$.
\end{proof}

\begin{claim}
\label{claim:3Leaves}
$H$ has at most three leaves. 
\end{claim}

\begin{proof}
Let $\varphi_1$ and $\varphi_2$ be homomorphisms as in the proof of Claim~\ref{claim:veryBalanced}. Suppose that $z_1,z_2,z_3$ and $z_4$ are distinct leaves of $H$. Analogous to the proof of Claim~\ref{claim:veryBalanced}, we let $\varphi_3$ be a homomorphism such that each of $z_1,z_2,z_3$ and $z_4$ is mapped to either $u_1$ or $u_4$ and all other vertices are mapped to $u_2$ or $u_3$. Also, let $\varphi_4$ be obtained by ``reversing'' $\varphi_3$; that is, we set $\varphi_4(v)=u_i$ whenever $\varphi_3(v)=u_{5-i}$ for $1\leq i\leq 4$.

Let $w:\Hom(H,P_4)\to\mathbb{R}$ be the weight function such that
\[w(\varphi_1)=w(\varphi_2)=\frac{e(H)-6}{e(H)\cdot(e(H)-4)},\]
\[w(\varphi_3)=w(\varphi_4)=\frac{1}{2(e(H)-4)}\]
and $w(\varphi)=0$ for every other homomorphism from $H$ to $P_4$. Note that, since $e(H)\geq8$ by Claim~\ref{claim:big}, we have that this weighting is well-defined and satisfies $w(\varphi)\geq0$ for every homomorphism $\varphi$. Next, we show that \eqref{eq:Econstraint} holds with equality for every edge of $P_4$. First, for the edge $e=u_2u_3$, we have
\[\sum_{\varphi:H\to P_4}\mu_\varphi(e)\cdot w(\varphi) = 2(e(H)-4)\cdot \left(\frac{1}{2(e(H)-4)}\right)=1.\]
Now, for $e=u_1u_2$ or $e=u_3u_4$, we have
\[\sum_{\varphi:H\to P_4}\mu_\varphi(e)\cdot w(\varphi)= e(H)\cdot\left(\frac{e(H)-6}{e(H)\cdot(e(H)-4)}\right) + 4\cdot\left(\frac{1}{2(e(H)-4)}\right)=1.\]
Finally, we verify the constraint \eqref{eq:Vconstraint}. For $v=u_1$ or $v=u_4$, we have
\[\sum_{\varphi:H\to P_4}\mu_\varphi(v)\cdot w(\varphi) =|A_1|\cdot\left(\frac{e(H)-6}{e(H)\cdot(e(H)-4)}\right) + 4\cdot\left(\frac{1}{2(e(H)-4)}\right)\]
\[=\frac{|A_1|\cdot (e(H)-6) + 2e(H)}{e(H)(e(H)-4)}.\]
We have $|A_1|\leq v(H)-1=e(H)$ and so the above expression is at most one. Now, for $v=u_2$ or $v=u_3$, we have
\[\sum_{\varphi:H\to P_4}\mu_\varphi(v)\cdot w(\varphi)=|A_2|\cdot\left(\frac{e(H)-6}{e(H)\cdot(e(H)-4)}\right) + (|A_1|+|A_2|-4)\cdot\left(\frac{1}{2(e(H)-4)}\right)\]
\[=\frac{2|A_2|\cdot(e(H)-6) + e(H)\cdot(|A_1|+|A_2|-4)}{2e(H)\cdot(e(H)-4)}=\frac{2|A_2|\cdot(e(H)-6) + e(H)\cdot(e(H)-3)}{2e(H)\cdot(e(H)-4)}.\]
By Claim~\ref{claim:veryBalanced}, we have $|A_2|=\frac{v(H)}{2}=\frac{e(H)+1}{2}$. Plugging this into the expression above and simplifying, we get
\[\frac{(e(H)+1)\cdot(e(H)-6) + e(H)\cdot(e(H)-3)}{2e(H)\cdot(e(H)-4)} = \frac{2e(H)^2-8e(H)-6}{2e(H)^2-8e(H)}<1.\]
Thus, the value of $LP(H,P_4)$ is at least $e(P_4)$ and we have $H\succcurlyeq P_4$ by Lemma~\ref{lem:LP}, which is a contradiction.
\end{proof}

If $H$ has only two leaves, then $H$ is a path and the result follows from Theorems~\ref{th:Godsil} and~\ref{th:Saglam} and transitivity of $\succcurlyeq$. So, by Claim~\ref{claim:3Leaves}, we may assume that $H$ has exactly three leaves. This implies that $H$ has a unique vertex of degree exactly three, say $z$, and all other vertices of $H$ have degree one or two. The graph $H\setminus\{z\}$ consists of three paths, say $w_1\cdots w_\ell$, $x_1\cdots x_k$ and $y_1\cdots y_m$. Since $z$ has degree three, we have $m,k,\ell\geq1$. Since $\sigma(H)=\frac{v(H)}{2}$, exactly one of $\ell$, $k$ or $m$ is odd; without loss of generality, $\ell$ is odd. Note that $v(H)=m+k+\ell+1$ and $e(H)=m+k+\ell$. By Claim~\ref{claim:big}, we have $m+k+\ell\geq8$.

Let $\varphi_1:H\to P_4$ be the unique homomorphism such that $\varphi_1(z)=u_2$, $\varphi_1(w_i)\in\{u_1,u_2\}$ for all $1\leq i\leq\ell$, $\varphi_1(x_i)\in\{u_3,u_4\}$ for all $1\leq i\leq k$ and $\varphi_1(y_i)\in\{u_3,u_4\}$ for all $1\leq i\leq m$. Let $\varphi_2:H\to P_4$ be defined so that, if $v\in V(H)$ and $\varphi_1(v)=u_i$, then $\varphi_2(v)=u_{5-i}$. Let $\varphi_3:H\to P_4$ be a homomorphism such that every vertex is mapped to $u_2$ or $u_3$. Define $w:\Hom(H,P_4)\to\mathbb{R}$ by 
\[w(\varphi_1)=w(\varphi_2)=\frac{1}{m+k+\ell-2},\]
\[w(\varphi_3)=\frac{m+k+\ell-6}{(m+k+\ell-2)(m+k+\ell)}\]
and $w(\varphi)=0$ for all other homomorphisms $\varphi$. Note that, since $m+k+\ell\geq 8>6$, we have $w(\varphi)\geq0$ for all homomorphisms $\varphi$. We verify that \eqref{eq:Econstraint} holds with equality for all edges $e$ of $P_4$. First, if $e=u_1u_2$, then
\[\sum_{\varphi:H\to P_4}\mu_\varphi(e)\cdot w(\varphi) = (m+k+\ell-2)\left(\frac{1}{m+k+\ell-2}\right)=1.\]
The calculation for $e=u_3u_4$ is the same. Now, for $e=u_2u_3$, 
\[\sum_{\varphi:H\to P_4}\mu_\varphi(e)\cdot w(\varphi) = 4\left(\frac{1}{m+k+\ell-2}\right) + (m+k+\ell)\left(\frac{m+k+\ell-6}{(m+k+\ell-2)(m+k+\ell)}\right)=1.\]
Finally, we show that \eqref{eq:Vconstraint} holds for all $v\in V(P_4)$. For $v=u_1$, we have
\[\sum_{\varphi:H\to P_4}\mu_\varphi(v)\cdot w(\varphi)=\left(\frac{\ell+1}{2}\right)\cdot\left(\frac{1}{m+k+\ell-2}\right)+\left(\frac{m+k}{2}\right)\cdot \left(\frac{1}{m+k+\ell-2}\right)\]
\[=\frac{m+k+\ell+1}{2(m+k+\ell-2)}\]
which is less than 1 because $m+k+\ell>5$. Now, for $v=u_2$, we have
\[\sum_{\varphi:H\to P_4}\mu_\varphi(v)\cdot w(\varphi)=\frac{m+k+\ell+1}{2(m+k+\ell-2)} + \left(\frac{m+k+\ell+1}{2}\right)\cdot\frac{m+k+\ell-6}{(m+k+\ell-2)(m+k+\ell)}\]
\[=\frac{(m+k+\ell)(m+k+\ell+1)+(m+k+\ell+1)(m+k+\ell-6)}{2(m+k+\ell-2)(m+k+\ell)}\]
\[=1-\frac{3}{(m+k+\ell-2)(m+k+\ell)}<1.\]
Thus, by Lemma~\ref{lem:LP}, we have $H\succcurlyeq P_4$ and the proof is complete.
\end{proof}

\section{Additional Necessary Conditions}
\label{sec:nec}

In this section, we recall several necessary conditions for two graphs $H$ and $T$ to satisfy $H\succcurlyeq T$ from the literature and add a few new conditions to the list. These results will be used to certify that $H\not\succcurlyeq T$ for many pairs of small trees in the next section.

\subsection{Radius}

The following lemma of~\cite{Leontovich89} is useful for proving that trees on the same number of vertices are incomparable. Recall that the \emph{eccentricity} of a vertex $v$ of a graph $T$ is the maximum over all $u\in V(T)$ of the distance in $T$ from $v$ to $u$. The \emph{radius} of $T$, denoted $\rad(T)$, is the minimum eccentricity of any vertex of $T$; a vertex of minimum eccentricity is called a \emph{centre} of $T$.

\begin{lem}[Leontovich~{\cite[p.~100]{Leontovich89}}]\label{lem:radius1}
Let $H$ and $T$ be trees such that $v(H)=v(T)$. If $H\succcurlyeq T$, then $\rad(T)\geq \rad(H)$. 
\end{lem}

We are able to generalize this result to the case of trees where $v(H) \ne v(T)$. 

\begin{lem}\label{lem:radius}
    Let $H$ and $T$ be trees. If $H\succcurlyeq T$, then $\frac{e(H)}{\rad(H)} \ge \frac{e(T)}{\rad(T)}$. 
\end{lem}

\begin{proof}
We assume that $\frac{e(H)}{\rad(H)} < \frac{e(T)}{\rad(T)}$ and prove that $H\not\succcurlyeq T$. Let $R:=\rad(H)$ and let $v_0$ be a centre of $H$. Let $v_0v_1\cdots v_R$ be a path in $H$ which certifies that the eccentricity of $v_0$ is $R$. Since $v_0$ is a centre of $H$, there must exist a vertex at distance at least $R$ from $v_1$. Moreover, the distance from $v_0$ to such a vertex must be equal to $R-1$, since the eccentricity of $v_0$ is $R$ and $H$ is bipartite. Thus, the unique path from $v_1$ to this vertex passes through $v_0$. So, we can let $v_0u_1\cdots u_{R-1}$ be a path in $H$, where $u_{R-1}$ is a vertex at distance $R$ from $v_1$. Since $H$ is a tree, we have that the set $\{u_1,\dots,u_{R-1}\}$ is disjoint from $\{v_1,\dots,v_R\}$.

Let $r:=\rad(T)$. For an integer $d\geq2$, we define $G_{d}$ to be a full $d$-ary tree of height $r$. Specifically, $G_d$ is defined as follows. First, let $L_0$ be a set consisting of one vertex. Then, for $1\leq i\leq r$, let $L_i$ be a set of $d^i$ vertices such that each vertex in $L_i$ has exactly one neighbour in $L_{i-1}$, called its \emph{parent}, and every vertex in $L_{i-1}$ has exactly $d$ neighbours in $L_i$, called its \emph{children}. Note that every vertex, except for the vertex of $L_0$, has a unique parent and all vertices have exactly $d$ children, except for the vertices of $L_{r}$ which have no children.

Now, we bound the number of homomorphisms from $T$ into the graph $G_{d}$ from below. Let $v$ be a centre of $T$ and note that all vertices of $T$ are at distance at most $r$ from $v$. So, $\hom(T,G_{d})$ is at least the number of homomorphisms from $T$ to $G_{d}$ that map each vertex at distance $i$ from $v$ to the set $L_i$ for $0\leq i\leq r$. Thus,
\[\hom(T,G_{d}) \geq d^{e(T)},\]
and so 
\[t(T,G_{d})^{e(H)} \geq \frac{d^{e(T)e(H)}}{v(G)^{v(T)e(H)}} = \frac{d^{e(T)e(H)}}{v(G_d)^{e(T)e(H)+e(H)}}\]

Finally, we bound $t(H,G_d)$ from above. Recall that $v_0$ is the centre of $H$ and that $v_0v_1\cdots v_R$ and $v_0u_1\cdots u_{R-1}$ are paths in $H$ such that $\{u_1,\dots,u_{R-1}\}$ is disjoint from $\{v_1,\dots,v_R\}$. Given a homomorphism $\varphi:H\to G_d$, let $L_{i^*}$ be the set containing $\varphi(v_0)$, let $S_1$ be the set of indices $1\leq j\leq R$ such that $\varphi(v_j)$ is the parent of $\varphi(v_{j-1})$ and let $S_2$ be the set of indices $1\leq j\leq R-1$ such that $\varphi(u_j)$ is the parent of $\varphi(u_{j-1})$, where we regard $u_0$ as being $v_0$. The number of homomorphisms from $H$ to $G_d$ corresponding to a given choice of $(i^*,S_1,S_2)$ is at most
\[|L_{i^*}|\cdot d^{R-|S_1|} \cdot d^{R-1-|S_2|}\cdot (d+1)^{e(H)-R-(R-1)}= O\left(d^{i^*+e(H) - |S_1|-|S_2|}\right).\]
where the asymptotics are as $d$ tends to infinity. By the construction of $G_d$, for any given choice of $i^*$, we have that
\[(R-|S_1|)-|S_1| \leq r-i^*  \qquad \text{and} \qquad (R-1-|S_2|)-|S_2| \leq r-i^*  \]
which implies
\[|S_1| \geq \left\lceil\frac{R-r+i^*}{2}\right\rceil  \qquad \text{and} \qquad |S_2| \geq \left\lceil\frac{R-r+i^*-1}{2}\right\rceil .\]
Together, this gives
\[|S_1| + |S_2| \ge R - r + i^*.\]
Also, the number of choices of $(i^*,S_1,S_2)$ is a constant depending on $H$ and $T$ and, in particular, is independent of $d$. So, we have
\[\hom(H,G_d) = O\left(d^{e(H)+r-R}\right)\]
and so
\[t(H,G_d)^{e(T)} =O\left(
\frac{d^{e(H)e(T)+re(T)-Re(T)}}{v(G)^{(e(H)+1)e(T)}}
\right).\]
Comparing this to the lower bound on $t(T,G_d)^{e(H)}$ proven earlier we see that we will be done if we can show that for sufficiently large $d$,
\[\frac{d^{e(H)e(T)+re(T)-Re(T)}}{v(G)^{e(T)e(H)+e(T)}} \ll \frac{d^{e(T)e(H)}}{v(G)^{e(T)e(H)+e(H)}}\]
or equivalently that
\[v(G)^{e(H)-e(T)} \ll d^{Re(T)-re(T)}.\] 
Note that  $v(G) = 1 + d + \ldots d^r < 2d^r$ for $d>1$ and so $v(G)^{e(H)-e(T)} = \Theta\left(d^{r(e(H) - e(T))}\right).$ By assumption, $r(e(H) - e(T)) < Re(T)-re(T)$ and so for sufficiently large $d$ we have $t(H,G_d)^{e(T)} < t(T,G_{d})^{e(H)}$. In particular, $H\not\succcurlyeq T$. 
\end{proof}

\subsection{Degree Sequence}

Given a graph $T$, the \emph{degree} sequence of $T$ is $D_T :=(d_T(v_1),\dots,d_T(v_k))$ where $v_1,\dots,v_k$ are the vertices of $T$ labelled in such a way that $d_T(v_1)\geq\cdots\geq d_T(v_k)$. Given two sequences $(a_1,\dots,a_k)$ and $(b_1,\dots,b_k)$ of non-negative real numbers, say that $(a_1,\dots,a_k)$ \emph{majorizes} $(b_1,\dots,b_k)$ if
\[a_1+\cdots+a_t\geq b_1+\cdots+b_t\]
for all $1\leq t\leq k$. The following necessary condition was proven by Leontovich~\cite{Leontovich89} by applying a well-known inequality of Muirhead~\cite{Muirhead02}. While the original proof was written in Russian, one can find an English summary of the proof in~\cite[p.~274]{Sidorenko94}.

\begin{lem}[Leontovich~\cite{Leontovich89}]
Let $H$ and $T$ be graphs such that $v(H)=v(T)$. If $H\succcurlyeq T$, then $D_H$ majorizes $D_T$.
\end{lem}

The following refinement of the above lemma for graphs with the same degree sequence was proved by Sidorenko~\cite{Sidorenko94}. Given a graph $T$ and an edge $e=uv$ of $T$, define the \emph{degree} of $e$, denoted $d_T(e)$, to be the set $\{d_T(u),d_T(v)\}$.

\begin{lem}[Sidorenko~{\cite[Theorem~3.2]{Sidorenko94}}]
Let $H$ and $T$ be graphs such that $D_H=D_T$. If $H\succcurlyeq T$, then there exists a bijection $\beta:E(H)\to E(T)$ such that $d_H(e)=d_T(\beta(e))$ for all $e\in E(H)$.
\end{lem}

\subsection{Independence Number}

Next, we give an alternative proof of~\cite[Proposition~2.6(d)]{Stoner22+} using Lemma~\ref{lem:DLP}; recall that $\alpha(H)$ is the size of the largest independent set in a graph $H$.

\begin{lem}[Stoner~{\cite[Proposition~2.6(d)]{Stoner22+}}]
\label{lem:alpha}
If $H$ and $T$ are graphs such that $H\succcurlyeq T$, then
\[\frac{e(H)}{v(H)-\alpha(H)}\geq \frac{e(T)}{v(T)-\alpha(T)}.\]
\end{lem}

\begin{proof}
Let $S$ be the largest independent set in $T$. Set $y(v)=0$ for all $v\in S$, $y(v)=\frac{e(H)}{v(H)-\alpha(H)}$ for all $v\in V(T)\setminus S$ and $y(e)=0$ for all $e\in E(T)$. Now, for any homomorphism $\varphi$ from $H$ to $T$, at least $v(H)-\alpha(H)$ vertices of $H$ are mapped to $V(T)\setminus S$. So, 
\[\sum_{v\in V(T)}\mu_\varphi(v)\cdot y(v) + \sum_{e\in E(T)}\mu_\varphi(e)\cdot y(e)\geq  (v(H)-\alpha(H))\left(\frac{e(H)}{v(H)-\alpha(H)}\right)=e(H).\]
Thus, the constraints of $DLP(H,T)$ are satisfied. Since $H\succcurlyeq T$, Lemma~\ref{lem:DLP} says that the value of $DLP(H,T)$ must be at least $e(T)$. So,
\[(v(T)-\alpha(T))\left(\frac{e(H)}{v(H)-\alpha(H)}\right)\geq e(T)\]
which completes the proof.
\end{proof}

\subsection{Exploiting Leaves}

The next lemma roughly says that, if $H$ and $T$ are trees with the same number of vertices satisfying $H\succcurlyeq T$ and $T$ has a vertex that is adjacent to a large number of leaves and a small number of non-leaves, then so does $H$. After proving this, we will build upon the ideas of the proof to establish Theorem~\ref{th:pathsNotUnique}. 

\begin{defn}
Given a non-empty tree $H$ and a vertex $v\in V(H)$, let $\ell_H(v)$ be the number of leaves of $H$ that are adjacent to $v$.
\end{defn}

\begin{defn}
Given a non-empty tree $H$ and a vertex $v\in V(H)$, define $\lambda_H(v):=\frac{\ell_H(v)}{d_H(v)-\ell_H(v)}$ if $d_H(v)>\ell_H(v)$ and $\lambda_H(v)=\infty$ otherwise. 
\end{defn}

\begin{defn}
Given a non-empty tree $H$, let $\lambda(H):=\max_{v\in V(H)}\lambda_H(v)$. 
\end{defn}

\begin{obs}
Every non-empty tree $H$ satisfies $\lambda(H)\geq1$.
\end{obs}

\begin{proof}
Take a longest path in $H$. The second to last vertex on this path, say $v$, has at least one leaf neighbour and at most one non-leaf neighbour. Thus, $\lambda(H)\geq \lambda_H(v)\geq 1$. 
\end{proof}

\begin{lem}
\label{lem:lambda}
Let $H$ and $T$ be trees with $v(H)=v(T)\geq2$. If $H\succcurlyeq T$, then $\lambda(H)\geq \lambda(T)$. 
\end{lem}

\begin{proof}
Suppose that $\lambda(T)>\lambda(H)$ and let $x\in V(T)$ such that $\lambda_T(x)=\lambda(T)$. If $\lambda(T)=\infty$, then $T$ is a star and the fact that $H\not\succcurlyeq T$ for every tree $H$ on $v(T)$ vertices, apart from $T$ itself, follows from Theorem~\ref{th:star}. So, we assume that $\lambda(T)\in \mathbb{R}$. 

Let $\ell:=\ell_T(x)$ and $d:=d_T(x)-\ell_T(x)$ and note that $\lambda(T)=\ell/d$ by definition of $x$. Let $u_1,\dots,u_\ell$ be the leaf neighbours of $x$ and $w_1,\dots,w_d$ be the non-leaf neighbours of $x$. Choose $\delta>0$ sufficiently small so that
\begin{equation}\label{eq:delta}(v(T)+1)\cdot \left(\frac{e(T)}{v(T)}-\delta\right)>e(T).\end{equation}
Next, choose $\varepsilon>0$ small enough so that
\begin{equation}\label{eq:epsilon}\varepsilon < (\lambda(T)-\lambda(H))\cdot \delta\end{equation}
We define a function $y:V(T)\cup E(T)\to\mathbb{R}$ as follows:
\[y(u_1)=\cdots=y(u_\ell)=0,\]
\[y(v)=\frac{e(T)}{v(T)}\text{ for all }v\in V(T)\setminus\{u_1,\dots,u_\ell\},\]
\[y(xu_1)=\cdots=y(xu_\ell)=\frac{e(T)}{v(T)}-\delta,\]
\[y(xw_1)=\cdots=y(xw_d)=\lambda(T)\cdot\delta-\varepsilon,\]
\[y(e)=0\text{ for all }e\in E(T)\text{ such that }x\notin e.\]
We have
\[\sum_{v\in V(T)}y(v) + \sum_{e\in E(T)}y(e) = \frac{e(T)}{v(T)}\cdot(v(T)-\ell) + \ell\cdot\left(\frac{e(T)}{v(T)}-\delta\right) + d\cdot\left(\lambda(T)\cdot \delta-\varepsilon\right)\]
\[=e(T)-d\cdot \varepsilon < e(T).\]
Thus, by Lemma~\ref{lem:DLP}, in order to show that $H\not\succcurlyeq T$, it suffices to show that \eqref{eq:DhomConstraint} holds for every homomorphism $\varphi:H\to T$. 

To analyze this, we define a notion of ``cost.'' Given an ordered pair $(u,v)\in V(T)\times V(T)$ with $uv\in E(T)$, define the \emph{cost} of $(u,v)$ to be $c(u,v) := y(uv) + y(v)$. Given an arbitrary root vertex $r$ in $V(H)$, let $D_r$ be the directed graph obtained from $H$ by orienting all edges away from $r$. Then, given a homomorphism $\varphi:H\to T$, we have
\[\sum_{v\in V(T)}\mu_\varphi(v)\cdot y(v) + \sum_{e\in E(T)}\mu_\varphi(e)\cdot y(e) = y(\varphi(r)) + \sum_{(a,b)\in A(D_r)}c(\varphi(a),\varphi(b)).\]
So, to verify \eqref{eq:DhomConstraint}, it is useful to analyze the cost function. We observe the following:
\[c(x,u_i) = \frac{e(T)}{v(T)}-\delta\text{ for }1\leq i\leq \ell,\]
\[c(u_i,x) = \frac{2e(T)}{v(T)}-\delta\text{ for }1\leq i\leq \ell,\]
\[c(x,w_i)=c(w_i,x)=\frac{e(T)}{v(T)} + \lambda(T)\cdot \delta - \varepsilon\text{ for }1\leq i\leq d,\]
\[c(u,v)=\frac{e(T)}{v(T)}\text{ for any other }uv\in E(T).\]
In particular, $c(u,v)\geq \frac{e(T)}{v(T)}-\delta$ for every adjacent pair $(u,v)$, and the only such pairs satisfying $c(u,v)<\frac{e(T)}{v(T)}$ are those of the form $(x,u_i)$ for $1\leq i\leq \ell$. Moreover, $y(v)\geq \frac{e(T)}{v(T)}$ for every vertex $v$ apart from $u_1,\dots,u_\ell$. Thus, any homomorphism $\varphi:H\to T$ that does not map any vertex to $u_1,\dots,u_\ell$ satisfies \eqref{eq:DhomConstraint} automatically. 

Now, suppose that $\varphi$ maps a non-leaf vertex $r\in V(H)$ to $u_i$ for some $i$. Recall that $u_i$ is a leaf, and so every neighbour of $r$ must be mapped to $x$. So, 
\[y(\varphi(r)) + \sum_{(a,b)\in A(D_r)}c(\varphi(a),\varphi(b))\geq d_H(r)\cdot \left(\frac{2e(T)}{v(T)}-\delta\right) + \left(v(H)-d_H(r)-1\right)\left(\frac{e(T)}{v(T)}-\delta\right)\]
\[ = d_H(r)\cdot\left(\frac{e(T)}{v(T)}\right) + \left(v(H)-1\right)\cdot \left(\frac{e(T)}{v(T)}-\delta\right)> (v(H)+1)\cdot \left(\frac{e(T)}{v(T)}-\delta\right)\]
which is greater than $e(T)$ by \eqref{eq:delta} and the fact that $v(H)=v(T)$. 

Thus, we can restrict our attention to only those homomorphisms $\varphi$ that map a non-empty set of leaves, and no other vertices, to $\{u_1,\dots,u_\ell\}$. Let $r_1,\dots,r_q$ be the vertices of $H$ that are mapped to $x$ by $\varphi$. Note that all of the vertices $r_1,\dots,r_q$ are on the same side of the bipartition of $H$ and, in particular, no two of them are adjacent. The cost of each arc of $D_{r_1}$ from $r_i$ to a leaf, for any $1\leq i\leq q$, is at least
\[\frac{e(T)}{v(T)}-\delta.\]
Since no non-leaf is mapped to $\{u_1,\dots,u_\ell\}$, for $1\leq i\leq q$, the cost of any arc of $D_{r_1}$ entering $r_i$ or leaving $r_i$ and entering a non-leaf vertex is precisely
\[\frac{e(T)}{v(T)} + \lambda(T)\cdot \delta - \varepsilon\]
and, since none of the vertices $r_1,\dots,r_q$ are adjacent, these arcs are all distinct. The cost of any other arc of $D_{r_1}$ is at least $\frac{e(T)}{v(T)}$. So, we get
\[ y(\varphi(r_1)) + \sum_{(a,b)\in A(D_{r_1})}c(\varphi(a),\varphi(b))\]
\[=e(T) + \sum_{i=1}^q \left((d_H(r_i)-\ell_H(r_i))\cdot\left(\lambda(T)\cdot \delta-\varepsilon\right) - \ell_H(r_i)\cdot \delta \right).\]
For each $1\leq i\leq q$, we have
\[(d_H(r_i)-\ell_H(r_i))\cdot\left(\lambda(T)\cdot \delta-\varepsilon\right) - \ell_H(r_i)\cdot \delta\]
\[\geq(d_H(r_i)-\ell_H(r_i))\cdot\left(\lambda(T)\cdot \delta-\varepsilon - \lambda(H)\cdot \delta\right)\]
which is non-negative by \eqref{eq:epsilon}. So \eqref{eq:DhomConstraint} holds which, by Lemma~\ref{lem:DLP}, implies that $H\not\succcurlyeq T$ and the proof is complete.
\end{proof}

Next, we build upon the ideas used in the proof of the previous lemma to prove Theorem~\ref{th:pathsNotUnique}. However, since $\lambda(P_k)=1$, we will need a slightly different argument which relies on the specific structure of near-stars. 

\begin{proof}[Proof of Theorem~\ref{th:pathsNotUnique}]
Let $H$ be a $k$-vertex near-star with $\ell$ leaves where $\frac{k+1}{2}\leq \ell\leq k-3$. Note that these inequalities together imply that $k\geq7$. Thus, $\ell\geq 4$ and there is a unique vertex of $H$, say $x$, of degree $\ell$. 

The tree $H$ was obtained from the star $S_{\ell+1}$ by subdividing $p:=k-\ell-1$ distinct edges. Note that the number of neighbours of $x$ of degree two is precisely $p$. By the assumed bounds on $\ell$, we have 
\[2\leq p\leq k- \frac{k+1}{2} - 1=\frac{k-3}{2}\leq \ell-2.\]
Thus, if we let $q$ be the number of leaves of $H$ adjacent to $x$, we have that $p\geq2$ and $q\geq 2$. 

Pick $0<\gamma<1$ small enough so that
\begin{equation}
\label{eq:gamma}
(k+1)\cdot\left(\frac{k-1}{k}-\gamma\right)>k-1.
\end{equation}
Next, we define
\begin{equation}
\label{eq:deltaAgain}
\delta:=\frac{\gamma\cdot(p-1)}{q+p-1}.
\end{equation}
Finally, pick $\varepsilon>0$ so that
\begin{equation}
\label{eq:epsilonAgain}
\frac{p\cdot\delta}{p+q-1}< \varepsilon<\delta.
\end{equation}
Note that such an $\varepsilon$ exists because $\delta<1$ and $q\geq2$. Let $a,b,c,d$ be the first four vertices of the path $P_k$, which all exist because $k\geq7$. We define a function $y:V(P_k)\cup E(P_k)\to\mathbb{R}$ as follows. Set
\[y(a)=0,\]
\[y(b)=\frac{k-1}{k} - \gamma+\varepsilon,\]
\[y(c)=\frac{k-1}{k} - \varepsilon,\]
\[y(v)=\frac{k-1}{k} \text{ for all other vertices }v,\]
\[y(ab)=\frac{k-1}{k}-\delta,\]
\[y(bc)=\gamma,\]
\[y(cd)=\varepsilon,\]
\[y(e)=0\text{ for all other edges }e.\]
We have
\[\sum_{v\in V(P_k)}y(v)+\sum_{e\in E(P_k)}y(e)\]
\[=0\left(\frac{k-1}{k}-\gamma+\varepsilon\right) + \left(\frac{k-1}{k}-\varepsilon\right) + (k-3)\left(\frac{k-1}{k}\right) + \left(\frac{k-1}{k}-\delta\right)+\gamma+\varepsilon\]
\[=k-1+\varepsilon-\delta\]
which is less than $k-1=e(P_k)$ by \eqref{eq:epsilonAgain}. 

So, by Lemma~\ref{lem:DLP}, to show that $H\not\succcurlyeq P_k$, it suffices to show that \eqref{eq:DhomConstraint} is satisfied. It is useful to define a cost function associated to $y$ as in the proof of Lemma~\ref{lem:lambda}. We have
\[c(a,b) = \frac{2(k-1)}{k}-\delta-\gamma+\varepsilon\]
\[c(b,a)= \frac{k-1}{k}-\delta,\]
\[c(b,c)=\frac{k-1}{k}+\gamma-\varepsilon,\]
\[c(c,b)=\frac{k-1}{k}+\varepsilon,\]
\[c(c,d)=\frac{k-1}{k}+\varepsilon,\]
and
\[c(u,v)=\frac{k-1}{k}\text{ for all other }uv\in E(P_k).\]
In particular, the only pair with a cost less than $\frac{k-1}{k}$ is $(b,a)$, which has a cost of $\frac{k-1}{k}-\delta$.

If $\varphi:H\to P_k$ maps a non-leaf vertex $r$ of $H$ to $a$, then we get
\[y(\varphi(r))+\sum_{(u,v)\in A(D_r)}c(\varphi(u),\varphi(v))\geq d_H(r)\cdot \left(\frac{2(k-1)}{k}-\delta-\gamma+\varepsilon\right) + (k-d_H(r)-1)\cdot\left(\frac{k-1}{k}-\delta\right)\]
\[=d_H(r)\cdot\left(\frac{k-1}{k}-\gamma+\varepsilon\right) + (k-1)\cdot\left(\frac{k-1}{k}-\delta\right)\geq (k+1)\left(\frac{k-1}{k}-\gamma\right)\]
which is greater than $k-1$ by  \eqref{eq:gamma}. So, we may focus our attention on homomorphisms which map a (possibly empty) set of leaves to $a$. 

Suppose that $x$ is not mapped to $b$ nor $c$. By the result of the previous paragraph, it is not mapped to $a$. So, $y(\varphi(x))=\frac{k-1}{k}$. Also, since every vertex of $H$ is at distance at most two from $x$, there is no vertex of $H$ mapped to $a$. Thus, every arc of $D_x$ contributes a cost of at least $\frac{k-1}{k}$ and so \eqref{eq:DhomConstraint} holds.

Suppose next that $x$ is mapped to $b$. Each leaf that is adjacent to $x$ contributes a cost of at least $\frac{k-1}{k}-\delta$. Since no non-leaf vertices are mapped to $a$, every non-leaf neighbour of $x$ is mapped to $c$, and so they contribute a cost of at least $\frac{k-1}{k}+\gamma-\varepsilon$ each. The vertices at distance two from $x$ are mapped to either $b$ or $d$; either way, they contribute a cost of $\frac{k-1}{k}+\varepsilon$ each. Thus, 
\[y(\varphi(x))+\sum_{(u,v)\in A(D_x)}c(\varphi(u),\varphi(v))\]
\[\geq\frac{k-1}{k}-\gamma+\varepsilon+ q\cdot \left(\frac{k-1}{k}-\delta\right) + p\cdot\left(\frac{k-1}{k}+\gamma-\varepsilon\right) + p\cdot\left(\frac{k-1}{k}+\varepsilon\right)\]
\[\geq k-1+(p-1)(\gamma-\varepsilon) - q\delta.\]
By the upper bound on $\varepsilon$ in \eqref{eq:epsilon}, this is at least
\[\geq k-1+(p-1)(\gamma-\delta) - q\delta = k-1+(p-1)\gamma -(q+p-1)\delta = k-1.\]

Finally, suppose that $x$ is mapped to $c$. The neighbours of $x$ contribute a cost of exactly $\frac{k-1}{k}+\varepsilon$ each. The $p$ leaves of $H$ that are not adjacent to $x$ contribute a cost of at least $\frac{k-1}{k}-\delta$ each. Thus, 
\[y(\varphi(x))+\sum_{(u,v)\in A(D_x)}c(\varphi(u),\varphi(v))\]
\[\geq \frac{k-1}{k}-\varepsilon + (p+q)\cdot \left(\frac{k-1}{k}+\varepsilon\right) + p\cdot \left(\frac{k-1}{k}-\delta\right)\]
\[=k-1 + (p+q-1)\varepsilon - \delta\cdot p>k-1\]
by \eqref{eq:epsilon}. Thus, \eqref{eq:DhomConstraint} holds and we have $H\not\succcurlyeq P_k$ by Lemma~\ref{lem:DLP}.
\end{proof}

Note that the previous proof did not require the full structure of the path $P_k$. The same argument yields the following more general result.

\begin{thm}
\label{th:nearStar}
Suppose that $H$ is a $k$-vertex near-star with $\ell$ leaves for $\frac{k+1}{2}\leq \ell\leq k-3$. If $T$ is a tree with $k$ vertices containing a path $abc$ such that $a$ is a leaf and $b$ and $c$ have degree two, then $H\not\succcurlyeq T$.
\end{thm}

\section{Trees on Few Vertices}
\label{sec:smallTrees}

In this section, we describe the relation $\succcurlyeq$ on all trees on at most $8$ vertices in a series of tables. In each table, the cell on the row corresponding to $T_i$ and column corresponding to $T_j$ is grey if and only if $T_j\succcurlyeq T_i$. The labelling of the trees is given in Appendix~\ref{app:smallTrees}. By Corollary~\ref{cor:biggerTreeIsBigger}, we need only consider pairs $i$ and $j$ such that $v(T_j)\geq v(T_i)$; for this reason, all other cells are left blank. Trees are listed in order of increasing number of vertices and thick lines in the tables separate trees on different numbers of vertices. 

Each cell contains a reference to the paper in which it first appeared (to the best of our knowledge), to a theorem or lemma from this paper, or to an ad hoc construction provided in Appendices~\ref{app:Primalcertificates} or~\ref{app:Dualcertificates}. All references to Erd\H{o}s and Simonovits~\cite{ErdosSimonovits82} refer to the result of Godsil which appeared there. All references to~\cite{Stoner22+} refer to Lemma~\ref{lem:alpha}. It is trivial that $H\succcurlyeq H$ for any graph $H$, and we do not reference anything in the diagonal entries of the tables. Likewise, proving $S_k\succcurlyeq S_\ell$ when $k\geq\ell$ is an easy application of H\"older's Inequality, and we are not sure who first observed it, so we will simply write ``$\st$'' for ``star'' inside of all such cells and not include a reference for it. When $T_j\succcurlyeq T_i$ can be deduced from transitivity involving another tree  $T_k$, then we write ``$\tr_k$'' inside the cell on the $T_i$ row and $T_j$ column, unless the result $T_j\succcurlyeq T_i$ was proven earlier chronologically than at least one of $T_j\succcurlyeq T_k$ or $T_k\succcurlyeq T_i$. Cases in which $T_j\not\succcurlyeq T_i$ can be deduced from transitivity are treated similarly. Likewise, if $T_i\succcurlyeq T_j$ for $T_i\neq T_j$, then $T_j\not\succcurlyeq T_i$ automatically follows from antisymmetry, in which case we write ``$\as$''' inside the cell on the $T_i$ row and $T_j$ column. 

Figure \ref{fig:poset} shows a Hasse diagram for the poset of trees on at most $7$ vertices, where a vertex labelled $i$ corresponds to tree $T_i$. 

\begin{figure}
    \centering
    \includegraphics[scale=.5]{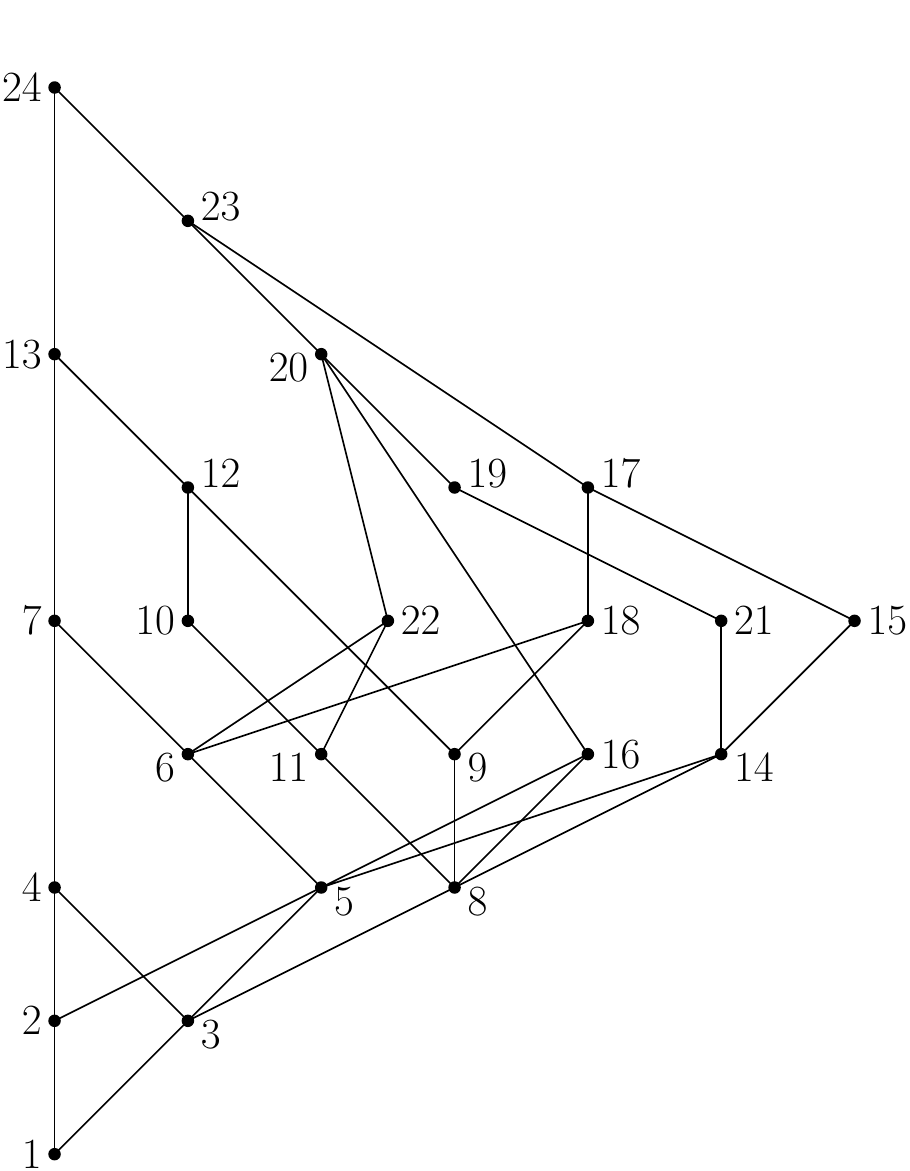}
    \caption{A Hasse diagram for the poset of trees on at most $7$ vertices.}
    \label{fig:poset}
\end{figure}

\begin{table}[htbp]
\centering
\begin{tblr}{
colspec = {|c|c|c|c|c|c|c|c|c|c|c|c|c|c|},
rowspec = {|c|c|c|c|c|c|c|c|c|c|c|c|c|c|},
cell{2}{2} = {lightgray},
cell{2}{3} = {lightgray},
cell{2}{4} = {lightgray},
cell{2}{5} = {lightgray},
cell{2}{6} = {lightgray},
cell{2}{7} = {lightgray},
cell{2}{8} = {lightgray},
cell{2}{9} = {lightgray},
cell{2}{10} = {lightgray},
cell{2}{11} = {lightgray},
cell{2}{12} = {lightgray},
cell{2}{13} = {lightgray},
cell{2}{14} = {lightgray},
cell{3}{3} = {lightgray},
cell{3}{5} = {lightgray},
cell{3}{6} = {lightgray},
cell{3}{7} = {lightgray},
cell{3}{8} = {lightgray},
cell{3}{10} = {lightgray},
cell{3}{13} = {lightgray},
cell{3}{14} = {lightgray},
cell{4}{4} = {lightgray},
cell{4}{5} = {lightgray},
cell{4}{6} = {lightgray},
cell{4}{7} = {lightgray},
cell{4}{8} = {lightgray},
cell{4}{9} = {lightgray},
cell{4}{10} = {lightgray},
cell{4}{11} = {lightgray},
cell{4}{12} = {lightgray},
cell{4}{13} = {lightgray},
cell{4}{14} = {lightgray},
cell{5}{5} = {lightgray},
cell{5}{8} = {lightgray},
cell{5}{14} = {lightgray},
cell{6}{6} = {lightgray},
cell{6}{7} = {lightgray},
cell{6}{8} = {lightgray},
cell{6}{10} = {lightgray},
cell{6}{13} = {lightgray},
cell{6}{14} = {lightgray},
cell{7}{7} = {lightgray},
cell{7}{8} = {lightgray},
cell{7}{13} = {lightgray},
cell{7}{14} = {lightgray},
cell{8}{8} = {lightgray},
cell{8}{14} = {lightgray},
cell{9}{9} = {lightgray},
cell{9}{10} = {lightgray},
cell{9}{11} = {lightgray},
cell{9}{12} = {lightgray},
cell{9}{13} = {lightgray},
cell{9}{14} = {lightgray},
cell{10}{10} = {lightgray},
cell{10}{13} = {lightgray},
cell{10}{14} = {lightgray},
cell{11}{11} = {lightgray},
cell{11}{13} = {lightgray},
cell{11}{14} = {lightgray},
cell{12}{11} = {lightgray},
cell{12}{12} = {lightgray},
cell{12}{13} = {lightgray},
cell{12}{14} = {lightgray},
cell{13}{13} = {lightgray},
cell{13}{14} = {lightgray},
cell{14}{14} = {lightgray},
vline{2,3,4,6,9}={1.5pt},
hline{2,3,4,6,9}={1.5pt},
}
 & $T_{1}$ & $T_{2}$ & $T_{3}$ & $T_{4}$ & $T_{5}$ & $T_{6}$ & $T_{7}$ & $T_{8}$ & $T_{9}$ & $T_{10}$ & $T_{11}$ & $T_{12}$ & $T_{13}$\\
$T_{1}$ &  & $\st$ & \cite{MulhollandSmith59} & $\st$ & \cite{MulhollandSmith59} & \cite{Sidorenko91} & $\st$ & \cite{MulhollandSmith59} & \cite{Sidorenko91} & \cite{Sidorenko91} & \cite{Sidorenko91} & \cite{Sidorenko91} & $\st$\\
$T_{2}$ &  &  & \cite{Stoner22+} & $\st$ & \cite{ErdosSimonovits82} & \ref{th:star} & $\st$ & \cite{Stoner22+} & \ref{th:star} & \ref{th:star} & \cite{Stoner22+} & \ref{th:star} & $\st$\\
$T_{3}$ &  &  &  & \cite{Hoffman67} & \cite{ErdosSimonovits82} & $\tr_{5}$ & \ref{lem:pathBalance} & \cite{Saglam18} & \ref{lem:pathBalance} & $\tr_{8}$ & $\tr_{8}$ & \ref{lem:pathBalance} & \ref{lem:pathBalance}\\
$T_{4}$ &  &  & $\as$ &  & \cite{Stoner22+} & \cite{Stoner22+} & $\st$ & \cite{Stoner22+} & \cite{Stoner22+} & \cite{Stoner22+} & \cite{Stoner22+} & \cite{Stoner22+} & $\st$\\
$T_{5}$ &  &  &  &  &  & \cite{Sidorenko85} & \cite{Hoffman67} & \cite{Stoner22+} & \ref{primal:5-9} & \ref{lem:sigma} & \cite{Stoner22+} & $\tr_{9}$ & \ref{lem:pathBalance}\\
$T_{6}$ &  &  &  &  & $\as$ &  & \cite{Sidorenko85} & \cite{Stoner22+} & \ref{dual:6-9} & \ref{lem:sigma} & \cite{Stoner22+} & \ref{primal:6-12} & $\tr_{7}$\\
$T_{7}$ &  &  &  &  & $\as$ & $\as$ &  & \cite{Stoner22+} & \cite{Stoner22+} & \cite{Stoner22+} & \cite{Stoner22+} & \cite{Stoner22+} & $\st$\\
$T_{8}$ &  &  &  &  &  &  &  &  & \cite{Sidorenko85} & \cite{Leontovich89} & \cite{Leontovich89} & \cite{Sidorenko85} & \cite{Hoffman67}\\
$T_{9}$ &  &  &  &  &  &  &  & $\as$ &  & \cite{Leontovich89} & \cite{Leontovich89} & \cite{Sidorenko85} & \cite{Sidorenko85}\\
$T_{10}$ &  &  &  &  &  &  &  & $\as$ & \cite{Leontovich89} &  & $\as$ & \cite{Leontovich89} & \cite{Leontovich89}\\
$T_{11}$ &  &  &  &  &  &  &  & $\as$ & \cite{Leontovich89} & \cite{Leontovich89} &  & \cite{Leontovich89} & \cite{Leontovich89}\\
$T_{12}$ &  &  &  &  &  &  &  & $\as$ & $\as$ & $\as$ & $\as$ &  & \cite{Sidorenko85}\\
$T_{13}$ &  &  &  &  &  &  &  & $\as$ & $\as$ & $\as$ & $\as$ & $\as$ & \\
\end{tblr}
\caption{The relation $\succcurlyeq$ restricted to trees on at most 6 vertices.}
\end{table}

\begin{table}[htbp]
\centering
\begin{tblr}{
colspec = {|c|c|c|c|c|c|c|c|c|c|c|c|},
rowspec = {|c|c|c|c|c|c|c|c|c|c|c|c|c|c|},
cell{2}{2} = {lightgray},
cell{2}{3} = {lightgray},
cell{2}{4} = {lightgray},
cell{2}{5} = {lightgray},
cell{2}{6} = {lightgray},
cell{2}{7} = {lightgray},
cell{2}{8} = {lightgray},
cell{2}{9} = {lightgray},
cell{2}{10} = {lightgray},
cell{2}{11} = {lightgray},
cell{2}{12} = {lightgray},
cell{3}{2} = {lightgray},
cell{3}{3} = {lightgray},
cell{3}{4} = {lightgray},
cell{3}{5} = {lightgray},
cell{3}{6} = {lightgray},
cell{3}{7} = {lightgray},
cell{3}{8} = {lightgray},
cell{3}{9} = {lightgray},
cell{3}{10} = {lightgray},
cell{3}{11} = {lightgray},
cell{3}{12} = {lightgray},
cell{4}{2} = {lightgray},
cell{4}{3} = {lightgray},
cell{4}{4} = {lightgray},
cell{4}{5} = {lightgray},
cell{4}{6} = {lightgray},
cell{4}{7} = {lightgray},
cell{4}{8} = {lightgray},
cell{4}{9} = {lightgray},
cell{4}{10} = {lightgray},
cell{4}{11} = {lightgray},
cell{4}{12} = {lightgray},
cell{5}{5} = {lightgray},
cell{5}{6} = {lightgray},
cell{5}{11} = {lightgray},
cell{5}{12} = {lightgray},
cell{6}{2} = {lightgray},
cell{6}{3} = {lightgray},
cell{6}{4} = {lightgray},
cell{6}{5} = {lightgray},
cell{6}{6} = {lightgray},
cell{6}{7} = {lightgray},
cell{6}{8} = {lightgray},
cell{6}{9} = {lightgray},
cell{6}{10} = {lightgray},
cell{6}{11} = {lightgray},
cell{6}{12} = {lightgray},
cell{7}{5} = {lightgray},
cell{7}{6} = {lightgray},
cell{7}{8} = {lightgray},
cell{7}{10} = {lightgray},
cell{7}{11} = {lightgray},
cell{7}{12} = {lightgray},
cell{8}{12} = {lightgray},
cell{9}{2} = {lightgray},
cell{9}{3} = {lightgray},
cell{9}{4} = {lightgray},
cell{9}{5} = {lightgray},
cell{9}{6} = {lightgray},
cell{9}{7} = {lightgray},
cell{9}{8} = {lightgray},
cell{9}{9} = {lightgray},
cell{9}{10} = {lightgray},
cell{9}{11} = {lightgray},
cell{9}{12} = {lightgray},
cell{10}{5} = {lightgray},
cell{10}{6} = {lightgray},
cell{10}{11} = {lightgray},
cell{10}{12} = {lightgray},
cell{11}{8} = {lightgray},
cell{11}{11} = {lightgray},
cell{11}{12} = {lightgray},
cell{12}{8} = {lightgray},
cell{12}{10} = {lightgray},
cell{12}{11} = {lightgray},
cell{12}{12} = {lightgray},
cell{13}{11} = {lightgray},
cell{13}{12} = {lightgray},
cell{14}{12} = {lightgray},
vline{2}={1.5pt},
hline{2,3,4,6,9}={1.5pt},
}
 & $T_{14}$ & $T_{15}$ & $T_{16}$ & $T_{17}$ & $T_{18}$ & $T_{19}$ & $T_{20}$ & $T_{21}$ & $T_{22}$ & $T_{23}$ & $T_{24}$\\
$T_{1}$ & \cite{MulhollandSmith59} & \cite{Sidorenko91} & \cite{Sidorenko91} & \cite{Sidorenko91} & \cite{Sidorenko91} & \cite{Sidorenko91} & \cite{Sidorenko91} & \cite{Sidorenko91} & \cite{Sidorenko91} & \cite{Sidorenko91} & $\st$\\
$T_{2}$ & \cite{ErdosSimonovits82} & \ref{th:star} & \ref{th:star} & \ref{th:star} & \ref{th:star} & \ref{th:star} & \ref{th:star} & \ref{th:star} & \ref{th:star} & \ref{th:star} & $\st$\\
$T_{3}$ & \cite{ErdosSimonovits82} & $\tr_{14}$ & $\tr_{5}$ & \ref{lem:pathBalance} & \ref{lem:pathBalance} & $\tr_{14}$ & $\tr_{14}$ & $\tr_{14}$ & $\tr_{5}$ & \ref{lem:pathBalance} & \ref{lem:pathBalance}\\
$T_{4}$ & \cite{Stoner22+} & \cite{Stoner22+} & \cite{Stoner22+} & \ref{th:star} & \ref{th:star} & \cite{Stoner22+} & \ref{th:star} & \cite{Stoner22+} & \cite{Stoner22+} & \ref{th:star} & $\st$\\
$T_{5}$ & \cite{ErdosSimonovits82} & $\tr_{14}$ & \ref{primal:5-16} & \ref{lem:pathBalance} & \ref{lem:pathBalance} & $\tr_{14}$ & $\tr_{14}$ & $\tr_{14}$ & $\tr_{6}$ & \ref{lem:pathBalance} & \ref{lem:pathBalance}\\
$T_{6}$ & $\tr_{15}$ & \ref{dual:6-15} & \ref{dual:6-16} & $\tr_{18}$ & \ref{primal:6-18} & \ref{dual:6-19} & $\tr_{22}$ & $\tr_{19}$ & \ref{primal:6-22} & $\tr_{17}$ & $\tr_{7}$\\
$T_{7}$ & \cite{Stoner22+} & \cite{Stoner22+} & \cite{Stoner22+} & \cite{Stoner22+} & \cite{Stoner22+} & \cite{Stoner22+} & \cite{Stoner22+} & \cite{Stoner22+} & \cite{Stoner22+} & \cite{Stoner22+} & $\st$\\
$T_{8}$ & \cite{ErdosSimonovits82} & $\tr_{14}$ & \ref{primal:8-16} & $\tr_{14}$ & $\tr_{14}$ & $\tr_{14}$ & $\tr_{14}$ & $\tr_{14}$ & $\tr_{11}$ & $\tr_{14}$ & \ref{lem:pathBalance}\\
$T_{9}$ & \cite{Stoner22+} & \cite{Stoner22+} & \cite{Stoner22+} & $\tr_{18}$ & \ref{primal:9-18} & \cite{Stoner22+} & \ref{lem:sigma} & \cite{Stoner22+} & \cite{Stoner22+} & $\tr_{12}$ & $\tr_{13}$\\
$T_{10}$ & \cite{Stoner22+} & \cite{Stoner22+} & \cite{Stoner22+} & $\tr_{11}$ & $\tr_{17}$ & \cite{Stoner22+} & \ref{primal:10-20} & \cite{Stoner22+} & \cite{Stoner22+} & $\tr_{12}$ & $\tr_{13}$\\
$T_{11}$ & \ref{lem:radius} & \ref{lem:radius} & \ref{lem:radius} & \ref{dual:11-17} & $\tr_{17}$ & \ref{dual:11-19} & $\tr_{22}$ & $\tr_{19}$ & \ref{primal:11-22} & $\tr_{10}$ & $\tr_{10}$\\
$T_{12}$ & \cite{Stoner22+} & \cite{Stoner22+} & \cite{Stoner22+} & $\tr_{10}$ & $\tr_{10}$ & \cite{Stoner22+} & \ref{lem:sigma} & \cite{Stoner22+} & \cite{Stoner22+} & \ref{primal:12-23} & $\tr_{13}$\\
$T_{13}$ & \cite{Stoner22+} & \cite{Stoner22+} & \cite{Stoner22+} & \cite{Stoner22+} & \cite{Stoner22+} & \cite{Stoner22+} & \cite{Stoner22+} & \cite{Stoner22+} & \cite{Stoner22+} & \cite{Stoner22+} & $\st$\\
\end{tblr}
\caption{The relation $H\succcurlyeq T$ restricted to trees $H$ and $T$ with $v(H)=7$ and $v(T)\leq 6$.}
\end{table}

\begin{table}[htbp]
\centering
\begin{tblr}{
colspec = {|c|c|c|c|c|c|c|c|c|c|c|c|c|},
rowspec = {|c|c|c|c|c|c|c|c|c|c|c|c|c|c|},
cell{2}{2} = {lightgray},
cell{2}{3} = {lightgray},
cell{2}{4} = {lightgray},
cell{2}{5} = {lightgray},
cell{2}{6} = {lightgray},
cell{2}{7} = {lightgray},
cell{2}{8} = {lightgray},
cell{2}{9} = {lightgray},
cell{2}{10} = {lightgray},
cell{2}{11} = {lightgray},
cell{2}{12} = {lightgray},
cell{2}{13} = {lightgray},
cell{3}{3} = {lightgray},
cell{3}{6} = {lightgray},
cell{3}{8} = {lightgray},
cell{3}{9} = {lightgray},
cell{3}{12} = {lightgray},
cell{3}{13} = {lightgray},
cell{4}{2} = {lightgray},
cell{4}{3} = {lightgray},
cell{4}{4} = {lightgray},
cell{4}{5} = {lightgray},
cell{4}{6} = {lightgray},
cell{4}{7} = {lightgray},
cell{4}{8} = {lightgray},
cell{4}{9} = {lightgray},
cell{4}{10} = {lightgray},
cell{4}{11} = {lightgray},
cell{4}{12} = {lightgray},
cell{4}{13} = {lightgray},
cell{5}{13} = {lightgray},
cell{6}{3} = {lightgray},
cell{6}{6} = {lightgray},
cell{6}{8} = {lightgray},
cell{6}{9} = {lightgray},
cell{6}{12} = {lightgray},
cell{6}{13} = {lightgray},
cell{7}{6} = {lightgray},
cell{7}{9} = {lightgray},
cell{7}{12} = {lightgray},
cell{7}{13} = {lightgray},
cell{9}{2} = {lightgray},
cell{9}{3} = {lightgray},
cell{9}{4} = {lightgray},
cell{9}{5} = {lightgray},
cell{9}{6} = {lightgray},
cell{9}{7} = {lightgray},
cell{9}{8} = {lightgray},
cell{9}{9} = {lightgray},
cell{9}{10} = {lightgray},
cell{9}{11} = {lightgray},
cell{9}{12} = {lightgray},
cell{9}{13} = {lightgray},
cell{10}{13} = {lightgray},
cell{11}{13} = {lightgray},
cell{12}{13} = {lightgray},
vline{2}={1.5pt},
hline{2,3,4,6,9}={1.5pt},
}
 & $T_{25}$ & $T_{26}$ & $T_{27}$ & $T_{28}$ & $T_{29}$ & $T_{30}$ & $T_{31}$ & $T_{32}$ & $T_{33}$ & $T_{34}$ & $T_{35}$ & $T_{36}$\\
$T_{1}$ & \cite{MulhollandSmith59} & \cite{Sidorenko91} & \cite{Sidorenko91} & \cite{Sidorenko91} & \cite{Sidorenko91} & \cite{Sidorenko91} & \cite{Sidorenko91} & \cite{Sidorenko91} & \cite{Sidorenko91} & \cite{Sidorenko91} & \cite{Sidorenko91} & \cite{Sidorenko91}\\
$T_{2}$ & \cite{Stoner22+} & \ref{th:star} & \ref{th:star} & \cite{Stoner22+} & \ref{th:star} & \cite{Stoner22+} & \ref{th:star} & \ref{th:star} & \ref{th:star} & \cite{Stoner22+} & \ref{th:star} & \ref{th:star}\\
$T_{3}$ & $\tr_{8}$ & \ref{lem:pathBalance} & $\tr_{8}$ & $\tr_{8}$ & \ref{lem:pathBalance} & $\tr_{8}$ & \ref{lem:pathBalance} & \ref{lem:pathBalance} & $\tr_{8}$ & $\tr_{8}$ & \ref{lem:pathBalance} & \ref{lem:pathBalance}\\
$T_{4}$ & \cite{Stoner22+} & \cite{Stoner22+} & \cite{Stoner22+} & \cite{Stoner22+} & \cite{Stoner22+} & \cite{Stoner22+} & \cite{Stoner22+} & \cite{Stoner22+} & \cite{Stoner22+} & \cite{Stoner22+} & \cite{Stoner22+} & \ref{th:star}\\
$T_{5}$ & \cite{Stoner22+} & $\tr_{14}$ & \ref{lem:sigma} & \cite{Stoner22+} & $\tr_{6}$ & \cite{Stoner22+} & \ref{primal:5-31} & $\tr_{26}$ & \ref{lem:sigma} & \cite{Stoner22+} & $\tr_{6}$ & \ref{lem:pathBalance}\\
$T_{6}$ & \cite{Stoner22+} & \ref{dual:6-26} & \ref{lem:sigma} & \cite{Stoner22+} & \ref{primal:6-29} & \cite{Stoner22+} & $\tr_{41}$ & \ref{primal:6-32} & \ref{lem:sigma} & \cite{Stoner22+} & \ref{primal:6-35} & $\tr_{32}$\\
$T_{7}$ & \cite{Stoner22+} & \cite{Stoner22+} & \cite{Stoner22+} & \cite{Stoner22+} & \cite{Stoner22+} & \cite{Stoner22+} & \cite{Stoner22+} & \cite{Stoner22+} & \cite{Stoner22+} & \cite{Stoner22+} & \cite{Stoner22+} & \cite{Stoner22+}\\
$T_{8}$ & \cite{Saglam18} & $\tr_{25}$ & $\tr_{25}$ & \ref{primal:8-28} & $\tr_{14}$ & \ref{primal:8-30} & $\tr_{25}$ & $\tr_{25}$ & $\tr_{25}$ & $\tr_{25}$ & $\tr_{16}$ & \ref{lem:pathBalance}\\
$T_{9}$ & \cite{Stoner22+} & \cite{Stoner22+} & \cite{Stoner22+} & \cite{Stoner22+} & \cite{Stoner22+} & \cite{Stoner22+} & \cite{Stoner22+} & \cite{Stoner22+} & \cite{Stoner22+} & \cite{Stoner22+} & \cite{Stoner22+} & $\tr_{17}$\\
$T_{10}$ & \cite{Stoner22+} & \cite{Stoner22+} & \cite{Stoner22+} & \cite{Stoner22+} & \cite{Stoner22+} & \cite{Stoner22+} & \cite{Stoner22+} & \cite{Stoner22+} & \cite{Stoner22+} & \cite{Stoner22+} & \cite{Stoner22+} & $\tr_{37}$\\
$T_{11}$ & \ref{lem:radius} & \ref{lem:radius} & \ref{lem:radius} & \ref{lem:radius} & \ref{lem:radius} & \ref{lem:radius} & \ref{lem:radius} & \ref{lem:radius} & \ref{lem:radius} & \ref{lem:radius} & \ref{lem:radius} & $\tr_{10}$\\
$T_{12}$ & \cite{Stoner22+} & \cite{Stoner22+} & \cite{Stoner22+} & \cite{Stoner22+} & \cite{Stoner22+} & \cite{Stoner22+} & \cite{Stoner22+} & \cite{Stoner22+} & \cite{Stoner22+} & \cite{Stoner22+} & \cite{Stoner22+} & \ref{dual:12-36}\\
$T_{13}$ & \cite{Stoner22+} & \cite{Stoner22+} & \cite{Stoner22+} & \cite{Stoner22+} & \cite{Stoner22+} & \cite{Stoner22+} & \cite{Stoner22+} & \cite{Stoner22+} & \cite{Stoner22+} & \cite{Stoner22+} & \cite{Stoner22+} & \cite{Stoner22+}\\
\end{tblr}
\caption{The relation $H\succcurlyeq T$ restricted to trees $H$ and $T$ with $H\in\{T_{25},\dots,T_{36}\}$ and $v(T)\leq 6$.}
\end{table}

\begin{table}[htbp]
\centering
\begin{tblr}{
colspec = {|c|c|c|c|c|c|c|c|c|c|c|c|},
rowspec = {|c|c|c|c|c|c|c|c|c|c|c|c|c|c|},
cell{2}{2} = {lightgray},
cell{2}{3} = {lightgray},
cell{2}{4} = {lightgray},
cell{2}{5} = {lightgray},
cell{2}{6} = {lightgray},
cell{2}{7} = {lightgray},
cell{2}{8} = {lightgray},
cell{2}{9} = {lightgray},
cell{2}{10} = {lightgray},
cell{2}{11} = {lightgray},
cell{2}{12} = {lightgray},
cell{3}{2} = {lightgray},
cell{3}{3} = {lightgray},
cell{3}{5} = {lightgray},
cell{3}{6} = {lightgray},
cell{3}{8} = {lightgray},
cell{3}{10} = {lightgray},
cell{3}{11} = {lightgray},
cell{3}{12} = {lightgray},
cell{4}{2} = {lightgray},
cell{4}{3} = {lightgray},
cell{4}{4} = {lightgray},
cell{4}{5} = {lightgray},
cell{4}{6} = {lightgray},
cell{4}{7} = {lightgray},
cell{4}{8} = {lightgray},
cell{4}{9} = {lightgray},
cell{4}{10} = {lightgray},
cell{4}{11} = {lightgray},
cell{4}{12} = {lightgray},
cell{5}{2} = {lightgray},
cell{5}{11} = {lightgray},
cell{5}{12} = {lightgray},
cell{6}{2} = {lightgray},
cell{6}{3} = {lightgray},
cell{6}{5} = {lightgray},
cell{6}{6} = {lightgray},
cell{6}{8} = {lightgray},
cell{6}{10} = {lightgray},
cell{6}{11} = {lightgray},
cell{6}{12} = {lightgray},
cell{7}{2} = {lightgray},
cell{7}{3} = {lightgray},
cell{7}{5} = {lightgray},
cell{7}{8} = {lightgray},
cell{7}{10} = {lightgray},
cell{7}{11} = {lightgray},
cell{7}{12} = {lightgray},
cell{8}{12} = {lightgray},
cell{9}{2} = {lightgray},
cell{9}{3} = {lightgray},
cell{9}{4} = {lightgray},
cell{9}{5} = {lightgray},
cell{9}{6} = {lightgray},
cell{9}{7} = {lightgray},
cell{9}{8} = {lightgray},
cell{9}{9} = {lightgray},
cell{9}{10} = {lightgray},
cell{9}{11} = {lightgray},
cell{9}{12} = {lightgray},
cell{10}{2} = {lightgray},
cell{10}{11} = {lightgray},
cell{10}{12} = {lightgray},
cell{11}{2} = {lightgray},
cell{11}{4} = {lightgray},
cell{11}{8} = {lightgray},
cell{11}{11} = {lightgray},
cell{11}{12} = {lightgray},
cell{12}{2} = {lightgray},
cell{12}{3} = {lightgray},
cell{12}{4} = {lightgray},
cell{12}{5} = {lightgray},
cell{12}{7} = {lightgray},
cell{12}{8} = {lightgray},
cell{12}{9} = {lightgray},
cell{12}{10} = {lightgray},
cell{12}{11} = {lightgray},
cell{12}{12} = {lightgray},
cell{13}{11} = {lightgray},
cell{13}{12} = {lightgray},
cell{14}{12} = {lightgray},
vline{2}={1.5pt},
hline{2,3,4,6,9}={1.5pt},
}
 & $T_{37}$ & $T_{38}$ & $T_{39}$ & $T_{40}$ & $T_{41}$ & $T_{42}$ & $T_{43}$ & $T_{44}$ & $T_{45}$ & $T_{46}$ & $T_{47}$\\
$T_{1}$ & \cite{Sidorenko91} & \cite{Sidorenko91} & \cite{Sidorenko91} & \cite{Sidorenko91} & \cite{Sidorenko91} & \cite{Sidorenko91} & \cite{Sidorenko91} & \cite{Sidorenko91} & \cite{Sidorenko91} & \cite{Sidorenko91} & $\st$\\
$T_{2}$ & \ref{th:star} & \ref{th:star} & \ref{th:star} & \ref{th:star} & \ref{th:star} & \ref{th:star} & \ref{th:star} & \cite{Stoner22+} & \ref{th:star} & \ref{th:star} & $\st$\\
$T_{3}$ & \ref{lem:pathBalance} & \ref{lem:pathBalance} & $\tr_{25}$ & \ref{lem:pathBalance} & \ref{lem:pathBalance} & $\tr_{25}$ & \ref{lem:pathBalance} & $\tr_{25}$ & \ref{lem:pathBalance} & \ref{lem:pathBalance} & \ref{lem:pathBalance}\\
$T_{4}$ & \ref{th:star} & \cite{Stoner22+} & \ref{th:star} & \cite{Stoner22+} & \cite{Stoner22+} & \cite{Stoner22+} & \ref{th:star} & \cite{Stoner22+} & \cite{Stoner22+} & \ref{th:star} & $\st$\\
$T_{5}$ & \ref{lem:pathBalance} & $\tr_{6}$ & \ref{lem:sigma} & $\tr_{6}$ & $\tr_{14}$ & \ref{lem:sigma} & $\tr_{6}$ & \cite{Stoner22+} & $\tr_{6}$ & \ref{lem:pathBalance} & \ref{lem:pathBalance}\\
$T_{6}$ & $\tr_{18}$ & $\tr_{40}$ & \ref{lem:sigma} & \ref{primal:6-40} & \ref{dual:6-41} & \ref{lem:sigma} & $\tr_{20}$ & \cite{Stoner22+} & $\tr_{22}$ & $\tr_{29}$ & $\tr_{7}$\\
$T_{7}$ & \cite{Stoner22+} & \cite{Stoner22+} & \cite{Stoner22+} & \cite{Stoner22+} & \cite{Stoner22+} & \cite{Stoner22+} & \cite{Stoner22+} & \cite{Stoner22+} & \cite{Stoner22+} & \cite{Stoner22+} & $\st$\\
$T_{8}$ & \ref{lem:pathBalance} & $\tr_{25}$ & $\tr_{25}$ & $\tr_{25}$ & $\tr_{25}$ & $\tr_{25}$ & $\tr_{39}$ & $\tr_{25}$ & $\tr_{11}$ & \ref{lem:pathBalance} & \ref{lem:pathBalance}\\
$T_{9}$ & $\tr_{18}$ & \cite{Stoner22+} & \ref{lem:sigma} & \cite{Stoner22+} & \cite{Stoner22+} & \cite{Stoner22+} & \ref{lem:sigma} & \cite{Stoner22+} & \cite{Stoner22+} & $\tr_{36}$ & $\tr_{13}$\\
$T_{10}$ & \ref{primal:10-37} & \cite{Stoner22+} & \ref{primal:10-39} & \cite{Stoner22+} & \cite{Stoner22+} & \cite{Stoner22+} & $\tr_{20}$ & \cite{Stoner22+} & \cite{Stoner22+} & $\tr_{43}$ & $\tr_{13}$\\
$T_{11}$ & $\tr_{10}$ & $\tr_{40}$ & $\tr_{10}$ & \ref{primal:11-40} & \ref{dual:11-41} & $\tr_{44}$ & $\tr_{38}$ & \ref{primal:11-44} & $\tr_{22}$ & $\tr_{36}$ & $\tr_{10}$\\
$T_{12}$ & $\tr_{36}$ & \cite{Stoner22+} & \ref{lem:sigma} & \cite{Stoner22+} & \cite{Stoner22+} & \cite{Stoner22+} & \ref{lem:sigma} & \cite{Stoner22+} & \cite{Stoner22+} & $\tr_{23}$ & $\tr_{13}$\\
$T_{13}$ & \cite{Stoner22+} & \cite{Stoner22+} & \cite{Stoner22+} & \cite{Stoner22+} & \cite{Stoner22+} & \cite{Stoner22+} & \cite{Stoner22+} & \cite{Stoner22+} & \cite{Stoner22+} & \cite{Stoner22+} & $\st$\\
\end{tblr}
\caption{The relation $H\succcurlyeq T$ restricted to trees $H$ and $T$ with $H\in\{T_{37},\dots,T_{47}\}$ and $v(T)\leq 6$.}
\end{table}

\begin{table}[htbp]
\centering
\begin{tblr}{
colspec = {|c|c|c|c|c|c|c|c|c|c|c|c|},
rowspec = {|c|c|c|c|c|c|c|c|c|c|c|c|},
cell{2}{2} = {lightgray},
cell{2}{3} = {lightgray},
cell{2}{5} = {lightgray},
cell{2}{6} = {lightgray},
cell{2}{7} = {lightgray},
cell{2}{8} = {lightgray},
cell{2}{9} = {lightgray},
cell{2}{11} = {lightgray},
cell{2}{12} = {lightgray},
cell{3}{3} = {lightgray},
cell{3}{5} = {lightgray},
cell{3}{6} = {lightgray},
cell{3}{11} = {lightgray},
cell{3}{12} = {lightgray},
cell{4}{4} = {lightgray},
cell{4}{5} = {lightgray},
cell{4}{8} = {lightgray},
cell{4}{11} = {lightgray},
cell{4}{12} = {lightgray},
cell{5}{5} = {lightgray},
cell{5}{11} = {lightgray},
cell{5}{12} = {lightgray},
cell{6}{5} = {lightgray},
cell{6}{6} = {lightgray},
cell{6}{11} = {lightgray},
cell{6}{12} = {lightgray},
cell{7}{7} = {lightgray},
cell{7}{8} = {lightgray},
cell{7}{11} = {lightgray},
cell{7}{12} = {lightgray},
cell{8}{8} = {lightgray},
cell{8}{11} = {lightgray},
cell{8}{12} = {lightgray},
cell{9}{7} = {lightgray},
cell{9}{8} = {lightgray},
cell{9}{9} = {lightgray},
cell{9}{11} = {lightgray},
cell{9}{12} = {lightgray},
cell{10}{8} = {lightgray},
cell{10}{10} = {lightgray},
cell{10}{11} = {lightgray},
cell{10}{12} = {lightgray},
cell{11}{11} = {lightgray},
cell{11}{12} = {lightgray},
cell{12}{12} = {lightgray},
vline{2}={1.5pt},
hline{2}={1.5pt},
}
 & $T_{14}$ & $T_{15}$ & $T_{16}$ & $T_{17}$ & $T_{18}$ & $T_{19}$ & $T_{20}$ & $T_{21}$ & $T_{22}$ & $T_{23}$ & $T_{24}$\\
$T_{14}$ &  & \cite{Sidorenko85} & \cite{Sidorenko94} & \cite{Sidorenko85} & \cite{Sidorenko94} & \cite{Leontovich89} & \cite{Sidorenko94} & \cite{Leontovich89} & \ref{th:pathsNotUnique} & \cite{Sidorenko85} & \cite{Hoffman67}\\
$T_{15}$ & $\as$ &  & \cite{Sidorenko94} & \cite{Sidorenko85} & \cite{Sidorenko94} & \cite{Sidorenko94} & \ref{dual:15-20} & \cite{Sidorenko94} & \ref{th:nearStar} & \cite{Sidorenko85} & \cite{Sidorenko85}\\
$T_{16}$ & \cite{Leontovich89} & \cite{Sidorenko94} &  & \cite{Sidorenko94} & \ref{dual:16-18} & \ref{dual:16-19} & \cite{Sidorenko94} & \cite{Sidorenko94} & \ref{th:nearStar} & \cite{Sidorenko94} & \cite{Sidorenko94}\\
$T_{17}$ & $\as$ & $\as$ & \cite{Leontovich89} &  & $\as$ & \cite{Leontovich89} & \cite{Sidorenko94} & \cite{Leontovich89} & \cite{Sidorenko94} & \cite{Sidorenko85} & \cite{Sidorenko85}\\
$T_{18}$ & \cite{Leontovich89} & \cite{Leontovich89} & \cite{Leontovich89} & \cite{Leontovich89} &  & \cite{Sidorenko94} & \cite{Sidorenko94} & \cite{Leontovich89} & \cite{Sidorenko94} & \cite{Sidorenko94} & \cite{Sidorenko94}\\
$T_{19}$ & $\as$ & \cite{Leontovich89} & \cite{Leontovich89} & \cite{Leontovich89} & \cite{Leontovich89} &  & \cite{Sidorenko94} & $\as$ & \cite{Leontovich89} & \cite{Sidorenko94} & \cite{Sidorenko94}\\
$T_{20}$ & \cite{Leontovich89} & \cite{Leontovich89} & \cite{Leontovich89} & \cite{Leontovich89} & \cite{Leontovich89} & \cite{Leontovich89} &  & \cite{Leontovich89} & $\as$ & \cite{Leontovich89} & \cite{Sidorenko94}\\
$T_{21}$ & $\as$ & \cite{Leontovich89} & \cite{Leontovich89} & \cite{Leontovich89} & \cite{Leontovich89} & \cite{Leontovich89} & \cite{Sidorenko94} &  & \cite{Leontovich89} & \cite{Sidorenko94} & \cite{Sidorenko94}\\
$T_{22}$ & \cite{Leontovich89} & \cite{Leontovich89} & \cite{Leontovich89} & \cite{Sidorenko94} & \cite{Leontovich89} & \cite{Leontovich89} & \cite{Leontovich89} & \cite{Leontovich89} &  & \cite{Leontovich89} & \cite{Sidorenko94}\\
$T_{23}$ & $\as$ & $\as$ & \cite{Leontovich89} & $\as$ & \cite{Leontovich89} & \cite{Leontovich89} & $\as$ & \cite{Leontovich89} & $\as$ &  & \cite{Sidorenko85}\\
$T_{24}$ & $\as$ & $\as$ & \cite{Leontovich89} & $\as$ & \cite{Leontovich89} & \cite{Leontovich89} & \cite{Leontovich89} & \cite{Leontovich89} & \cite{Leontovich89} & $\as$ & \\
\end{tblr}
\caption{The relation $\succcurlyeq$ restricted to $7$-vertex trees.}
\end{table}

\begin{table}[htbp]
\centering
\begin{tblr}{
colspec = {|c|c|c|c|c|c|c|c|c|c|c|c|c|},
rowspec = {|c|c|c|c|c|c|c|c|c|c|c|c|},
cell{2}{3} = {lightgray},
cell{2}{6} = {lightgray},
cell{2}{9} = {lightgray},
cell{2}{13} = {lightgray},
cell{3}{9} = {lightgray},
cell{3}{13} = {lightgray},
cell{4}{12} = {lightgray},
cell{4}{13} = {lightgray},
cell{5}{13} = {lightgray},
cell{6}{13} = {lightgray},
vline{2}={1.5pt},
hline{2}={1.5pt},
}
 & $T_{25}$ & $T_{26}$ & $T_{27}$ & $T_{28}$ & $T_{29}$ & $T_{30}$ & $T_{31}$ & $T_{32}$ & $T_{33}$ & $T_{34}$ & $T_{35}$ & $T_{36}$\\
$T_{14}$ & \cite{Stoner22+} & \ref{primal:14-26} & \ref{lem:sigma} & \cite{Stoner22+} & \ref{primal:14-29} & \cite{Stoner22+} & \ref{dual:14-31} & $\tr_{26}$ & \ref{lem:sigma} & \cite{Stoner22+} & \ref{dual:14-35} & $\tr_{26}$\\
$T_{15}$ & \cite{Stoner22+} & \ref{dual:15-26} & \ref{lem:sigma} & \cite{Stoner22+} & \ref{dual:15-29} & \cite{Stoner22+} & $\tr_{14}$ & \ref{primal:15-32} & \ref{lem:sigma} & \cite{Stoner22+} & $\tr_{14}$ & $\tr_{32}$\\
$T_{16}$ & \cite{Stoner22+} & $\tr_{32}$ & \ref{lem:sigma} & \cite{Stoner22+} & \ref{dual:16-29} & \cite{Stoner22+} & $\tr_{29}$ & \ref{dual:16-32} & \ref{lem:sigma} & \cite{Stoner22+} & \ref{primal:16-35} & $\tr_{35}$\\
$T_{17}$ & \cite{Stoner22+} & \cite{Stoner22+} & \cite{Stoner22+} & \cite{Stoner22+} & \cite{Stoner22+} & \cite{Stoner22+} & \cite{Stoner22+} & \cite{Stoner22+} & \cite{Stoner22+} & \cite{Stoner22+} & \cite{Stoner22+} & \ref{primal:17-36}\\
$T_{18}$ & \cite{Stoner22+} & \cite{Stoner22+} & \cite{Stoner22+} & \cite{Stoner22+} & \cite{Stoner22+} & \cite{Stoner22+} & \cite{Stoner22+} & \cite{Stoner22+} & \cite{Stoner22+} & \cite{Stoner22+} & \cite{Stoner22+} & $\tr_{17}$\\
$T_{19}$ & \cite{Stoner22+} & \ref{lem:radius} & \ref{lem:radius} & \cite{Stoner22+} & \ref{lem:radius} & \cite{Stoner22+} & \ref{lem:radius} & \ref{lem:radius} & \ref{lem:radius} & \cite{Stoner22+} & \ref{lem:radius} & $\tr_{21}$\\
$T_{20}$ & \cite{Stoner22+} & \cite{Stoner22+} & \cite{Stoner22+} & \cite{Stoner22+} & \cite{Stoner22+} & \cite{Stoner22+} & \cite{Stoner22+} & \cite{Stoner22+} & \cite{Stoner22+} & \cite{Stoner22+} & \cite{Stoner22+} & $\tr_{19}$\\
$T_{21}$ & \cite{Stoner22+} & \ref{lem:radius} & \ref{lem:radius} & \cite{Stoner22+} & \ref{lem:radius} & \cite{Stoner22+} & \ref{lem:radius} & \ref{lem:radius} & \ref{lem:radius} & \cite{Stoner22+} & \ref{lem:radius} & \ref{dual:21-36}\\
$T_{22}$ & \cite{Stoner22+} & \ref{lem:radius} & \ref{lem:radius} & \cite{Stoner22+} & \ref{lem:radius} & \cite{Stoner22+} & \ref{lem:radius} & \ref{lem:radius} & \ref{lem:radius} & \cite{Stoner22+} & \ref{lem:radius} & \ref{dual:22-36}\\
$T_{23}$ & \cite{Stoner22+} & \cite{Stoner22+} & \cite{Stoner22+} & \cite{Stoner22+} & \cite{Stoner22+} & \cite{Stoner22+} & \cite{Stoner22+} & \cite{Stoner22+} & \cite{Stoner22+} & \cite{Stoner22+} & \cite{Stoner22+} & $\tr_{12}$\\
$T_{24}$ & \cite{Stoner22+} & \cite{Stoner22+} & \cite{Stoner22+} & \cite{Stoner22+} & \cite{Stoner22+} & \cite{Stoner22+} & \cite{Stoner22+} & \cite{Stoner22+} & \cite{Stoner22+} & \cite{Stoner22+} & \cite{Stoner22+} & \cite{Stoner22+}\\
\end{tblr}
\caption{The relation $H\succcurlyeq T$ restricted to trees $H$ and $T$ with $H\in\{T_{25},\dots,T_{36}\}$ and $v(T)=7$.}
\end{table}

\begin{table}[htbp]
\centering
\begin{tblr}{
colspec = {|c|c|c|c|c|c|c|c|c|c|c|c|},
rowspec = {|c|c|c|c|c|c|c|c|c|c|c|c|},
cell{2}{2} = {lightgray},
cell{2}{3} = {lightgray},
cell{2}{5} = {lightgray},
cell{2}{6} = {lightgray},
cell{2}{8} = {lightgray},
cell{2}{11} = {lightgray},
cell{2}{12} = {lightgray},
cell{3}{2} = {lightgray},
cell{3}{3} = {lightgray},
cell{3}{5} = {lightgray},
cell{3}{8} = {lightgray},
cell{3}{11} = {lightgray},
cell{3}{12} = {lightgray},
cell{4}{2} = {lightgray},
cell{4}{3} = {lightgray},
cell{4}{5} = {lightgray},
cell{4}{8} = {lightgray},
cell{4}{10} = {lightgray},
cell{4}{11} = {lightgray},
cell{4}{12} = {lightgray},
cell{5}{11} = {lightgray},
cell{5}{12} = {lightgray},
cell{6}{2} = {lightgray},
cell{6}{11} = {lightgray},
cell{6}{12} = {lightgray},
cell{7}{3} = {lightgray},
cell{7}{5} = {lightgray},
cell{7}{8} = {lightgray},
cell{7}{11} = {lightgray},
cell{7}{12} = {lightgray},
cell{8}{8} = {lightgray},
cell{8}{11} = {lightgray},
cell{8}{12} = {lightgray},
cell{9}{3} = {lightgray},
cell{9}{5} = {lightgray},
cell{9}{8} = {lightgray},
cell{9}{11} = {lightgray},
cell{9}{12} = {lightgray},
cell{10}{8} = {lightgray},
cell{10}{10} = {lightgray},
cell{10}{11} = {lightgray},
cell{10}{12} = {lightgray},
cell{11}{11} = {lightgray},
cell{11}{12} = {lightgray},
cell{12}{12} = {lightgray},
vline{2}={1.5pt},
hline{2}={1.5pt},
}
 & $T_{37}$ & $T_{38}$ & $T_{39}$ & $T_{40}$ & $T_{41}$ & $T_{42}$ & $T_{43}$ & $T_{44}$ & $T_{45}$ & $T_{46}$ & $T_{47}$\\
$T_{14}$ & $\tr_{26}$ & $\tr_{26}$ & \ref{lem:sigma} & $\tr_{26}$ & \ref{primal:14-41} & \ref{lem:sigma} & $\tr_{26}$ & \cite{Stoner22+} & \ref{dual:14-45} & $\tr_{26}$ & \ref{lem:pathBalance}\\
$T_{15}$ & $\tr_{18}$ & $\tr_{40}$ & \ref{lem:sigma} & \ref{primal:15-40} & \ref{dual:15-41} & \ref{lem:sigma} & $\tr_{38}$ & \cite{Stoner22+} & $\tr_{14}$ & $\tr_{32}$ & $\tr_{24}$\\
$T_{16}$ & \ref{primal:16-37} & $\tr_{40}$ & \ref{lem:sigma} & \ref{primal:16-40} & \ref{dual:16-41} & \ref{lem:sigma} & $\tr_{38}$ & \cite{Stoner22+} & \ref{primal:16-45} & $\tr_{35}$ & $\tr_{24}$\\
$T_{17}$ & \ref{dual:17-37} & \cite{Stoner22+} & \ref{lem:sigma} & \cite{Stoner22+} & \cite{Stoner22+} & \cite{Stoner22+} & \ref{lem:sigma} & \cite{Stoner22+} & \cite{Stoner22+} & $\tr_{36}$ & $\tr_{24}$\\
$T_{18}$ & \ref{primal:18-37} & \cite{Stoner22+} & \ref{lem:sigma} & \cite{Stoner22+} & \cite{Stoner22+} & \cite{Stoner22+} & \ref{lem:sigma} & \cite{Stoner22+} & \cite{Stoner22+} & $\tr_{17}$ & $\tr_{17}$\\
$T_{19}$ & $\tr_{36}$ & $\tr_{40}$ & \ref{lem:sigma} & \ref{primal:19-40} & $\tr_{21}$ & \ref{lem:sigma} & $\tr_{38}$ & \cite{Stoner22+} & $\tr_{14}$ & $\tr_{38}$ & $\tr_{24}$\\
$T_{20}$ & $\tr_{19}$ & \cite{Stoner22+} & \ref{lem:sigma} & \cite{Stoner22+} & \cite{Stoner22+} & \cite{Stoner22+} & \ref{primal:20-43} & \cite{Stoner22+} & \cite{Stoner22+} & $\tr_{43}$ & $\tr_{24}$\\
$T_{21}$ & $\tr_{36}$ & $\tr_{19}$ & \ref{lem:sigma} & $\tr_{19}$ & \ref{dual:21-41} & \ref{lem:sigma} & $\tr_{19}$ & \cite{Stoner22+} & $\tr_{14}$ & $\tr_{19}$ & $\tr_{19}$\\
$T_{22}$ & $\tr_{36}$ & \ref{dual:22-38} & \ref{lem:sigma} & $\tr_{38}$ & $\tr_{6}$ & \ref{lem:sigma} & $\tr_{20}$ & \cite{Stoner22+} & \ref{primal:22-45} & $\tr_{20}$ & $\tr_{20}$\\
$T_{23}$ & $\tr_{12}$ & \cite{Stoner22+} & \ref{lem:sigma} & \cite{Stoner22+} & \cite{Stoner22+} & \cite{Stoner22+} & \ref{lem:sigma} & \cite{Stoner22+} & \cite{Stoner22+} & \ref{primal:23-46} & $\tr_{24}$\\
$T_{24}$ & \cite{Stoner22+} & \cite{Stoner22+} & \cite{Stoner22+} & \cite{Stoner22+} & \cite{Stoner22+} & \cite{Stoner22+} & \cite{Stoner22+} & \cite{Stoner22+} & \cite{Stoner22+} & \cite{Stoner22+} & $\st$\\
\end{tblr}
\caption{The relation $H\succcurlyeq T$ restricted to trees $H$ and $T$ with $H\in\{T_{37},\dots,T_{47}\}$ and $v(T)=7$.}
\end{table}

\begin{table}[htbp]
\centering
\begin{tblr}{
colspec = {|c|c|c|c|c|c|c|c|c|c|c|c|c|},
rowspec = {|c|c|c|c|c|c|c|c|c|c|c|c|c|},
cell{2}{2} = {lightgray},
cell{2}{3} = {lightgray},
cell{2}{4} = {lightgray},
cell{2}{6} = {lightgray},
cell{2}{8} = {lightgray},
cell{2}{9} = {lightgray},
cell{2}{10} = {lightgray},
cell{2}{11} = {lightgray},
cell{2}{13} = {lightgray},
cell{3}{3} = {lightgray},
cell{3}{9} = {lightgray},
cell{3}{13} = {lightgray},
cell{4}{4} = {lightgray},
cell{4}{9} = {lightgray},
cell{4}{13} = {lightgray},
cell{5}{5} = {lightgray},
cell{5}{9} = {lightgray},
cell{5}{13} = {lightgray},
cell{6}{6} = {lightgray},
cell{6}{13} = {lightgray},
cell{7}{7} = {lightgray},
cell{7}{12} = {lightgray},
cell{7}{13} = {lightgray},
cell{8}{6} = {lightgray},
cell{8}{8} = {lightgray},
cell{8}{13} = {lightgray},
cell{9}{9} = {lightgray},
cell{9}{13} = {lightgray},
cell{10}{10} = {lightgray},
cell{10}{13} = {lightgray},
cell{11}{10} = {lightgray},
cell{11}{11} = {lightgray},
cell{11}{13} = {lightgray},
cell{12}{12} = {lightgray},
cell{12}{13} = {lightgray},
cell{13}{13} = {lightgray},
vline{2}={1.5pt},
hline{2}={1.5pt},
}
 & $T_{25}$ & $T_{26}$ & $T_{27}$ & $T_{28}$ & $T_{29}$ & $T_{30}$ & $T_{31}$ & $T_{32}$ & $T_{33}$ & $T_{34}$ & $T_{35}$ & $T_{36}$\\
$T_{25}$ &  & \cite{Sidorenko85} & \ref{primal:25-27} & \ref{dual:25-28} & $\tr_{31}$ & $\tr_{35}$ & \ref{primal:25-31} & \cite{Sidorenko85} & $\tr_{34}$ & \ref{primal:25-34} & \ref{dual:25-35} & \cite{Sidorenko85}\\
$T_{26}$ & $\as$ &  & \cite{Sidorenko94} & \cite{Stoner22+} & \ref{dual:26-29} & \cite{Stoner22+} & \cite{Sidorenko94} & \cite{Sidorenko85} & \cite{Sidorenko94} & \cite{Sidorenko94} & \ref{lem:lambda} & \cite{Sidorenko85}\\
$T_{27}$ & \cite{Leontovich89} & \cite{Leontovich89} &  & \cite{Leontovich89} & \cite{Sidorenko94} & \cite{Sidorenko94} & \cite{Leontovich89} & \cite{Leontovich89} & \cite{Sidorenko94} & \cite{Leontovich89} & \ref{lem:lambda} & $\tr_{32}$\\
$T_{28}$ & \cite{Leontovich89} & \ref{dual:28-26} & \ref{dual:28-27} &  & \ref{dual:28-29} & $\tr_{35}$ & \cite{Sidorenko94} & \ref{primal:28-32} & \ref{dual:28-33} & \cite{Sidorenko94} & \ref{dual:28-35} & $\tr_{32}$\\
$T_{29}$ & \cite{Leontovich89} & \cite{Leontovich89} & \cite{Sidorenko94} & \cite{Leontovich89} &  & \cite{Sidorenko94} & \cite{Leontovich89} & $\tr_{31}$ & \cite{Sidorenko94} & \cite{Leontovich89} & \ref{lem:lambda} & $\tr_{37}$\\
$T_{30}$ & \cite{Leontovich89} & \cite{Leontovich89} & \cite{Sidorenko94} & \cite{Leontovich89} & \cite{Sidorenko94} &  & \cite{Leontovich89} & \ref{dual:30-32} & \cite{Sidorenko94} & \cite{Leontovich89} & \ref{primal:30-35} & $\tr_{35}$\\
$T_{31}$ & \cite{Leontovich89} & \cite{Sidorenko94} & \cite{Sidorenko94} & \cite{Sidorenko94} & \ref{primal:31-29} & \cite{Stoner22+} &  & \ref{dual:31-32} & \cite{Sidorenko94} & \cite{Stoner22+} & $\tr_{25}$ & $\tr_{29}$\\
$T_{32}$ & $\as$ & $\as$ & $\as$ & \cite{Leontovich89} & \cite{Leontovich89} & \cite{Leontovich89} & \cite{Leontovich89} &  & \cite{Leontovich89} & \cite{Leontovich89} & \cite{Sidorenko94} & \cite{Sidorenko85}\\
$T_{33}$ & \cite{Leontovich89} & \cite{Leontovich89} & \cite{Sidorenko94} & \cite{Leontovich89} & $\tr_{34}$ & \cite{Sidorenko94} & \cite{Leontovich89} & $\tr_{34}$ &  & \cite{Leontovich89} & \ref{lem:lambda} & $\tr_{37}$\\
$T_{34}$ & \cite{Leontovich89} & \cite{Sidorenko94} & $\tr_{32}$ & \cite{Sidorenko94} & \ref{dual:34-29} & $\tr_{35}$ & $\tr_{29}$ & \ref{dual:34-32} & \ref{primal:34-33} &  & $\tr_{25}$ & $\tr_{33}$\\
$T_{35}$ & \cite{Leontovich89} & \cite{Leontovich89} & \cite{Leontovich89} & \cite{Leontovich89} & \cite{Leontovich89} & \cite{Leontovich89} & \cite{Leontovich89} & \cite{Sidorenko94} & \cite{Leontovich89} & \cite{Leontovich89} &  & \ref{primal:35-36}\\
$T_{36}$ & $\as$ & $\as$ & $\as$ & \cite{Leontovich89} & \cite{Leontovich89} & \cite{Leontovich89} & \cite{Leontovich89} & $\as$ & \cite{Leontovich89} & \cite{Leontovich89} & \cite{Leontovich89} & \\
\end{tblr}
\caption{The relation $H\succcurlyeq T$ restricted to trees $H$ and $T$ with $H,T\in\{T_{25},\dots,T_{36}\}$.}
\end{table}

\begin{table}[htbp]
\centering
\begin{tblr}{
colspec = {|c|c|c|c|c|c|c|c|c|c|c|c|},
rowspec = {|c|c|c|c|c|c|c|c|c|c|c|c|},
cell{2}{2} = {lightgray},
cell{2}{11} = {lightgray},
cell{2}{12} = {lightgray},
cell{3}{3} = {lightgray},
cell{3}{8} = {lightgray},
cell{3}{11} = {lightgray},
cell{3}{12} = {lightgray},
cell{4}{4} = {lightgray},
cell{4}{8} = {lightgray},
cell{4}{11} = {lightgray},
cell{4}{12} = {lightgray},
cell{5}{3} = {lightgray},
cell{5}{5} = {lightgray},
cell{5}{8} = {lightgray},
cell{5}{11} = {lightgray},
cell{5}{12} = {lightgray},
cell{6}{3} = {lightgray},
cell{6}{5} = {lightgray},
cell{6}{6} = {lightgray},
cell{6}{8} = {lightgray},
cell{6}{11} = {lightgray},
cell{6}{12} = {lightgray},
cell{7}{4} = {lightgray},
cell{7}{7} = {lightgray},
cell{7}{8} = {lightgray},
cell{7}{11} = {lightgray},
cell{7}{12} = {lightgray},
cell{8}{8} = {lightgray},
cell{8}{11} = {lightgray},
cell{8}{12} = {lightgray},
cell{9}{4} = {lightgray},
cell{9}{7} = {lightgray},
cell{9}{8} = {lightgray},
cell{9}{9} = {lightgray},
cell{9}{11} = {lightgray},
cell{9}{12} = {lightgray},
cell{10}{8} = {lightgray},
cell{10}{10} = {lightgray},
cell{10}{11} = {lightgray},
cell{10}{12} = {lightgray},
cell{11}{11} = {lightgray},
cell{11}{12} = {lightgray},
cell{12}{12} = {lightgray},
vline{2}={1.5pt},
hline{2}={1.5pt},
}
 & $T_{37}$ & $T_{38}$ & $T_{39}$ & $T_{40}$ & $T_{41}$ & $T_{42}$ & $T_{43}$ & $T_{44}$ & $T_{45}$ & $T_{46}$ & $T_{47}$\\
$T_{37}$ &  & \cite{Sidorenko94} & \cite{Sidorenko94} & \cite{Leontovich89} & \cite{Leontovich89} & \cite{Sidorenko94} & \cite{Sidorenko94} & \cite{Leontovich89} & \cite{Stoner22+} & \cite{Sidorenko94} & \cite{Sidorenko94}\\
$T_{38}$ & \cite{Sidorenko94} &  & \cite{Sidorenko94} & \cite{Leontovich89} & \cite{Leontovich89} & \cite{Sidorenko94} & \ref{primal:38-43} & \cite{Leontovich89} & \ref{lem:lambda} & \cite{Sidorenko94} & \cite{Sidorenko94}\\
$T_{39}$ & \cite{Leontovich89} & \cite{Leontovich89} &  & \cite{Leontovich89} & \cite{Leontovich89} & \cite{Leontovich89} & \cite{Leontovich89} & \cite{Leontovich89} & \cite{Leontovich89} & \cite{Leontovich89} & \cite{Sidorenko94}\\
$T_{40}$ & $\tr_{19}$ & \ref{primal:40-38} & \cite{Sidorenko94} &  & \cite{Leontovich89} & \cite{Sidorenko94} & $\tr_{38}$ & \cite{Leontovich89} & \ref{lem:lambda} & \cite{Sidorenko94} & \cite{Sidorenko94}\\
$T_{41}$ & $\tr_{36}$ & $\tr_{40}$ & \cite{Sidorenko94} & \ref{primal:41-40} &  & \cite{Sidorenko94} & $\tr_{38}$ & \cite{Stoner22+} & \ref{lem:lambda} & \cite{Sidorenko94} & \cite{Sidorenko94}\\
$T_{42}$ & \cite{Sidorenko94} & \cite{Sidorenko94} & \ref{primal:42-39} & \cite{Leontovich89} & \cite{Leontovich89} &  & $\tr_{39}$ & \cite{Leontovich89} & \ref{lem:lambda} & \cite{Sidorenko94} & \cite{Sidorenko94}\\
$T_{43}$ & \cite{Leontovich89} & \cite{Leontovich89} & $\as$ & \cite{Leontovich89} & \cite{Leontovich89} & \cite{Leontovich89} &  & \cite{Leontovich89} & \cite{Leontovich89} & \cite{Leontovich89} & \cite{Sidorenko94}\\
$T_{44}$ & $\tr_{36}$ & \ref{dual:44-38} & $\tr_{42}$ & \cite{Leontovich89} & \cite{Leontovich89} & \ref{primal:44-42} & $\tr_{39}$ &  & $\tr_{25}$ & \cite{Sidorenko94} & \cite{Sidorenko94}\\
$T_{45}$ & \cite{Leontovich89} & \cite{Leontovich89} & \cite{Leontovich89} & \cite{Leontovich89} & \cite{Leontovich89} & \cite{Leontovich89} & \ref{primal:45-43} & \cite{Leontovich89} &  & \cite{Sidorenko94} & \cite{Sidorenko94}\\
$T_{46}$ & \cite{Leontovich89} & \cite{Leontovich89} & $\as$ & \cite{Leontovich89} & \cite{Leontovich89} & \cite{Leontovich89} & $\as$ & \cite{Leontovich89} & \cite{Leontovich89} &  & \cite{Sidorenko85}\\
$T_{47}$ & \cite{Leontovich89} & \cite{Leontovich89} & \cite{Leontovich89} & \cite{Leontovich89} & \cite{Leontovich89} & \cite{Leontovich89} & \cite{Leontovich89} & \cite{Leontovich89} & \cite{Leontovich89} & $\as$ & \\
\end{tblr}
\caption{The relation $H\succcurlyeq T$ restricted to trees $H$ and $T$ with  $H,T\in\{T_{37},\dots,T_{47}\}$.}
\end{table}

\begin{table}[htbp]
\centering
\begin{tblr}{
colspec = {|c|c|c|c|c|c|c|c|c|c|c|c|},
rowspec = {|c|c|c|c|c|c|c|c|c|c|c|c|c|},
cell{2}{2} = {lightgray},
cell{2}{3} = {lightgray},
cell{2}{4} = {lightgray},
cell{2}{5} = {lightgray},
cell{2}{6} = {lightgray},
cell{2}{7} = {lightgray},
cell{2}{8} = {lightgray},
cell{2}{9} = {lightgray},
cell{2}{11} = {lightgray},
cell{2}{12} = {lightgray},
cell{3}{2} = {lightgray},
cell{3}{3} = {lightgray},
cell{3}{5} = {lightgray},
cell{3}{8} = {lightgray},
cell{3}{11} = {lightgray},
cell{3}{12} = {lightgray},
cell{4}{2} = {lightgray},
cell{4}{3} = {lightgray},
cell{4}{5} = {lightgray},
cell{4}{8} = {lightgray},
cell{4}{11} = {lightgray},
cell{4}{12} = {lightgray},
cell{5}{2} = {lightgray},
cell{5}{3} = {lightgray},
cell{5}{4} = {lightgray},
cell{5}{5} = {lightgray},
cell{5}{7} = {lightgray},
cell{5}{8} = {lightgray},
cell{5}{9} = {lightgray},
cell{5}{11} = {lightgray},
cell{5}{12} = {lightgray},
cell{6}{2} = {lightgray},
cell{6}{8} = {lightgray},
cell{6}{11} = {lightgray},
cell{6}{12} = {lightgray},
cell{7}{2} = {lightgray},
cell{7}{3} = {lightgray},
cell{7}{4} = {lightgray},
cell{7}{5} = {lightgray},
cell{7}{6} = {lightgray},
cell{7}{7} = {lightgray},
cell{7}{8} = {lightgray},
cell{7}{9} = {lightgray},
cell{7}{10} = {lightgray},
cell{7}{11} = {lightgray},
cell{7}{12} = {lightgray},
cell{8}{2} = {lightgray},
cell{8}{3} = {lightgray},
cell{8}{5} = {lightgray},
cell{8}{6} = {lightgray},
cell{8}{8} = {lightgray},
cell{8}{11} = {lightgray},
cell{8}{12} = {lightgray},
cell{9}{11} = {lightgray},
cell{9}{12} = {lightgray},
cell{10}{2} = {lightgray},
cell{10}{3} = {lightgray},
cell{10}{4} = {lightgray},
cell{10}{5} = {lightgray},
cell{10}{8} = {lightgray},
cell{10}{11} = {lightgray},
cell{10}{12} = {lightgray},
cell{11}{2} = {lightgray},
cell{11}{3} = {lightgray},
cell{11}{4} = {lightgray},
cell{11}{5} = {lightgray},
cell{11}{6} = {lightgray},
cell{11}{7} = {lightgray},
cell{11}{8} = {lightgray},
cell{11}{9} = {lightgray},
cell{11}{11} = {lightgray},
cell{11}{12} = {lightgray},
cell{12}{8} = {lightgray},
cell{12}{11} = {lightgray},
cell{12}{12} = {lightgray},
cell{13}{11} = {lightgray},
cell{13}{12} = {lightgray},
vline{2}={1.5pt},
hline{2}={1.5pt},
}
 & $T_{37}$ & $T_{38}$ & $T_{39}$ & $T_{40}$ & $T_{41}$ & $T_{42}$ & $T_{43}$ & $T_{44}$ & $T_{45}$ & $T_{46}$ & $T_{47}$\\
$T_{25}$ & $\tr_{26}$ & $\tr_{26}$ & $\tr_{33}$ & $\tr_{26}$ & $\tr_{31}$ & $\tr_{34}$ & $\tr_{39}$ & $\tr_{34}$ & \ref{th:pathsNotUnique} & \cite{Sidorenko85} & \cite{Hoffman67}\\
$T_{26}$ & \ref{primal:26-37} & $\tr_{40}$ & \cite{Sidorenko94} & \ref{primal:26-40} & \ref{dual:26-41} & \cite{Sidorenko94} & $\tr_{38}$ & \cite{Stoner22+} & \ref{th:nearStar} & \cite{Sidorenko85} & \cite{Sidorenko85}\\
$T_{27}$ & \ref{primal:27-37} & $\tr_{40}$ & \ref{dual:27-39} & \ref{primal:27-40} & \cite{Sidorenko94} & $\tr_{39}$ & $\tr_{38}$ & \cite{Stoner22+} & \ref{lem:lambda} & \cite{Sidorenko94} & \cite{Sidorenko94}\\
$T_{28}$ & \ref{primal:28-37} & $\tr_{40}$ & $\tr_{42}$ & \ref{primal:28-40} & \ref{dual:28-41} & $\tr_{44}$ & $\tr_{38}$ & \ref{primal:28-44} & \ref{th:nearStar} & \cite{Sidorenko94} & \cite{Sidorenko94}\\
$T_{29}$ & \ref{primal:29-37} & \ref{dual:29-38} & \cite{Sidorenko94} & $\tr_{38}$ & \cite{Sidorenko94} & \cite{Sidorenko94} & \ref{primal:29-43} & \cite{Stoner22+} & \ref{lem:lambda} & \cite{Sidorenko94} & \cite{Sidorenko94}\\
$T_{30}$ & \ref{primal:30-37} & $\tr_{40}$ & $\tr_{42}$ & $\tr_{41}$ & \cite{Sidorenko94} & $\tr_{44}$ & $\tr_{38}$ & \ref{primal:30-44} & \ref{primal:30-45} & \cite{Sidorenko94} & \cite{Sidorenko94}\\
$T_{31}$ & $\tr_{29}$ & $\tr_{40}$ & \cite{Sidorenko94} & $\tr_{41}$ & \ref{primal:31-41} & \cite{Sidorenko94} & $\tr_{29}$ & \cite{Stoner22+} & \ref{th:nearStar} & \cite{Sidorenko94} & \cite{Sidorenko94}\\
$T_{32}$ & \ref{dual:32-37} & $\tr_{43}$ & \cite{Sidorenko94} & \cite{Leontovich89} & \cite{Leontovich89} & \cite{Sidorenko94} & \ref{dual:32-43} & \cite{Sidorenko94} & \ref{th:nearStar} & \cite{Sidorenko85} & \cite{Sidorenko85}\\
$T_{33}$ & \ref{primal:33-37} & $\tr_{40}$ & \ref{primal:33-39} & \ref{primal:33-40} & \cite{Sidorenko94} & \ref{dual:33-42} & $\tr_{38}$ & \cite{Stoner22+} & \ref{th:nearStar} & \cite{Sidorenko94} & \cite{Sidorenko94}\\
$T_{34}$ & $\tr_{33}$ & $\tr_{33}$ & $\tr_{33}$ & $\tr_{33}$ & \ref{primal:34-41} & $\tr_{44}$ & $\tr_{33}$ & \ref{primal:34-44} & \ref{th:nearStar} & \cite{Sidorenko94} & \cite{Sidorenko94}\\
$T_{35}$ & \ref{dual:35-37} & \ref{dual:35-38} & \cite{Sidorenko94} & \cite{Leontovich89} & \cite{Leontovich89} & \cite{Sidorenko94} & \ref{primal:35-43} & \cite{Sidorenko94} & \ref{th:nearStar} & \cite{Sidorenko94} & \cite{Sidorenko94}\\
$T_{36}$ & $\as$ & \cite{Leontovich89} & \cite{Leontovich89} & \cite{Leontovich89} & \cite{Leontovich89} & \cite{Leontovich89} & \cite{Sidorenko94} & \cite{Leontovich89} & \cite{Sidorenko94} & \cite{Sidorenko85} & \cite{Sidorenko85}\\
\end{tblr}
\caption{The relation $H\succcurlyeq T$ restricted to trees $H$ and $T$ with $H\in\{T_{37},\dots,T_{47}\}$ and $T\in\{T_{25},\dots,T_{36}\}$.}
\end{table}

\begin{table}[htbp]
\centering
\begin{tblr}{
colspec = {|c|c|c|c|c|c|c|c|c|c|c|c|c|},
rowspec = {|c|c|c|c|c|c|c|c|c|c|c|c|},
cell{2}{13} = {lightgray},
vline{2}={1.5pt},
hline{2}={1.5pt},
}
 & $T_{25}$ & $T_{26}$ & $T_{27}$ & $T_{28}$ & $T_{29}$ & $T_{30}$ & $T_{31}$ & $T_{32}$ & $T_{33}$ & $T_{34}$ & $T_{35}$ & $T_{36}$\\
$T_{37}$ & \cite{Leontovich89} & \cite{Leontovich89} & \cite{Leontovich89} & \cite{Leontovich89} & \cite{Leontovich89} & \cite{Leontovich89} & \cite{Leontovich89} & \cite{Leontovich89} & \cite{Leontovich89} & \cite{Leontovich89} & \cite{Leontovich89} & \cite{Leontovich89}\\
$T_{38}$ & \cite{Leontovich89} & \cite{Leontovich89} & \cite{Leontovich89} & \cite{Leontovich89} & \cite{Leontovich89} & \cite{Leontovich89} & \cite{Leontovich89} & \cite{Leontovich89} & \cite{Leontovich89} & \cite{Leontovich89} & \cite{Leontovich89} & $\tr_{19}$\\
$T_{39}$ & \cite{Leontovich89} & \cite{Leontovich89} & \cite{Leontovich89} & \cite{Leontovich89} & \cite{Leontovich89} & \cite{Leontovich89} & \cite{Leontovich89} & \cite{Leontovich89} & \cite{Leontovich89} & \cite{Leontovich89} & \cite{Leontovich89} & \cite{Leontovich89}\\
$T_{40}$ & \cite{Leontovich89} & \cite{Leontovich89} & \cite{Leontovich89} & \cite{Leontovich89} & \cite{Leontovich89} & \cite{Leontovich89} & \cite{Leontovich89} & \cite{Leontovich89} & \cite{Leontovich89} & \cite{Leontovich89} & \cite{Leontovich89} & $\tr_{19}$\\
$T_{41}$ & \cite{Leontovich89} & \cite{Leontovich89} & \cite{Leontovich89} & \cite{Leontovich89} & \cite{Leontovich89} & \cite{Leontovich89} & \cite{Leontovich89} & \cite{Leontovich89} & \cite{Leontovich89} & \cite{Leontovich89} & \cite{Leontovich89} & \ref{dual:41-36}\\
$T_{42}$ & \cite{Leontovich89} & \cite{Leontovich89} & \cite{Leontovich89} & \cite{Leontovich89} & \cite{Leontovich89} & \cite{Leontovich89} & \cite{Leontovich89} & \cite{Leontovich89} & \cite{Leontovich89} & \cite{Leontovich89} & \cite{Leontovich89} & $\tr_{44}$\\
$T_{43}$ & \cite{Leontovich89} & \cite{Leontovich89} & \cite{Leontovich89} & \cite{Leontovich89} & \cite{Leontovich89} & \cite{Leontovich89} & \cite{Leontovich89} & \cite{Leontovich89} & \cite{Leontovich89} & \cite{Leontovich89} & \cite{Leontovich89} & \cite{Leontovich89}\\
$T_{44}$ & \cite{Leontovich89} & \cite{Leontovich89} & \cite{Leontovich89} & \cite{Leontovich89} & \cite{Leontovich89} & \cite{Leontovich89} & \cite{Leontovich89} & \cite{Leontovich89} & \cite{Leontovich89} & \cite{Leontovich89} & \cite{Leontovich89} & \ref{dual:44-36}\\
$T_{45}$ & \cite{Leontovich89} & \cite{Leontovich89} & \cite{Leontovich89} & \cite{Leontovich89} & \cite{Leontovich89} & \cite{Leontovich89} & \cite{Leontovich89} & \cite{Leontovich89} & \cite{Leontovich89} & \cite{Leontovich89} & \cite{Leontovich89} & \cite{Sidorenko94}\\
$T_{46}$ & $\as$ & $\as$ & \cite{Leontovich89} & \cite{Leontovich89} & \cite{Leontovich89} & \cite{Leontovich89} & \cite{Leontovich89} & $\as$ & \cite{Leontovich89} & \cite{Leontovich89} & \cite{Leontovich89} & $\as$\\
$T_{47}$ & $\as$ & $\as$ & \cite{Leontovich89} & \cite{Leontovich89} & \cite{Leontovich89} & \cite{Leontovich89} & \cite{Leontovich89} & $\as$ & \cite{Leontovich89} & \cite{Leontovich89} & \cite{Leontovich89} & $\as$\\
\end{tblr}
\caption{The relation $H\succcurlyeq T$ restricted to trees $H$ and $T$ with $H\in\{T_{25},\dots,T_{36}\}$ and $T\in\{T_{37},\dots,T_{47}\}$.}
\end{table}

\section{Concluding Remarks}
\label{sec:conclusion}

While Theorem~\ref{th:pathsNotUnique} shows that Theorem~\ref{th:P4} does not directly generalize to longer paths, we believe that the following weaker generalization may hold. This statement, if true, would support the rough intuition that path-like graphs are near the bottom of the partial order restricted to trees.

\begin{conj}
For any $k\geq1$, there exists $n_0(k)$ such that if $H$ is a tree with at least $n_0(k)$ vertices, then $H\succcurlyeq P_{2k}$.
\end{conj}

In this paper, we have focused mainly on trees. A modest, but still interesting, extension of the results in this paper could be to study $\succcurlyeq$ restricted to forests. We remark that the full structure of the relation $\succcurlyeq$ on forests consisting of disjoint unions of paths is determined by~\cite[Theorem~1.3]{BlekhermanRaymond22}.

The ideas in this paper can be equivalently expressed in terms of graph limits. A \emph{kernel} is a bounded measurable function $U:[0,1]^2\to\mathbb{R}$ such that $U(x,y)=U(y,x)$ for all $x,y\in[0,1]$. The \emph{homomorphism density} of a graph $H$ in a kernel $U$ is defined to be
\[t(H,U):=\int_{[0,1]^{V(H)}}\prod_{uv\in E(H)}U(x_u,x_v)\prod_{u\in V(H)}dx_{u}.\]
A kernel $W$ such that $0\leq W(x,y)\leq 1$ for all $x,y\in[0,1]$ is called a \emph{graphon}. Standard results in graph limit theory (see~\cite{Lovasz12}) tell us that $H\succcurlyeq T$ if and only if $t(H,W)^{e(T)}\geq t(T,W)^{e(H)}$ for every graphon $W$. It would be interesting to study the relation in which $H$ is related to $T$ if  $t(H,U)^{e(T)}\geq t(T,U)^{e(H)}$ for every kernel $U$. Note that this relation is much more restrictive than $\succcurlyeq$. For example, by letting $U$ be the constant kernel $U:=-1$, we have $t(H,U)=(-1)^{e(H)}$ for every graph $H$ and so no graph with an odd number of edges can be related to a graph with an even number of edges under this relation. Inequalities involving homomorphism densities of kernels are frequently studied and have many applications; see, e.g.,~\cite{KimLee22+,Lovasz11}. It is unclear to us whether the approach in this paper, or a modification of it, can be applied in this setting.

While our results provide some insight into the structure of the partial order $\succcurlyeq$ on the set of all trees, our understanding of this poset in general is very limited. Two basic parameters of a poset are its \emph{height} and \emph{width}, which are the sizes of the largest chain and antichain, respectively. It would be an interesting (and possibly very challenging) problem to obtain good asymptotic estimates of the height and width of the poset $\succcurlyeq$ restricted to $n$-vertex trees as $n$ tends to infinity. 

Of course, it would be extremely interesting to gain a better understanding of the relation $\succcurlyeq$ on general pairs of graphs, although this would be quite challenging; indeed, a classification of graphs $H$ satisfying $H\succcurlyeq K_2$ would resolve Sidorenko's Conjecture. Note that the linear programming approach of~\cite{KoppartyRossman11} applied in this paper can be generalized beyond the setting of forests, but obtaining a tight bound on $t(H,G)$ in terms of $t(T,G)$ for all $G$ using this method requires some conditions on $H$ and $T$. 

As a very open-ended direction for possible future work, it would be interesting to study similar partial orders for other combinatorial structures, such as digraphs, hypergraphs, tournaments, etc.

\appendix

\section{Primal Certificates}
\label{app:Primalcertificates}

We provide certificates for the value of $LP(H,T)$ being equal to $e(T)$ for various trees $H$ and $T$ on at most 8 vertices. All such trees $T$ are listed in Appendix~\ref{app:smallTrees} together with a labelling of the vertices of $T$ using the numbers $1,\dots,v(T)$. Throughout the following constructions, we describe a homomorphism $\varphi$ from $H$ to $T$ as the vector $(\varphi(1),\dots,\varphi(v(T))$. Whenever we specify the weighting of $\Hom(H,T)$, any homomorphisms that are not listed are assumed to have weight zero.

{\footnotesize
\begin{const}
\label{primal:5-9}
The following weighting certifies that the value of $LP(T_{9},T_{5})$ is $4$.
\[w (1,4,5,5,2,1)  =  1/5,  \qquad w (1,4,5,5,4,5)  =  1/5,  \qquad w (2,1,4,2,1,4)  =  1/5, \]
\[w (3,2,3,3,2,3)  =  1/5.\]
\end{const}

\begin{const}
\label{primal:5-16}
The following weighting certifies that the value of $LP(T_{16},T_5)$ is $4$.
\[w (2,1,2,3,1,4,3)  =  1/4,  \qquad w (2,3,2,3,3,2,3)  =  1/12,  \qquad w (4,1,4,5,1,2,5)  =  1/4, \]
\[w (4,5,4,5,5,4,5)  =  1/12.\]
\end{const}




\begin{const}
\label{primal:5-31}
The following weighting certifies that the value of $LP(T_{31},T_{5})$ is $4$.
\[w (4,1,2,3,1,2,3,5)  =  1/2,  \qquad w (4,5,4,5,5,4,5,5)  =  1/14.\]
\end{const}

\begin{const}
\label{primal:6-12}
The following weighting certifies that the value of $LP(T_{12},T_{6})$ is $4$.
\[w (1,2,3,4,2,4)  =  3/10,  \qquad w (1,2,3,5,5,2)  =  1/10,  \qquad w (1,5,1,4,5,4)  =  1/5, \]
\[w (2,1,5,3,3,3)  =  1/5.\]
\end{const}


\begin{const}
\label{primal:6-18}
The following weighting certifies that the value of $LP(T_{18},T_{6})$ is $4$.
\[w (2,1,4,5,1,4,5)  =  1/2,  \qquad w (2,3,2,2,3,2,2)  =  1/18,  \qquad w (3,2,3,3,2,3,3)  =  1/9.\]
\end{const}


\begin{const}
\label{primal:6-22}
The following weighting certifies that the value of $LP(T_{22},T_{6})$ is $4$.
\[w (1,2,3,2,3,4,4)  =  1/2,  \qquad w (1,5,1,5,1,5,5)  =  1/6.\]
\end{const}

\begin{const}
\label{primal:6-29}
The following weighting certifies that the value of $LP(T_{29},T_{6})$ is $4$.
\[w (2,1,2,3,4,1,4,2)  =  1/14,  \qquad w (2,1,2,3,5,1,5,5)  =  1/7,  \qquad w (3,2,3,2,3,2,3,3)  =  1/14, \]
\[w (5,1,2,3,4,1,4,4)  =  2/7.\]
\end{const}

\begin{const}
\label{primal:6-32}
The following weighting certifies that the value of $LP(T_{32},T_{6})$ is $4$.
\[w (1,2,1,5,2,3,3,3)  =  1/14,  \qquad w (2,1,2,3,1,5,4,2)  =  1/14,  \qquad w (2,3,2,3,1,4,4,2)  =  1/14, \]
\[w (2,3,2,3,1,4,4,4)  =  1/14,  \qquad w (5,1,2,3,1,5,4,4)  =  2/7.\]
\end{const}

\begin{const}
\label{primal:6-35}
The following weighting certifies that the value of $LP(T_{35},T_{6})$ is $4$.
\[w (1,2,1,4,2,3,4,2)  =  1/14,  \qquad w (1,2,1,5,2,3,5,5)  =  1/7,  \qquad w (1,4,1,4,2,3,5,5)  =  2/7, \]
\[w (2,3,2,3,3,2,3,3)  =  1/14.\]
\end{const}



\begin{const}
\label{primal:6-40}
The following weighting certifies that the value of $LP(T_{40},T_{6})$ is $4$.
\[w (2,1,5,4,1,4,5,3)  =  1/2,  \qquad w (2,3,2,2,3,2,2,3)  =  1/14.\]
\end{const}



\begin{const}
\label{primal:8-16}
The following weighting certifies that the value of $LP(T_{16},T_{8})$ is $5$.
\[w (1,2,3,4,5,6,2)  =  1/4,  \qquad w (2,1,5,6,3,4,1)  =  1/12,  \qquad w (2,1,5,6,3,4,3)  =  1/6, \]
\[w (3,2,3,4,2,1,4)  =  1/12,  \qquad w (3,4,3,4,2,1,4)  =  1/12,  \qquad w (5,1,5,6,1,5,6)  =  1/8, \]
\[w (5,6,5,6,6,5,6)  =  1/24.\]
\end{const}




\begin{const}
\label{primal:8-28}
The following weighting certifies that the value of $LP(T_{28},T_{8})$ is $5$.
\[w (1,2,3,4,1,2,3,4)  =  1/14,  \qquad w (1,5,1,5,6,2,3,4)  =  1/7,  \qquad w (2,1,5,6,2,1,5,6)  =  1/14, \]
\[w (2,3,2,3,4,1,5,6)  =  1/7,  \qquad w (3,2,3,4,1,4,3,4)  =  1/7,  \qquad w (5,1,5,6,2,6,5,6)  =  1/7.\]
\end{const}


\begin{const}
\label{primal:8-30}
The following weighting certifies that the value of $LP(T_{30},T_{8})$ is $5$.
\[w (1,5,6,5,6,5,6,5)  =  1/7,  \qquad w (2,1,5,6,2,3,4,1)  =  1/7,  \qquad w (3,2,3,4,1,2,1,4)  =  1/7, \]
\[w (3,2,3,4,3,4,3,4)  =  1/7,  \qquad w (5,1,5,6,2,1,2,6)  =  1/7.\]
\end{const}

\begin{const}
\label{primal:9-18}
The following weighting certifies that the value of $LP(T_{18},T_{9})$ is $5$.
\[w (1,2,3,4,2,4,3)  =  1/2,  \qquad w (1,5,1,1,5,1,1)  =  1/18,  \qquad w (5,1,5,5,1,5,5)  =  1/9, \]
\[w (6,5,6,6,5,6,6)  =  1/6.\]
\end{const}

\begin{const}
\label{primal:10-20}
The following weighting certifies that the value of $LP(T_{20},T_{10})$ is $5$.
\[w (1,2,3,3,5,6,6)  =  1/2,  \qquad w (2,1,5,5,4,1,4)  =  1/4,  \qquad w (2,4,2,2,4,4,4)  =  1/12.\]
\end{const}


\begin{const}
\label{primal:10-37}
The following weighting certifies that the value of $LP(T_{37},T_{10})$ is $5$.
\[w (1,2,3,3,1,2,3,3)  =  1/49,  \qquad w (1,2,4,3,3,2,4,3)  =  3/49,  \qquad w (1,2,4,4,4,2,4,1)  =  4/49, \]
\[w (2,1,5,5,2,1,5,5)  =  2/49,  \qquad w (2,1,6,6,6,1,6,5)  =  11/49,  \qquad w (3,2,4,4,4,2,3,3)  =  9/49, \]
\[w (5,1,5,5,5,1,5,6)  =  5/49.\]
\end{const}

\begin{const}
\label{primal:10-39}
The following weighting certifies that the value of $LP(T_{39},T_{10})$ is $5$.
\[w (1,2,4,3,3,6,5,5)  =  1/3,  \qquad w (1,2,4,3,4,2,6,6)  =  1/3,  \qquad w (1,5,1,1,1,5,5,5)  =  1/21.\]
\end{const}

\begin{const}
\label{primal:11-22}
The following weighting certifies that the value of $LP(T_{22},T_{11})$ is $5$.
\[w (1,2,3,2,3,6,6)  =  1/4,  \qquad w (1,4,5,4,5,6,6)  =  1/4,  \qquad w (2,1,4,1,4,3,3)  =  1/8, \]
\[w (2,3,2,3,2,3,3)  =  1/24,  \qquad w (4,1,2,1,2,5,5)  =  1/8,  \qquad w (4,5,4,5,4,5,5)  =  1/24.\]
\end{const}

\begin{const}
\label{primal:11-40}
The following weighting certifies that the value of $LP(T_{40},T_{11})$ is $5$.
\[w (1,2,3,3,4,5,5,6)  =  2/7,  \qquad w (2,1,6,4,1,4,6,3)  =  5/28,  \qquad w (2,3,2,2,3,2,2,3)  =  1/28, \]
\[w (4,1,6,2,1,6,2,5)  =  5/28,  \qquad w (4,5,4,4,5,4,4,5)  =  1/28.\]
\end{const}


\begin{const}
\label{primal:11-44}
The following weighting certifies that the value of $LP(T_{44},T_{11})$ is $5$.
\[w (1,2,3,2,3,2,3,2)  =  1/28,  \qquad w (1,2,3,2,3,4,5,6)  =  2/7,  \qquad w (1,2,3,4,5,2,3,4)  =  3/28, \]
\[w (1,4,5,6,1,4,5,6)  =  1/7,  \qquad w (2,3,2,1,6,1,6,3)  =  1/28,  \qquad w (4,5,4,1,6,1,6,5)  =  3/28.\]
\end{const}


\begin{const}
\label{primal:12-23}
The following weighting certifies that the value of $LP(T_{23},T_{12})$ is $5$.
\[w (1,2,1,2,2,2,2)  =  1/24,  \qquad w (1,2,1,4,5,5,4)  =  1/8,  \qquad w (1,2,3,6,4,4,6)  =  1/8, \]
\[w (1,2,3,6,6,5,5)  =  3/8,  \qquad w (1,4,1,4,4,4,4)  =  1/12,  \qquad w (2,3,2,3,3,3,3)  =  1/12.\]
\end{const}

\begin{const}
\label{primal:14-26}
The following weighting certifies that the value of $LP(T_{26},T_{14})$ is $6$.
\[w (1,2,3,4,4,5,6,7)  =  3/14,  \qquad w (1,5,1,2,5,2,1,5)  =  1/14,  \qquad w (1,5,6,7,7,2,3,4)  =  3/14, \]
\[w (2,1,5,6,6,3,2,1)  =  1/14,  \qquad w (3,4,3,4,4,2,3,4)  =  1/14,  \qquad w (5,1,2,3,3,1,2,3)  =  1/14, \]
\[w (5,6,5,6,6,1,2,3)  =  1/14,  \qquad w (6,5,6,7,7,7,6,7)  =  1/14.\]
\end{const}

\begin{const}
\label{primal:14-29}
The following weighting certifies that the value of $LP(T_{29},T_{14})$ is $6$.
\[w (1,2,3,4,3,5,6,6)  =  1/7,  \qquad w (2,1,2,3,5,3,4,4)  =  2/7,  \qquad w (2,1,5,1,2,3,4,4)  =  1/7, \]
\[w (5,6,5,1,7,6,7,7)  =  1/14,  \qquad w (5,6,7,6,7,6,7,7)  =  1/7,  \qquad w (6,5,6,7,1,5,1,1)  =  1/14.\]
\end{const}

\begin{const}
\label{primal:14-41}
The following weighting certifies that the value of $LP(T_{41},T_{14})$ is $6$.
\[w (1,2,1,3,5,6,2,1)  =  1/7,  \qquad w (1,2,3,3,5,6,5,6)  =  1/14,  \qquad w (2,3,2,4,3,4,3,4)  =  3/28, \]
\[w (2,3,4,2,1,2,1,2)  =  1/28,  \qquad w (2,3,4,4,3,4,1,2)  =  1/14,  \qquad w (4,3,4,4,3,4,3,2)  =  1/14, \]
\[w (5,1,2,5,6,7,6,7)  =  1/14,  \qquad w (5,6,7,7,6,7,1,5)  =  2/7.\]
\end{const}

\begin{const}
\label{primal:15-32}
The following weighting certifies that the value of $LP(T_{32},T_{15})$ is $6$.
\[w (1,2,1,5,5,7,7,7)  =  5/28,  \qquad w (1,2,3,2,5,7,7,7)  =  1/28,  \qquad w (1,2,3,4,5,6,6,6)  =  1/4, \]
\[w (1,5,1,2,5,6,7,6)  =  3/28,  \qquad w (2,1,2,3,3,4,4,4)  =  3/28,  \qquad w (2,1,5,6,3,4,4,4)  =  1/28, \]
\[w (3,2,3,4,2,3,3,3)  =  1/14,  \qquad w (4,3,4,3,3,4,4,4)  =  1/28,  \qquad w (7,5,7,5,5,7,7,7)  =  1/28.\]
\end{const}



\begin{const}
\label{primal:15-40}
The following weighting certifies that the value of $LP(T_{40},T_{15})$ is $6$.
\[w (1,5,7,6,5,7,6,2)  =  1/2,  \qquad w (2,3,4,4,3,4,4,1)  =  3/14,  \qquad w (3,2,1,1,2,1,1,4)  =  1/14, \]
\[w (3,2,3,3,2,3,3,4)  =  1/14.\]
\end{const}

\begin{const}
\label{primal:16-35}
The following weighting certifies that the value of $LP(T_{35},T_{16})$ is $6$.
\[w (1,2,1,7,5,6,7,7)  =  1/14,  \qquad w (1,2,3,4,5,6,2,7)  =  3/14,  \qquad w (1,2,3,4,5,6,7,7)  =  5/28, \]
\[w (1,5,1,5,5,6,7,7)  =  1/14,  \qquad w (2,1,5,6,3,4,3,3)  =  3/28,  \qquad w (3,2,1,7,4,3,4,4)  =  1/14, \]
\[w (3,2,3,4,2,1,4,4)  =  1/14,  \qquad w (5,6,5,6,1,5,6,6)  =  1/14.\]
\end{const}

\begin{const}
\label{primal:16-37}
The following weighting certifies that the value of $LP(T_{37},T_{16})$ is $6$.
\[w (1,2,1,3,3,2,3,3)  =  1/7,  \qquad w (2,1,5,5,7,3,4,4)  =  1/7,  \qquad w (2,3,4,4,4,3,4,4)  =  1/7, \]
\[w (5,1,2,5,2,1,2,2)  =  3/49,  \qquad w (5,1,2,7,7,1,7,7)  =  9/49,  \qquad w (5,1,7,5,5,1,7,7)  =  2/49, \]
\[w (6,5,6,6,6,5,6,6)  =  1/7.\]
\end{const}


\begin{const}
\label{primal:16-40}
The following weighting certifies that the value of $LP(T_{40},T_{16})$ is $6$.
\[w (1,5,6,6,5,6,6,5)  =  1/14,  \qquad w (2,1,5,5,1,7,7,3)  =  5/28,  \qquad w (2,1,7,2,1,2,7,3)  =  1/14, \]
\[w (2,3,4,4,3,4,4,3)  =  1/4,  \qquad w (5,1,2,7,1,7,2,6)  =  5/28,  \qquad w (5,1,7,7,1,7,7,6)  =  1/28, \]
\[w (6,5,6,6,5,6,6,5)  =  1/14.\]
\end{const}

\begin{const}
\label{primal:16-45}
The following weighting certifies that the value of $LP(T_{45},T_{16})$ is $6$.
\[w (1,5,6,5,6,7,7,2)  =  1/2,  \qquad w (2,3,2,3,2,3,3,3)  =  1/21,  \qquad w (2,3,4,3,4,1,1,1)  =  1/6, \]
\[w (3,2,3,2,3,4,4,4)  =  1/18,  \qquad w (3,4,3,4,3,4,4,4)  =  1/14.\]
\end{const}

\begin{const}
\label{primal:17-36}
The following weighting certifies that the value of $LP(T_{36},T_{17})$ is $6$.
\[w (1,2,1,3,3,3,6,7)  =  3/70,  \qquad w (1,2,1,5,5,5,2,5)  =  1/210,  \qquad w (1,2,3,5,5,3,6,1)  =  11/84, \]
\[w (1,2,4,3,4,3,2,5)  =  32/105,  \qquad w (1,2,4,5,4,5,6,7)  =  13/105,  \qquad w (1,2,4,5,5,4,6,1)  =  1/28, \]
\[w (4,2,4,5,5,5,2,5)  =  1/42,  \qquad w (6,1,6,6,6,6,1,6)  =  1/14,  \qquad w (7,6,7,7,7,7,6,7)  =  5/42.\]
\end{const}

\begin{const}
\label{primal:18-37}
The following weighting certifies that the value of $LP(T_{37},T_{18})$ is $6$.
\[w (1,2,1,4,4,5,6,6)  =  3/14,  \qquad w (1,5,7,6,6,5,6,6)  =  1/14,  \qquad w (1,5,7,6,7,2,3,3)  =  3/14, \]
\[w (1,5,7,7,7,2,4,4)  =  1/14,  \qquad w (1,5,7,7,7,5,7,6)  =  1/14,  \qquad w (2,1,2,2,2,1,2,5)  =  1/42, \]
\[w (2,1,5,5,5,1,5,2)  =  1/21,  \qquad w (3,2,4,4,3,2,4,3)  =  1/7.\]
\end{const}


\begin{const}
\label{primal:19-40}
The following weighting certifies that the value of $LP(T_{40},T_{19})$ is $6$.
\[w (1,2,3,4,2,4,3,7)  =  1/2,  \qquad w (1,5,6,6,5,6,6,5)  =  3/14,  \qquad w (5,1,7,7,1,7,7,6)  =  1/8.\]
\end{const}

\begin{const}
\label{primal:20-43}
The following weighting certifies that the value of $LP(T_{43},T_{20})$ is $6$.
\[w (1,2,3,3,2,5,6,5)  =  3/28,  \qquad w (1,2,3,3,6,6,7,7)  =  5/28,  \qquad w (1,2,4,4,7,5,5,7)  =  5/28, \]
\[w (1,2,4,4,7,7,6,2)  =  3/28,  \qquad w (1,6,1,1,6,6,6,6)  =  3/49,  \qquad w (2,1,5,5,3,4,3,4)  =  3/14.\]
\end{const}


\begin{const}
\label{primal:22-45}
The following weighting certifies that the value of $LP(T_{45},T_{22})$ is $6$.
\[w (1,2,3,2,3,7,7,7)  =  1/3,  \qquad w (1,2,3,4,5,4,6,6)  =  1/3,  \qquad w (4,1,6,1,6,5,5,5)  =  1/6, \]
\[w (5,4,5,4,5,4,4,4)  =  1/42.\]
\end{const}

\begin{const}
\label{primal:23-46}
The following weighting certifies that the value of $LP(T_{46},T_{23})$ is $6$.
\[w (1,2,3,4,4,5,5,7)  =  5/28,  \qquad w (1,2,3,4,4,7,5,6)  =  2/7,  \qquad w (1,2,3,6,5,6,5,6)  =  5/28, \]
\[w (1,2,3,7,2,6,6,7)  =  1/14,  \qquad w (1,7,1,7,6,4,7,6)  =  1/56,  \qquad w (1,7,1,7,7,7,7,4)  =  3/56, \]
\[w (2,1,2,3,3,3,3,1)  =  1/14.\]
\end{const}

\begin{const}
\label{primal:25-27}
The following weighting certifies that the value of $LP(T_{27},T_{25})$ is $7$.
\[w (2,3,4,5,5,1,2,2)  =  1/8,  \qquad w (2,3,4,5,5,1,6,6)  =  1/8,  \qquad w (3,2,3,4,4,4,5,5)  =  1/4, \]
\[w (6,1,2,1,3,7,8,8)  =  1/4,  \qquad w (6,1,6,7,7,7,8,8)  =  1/4.\]
\end{const}


\begin{const}
\label{primal:25-31}
The following weighting certifies that the value of $LP(T_{31},T_{25})$ is $7$.
\[w (1,6,7,8,6,7,8,2)  =  1/2,  \qquad w (2,3,4,5,3,4,5,1)  =  1/2.\]
\end{const}


\begin{const}
\label{primal:25-34}
The following weighting certifies that the value of $LP(T_{34},T_{25})$ is $7$.
\[w (3,2,1,6,4,5,4,5)  =  1/2,  \qquad w (6,1,2,3,7,8,7,8)  =  1/2.\]
\end{const}

\begin{const}
\label{primal:26-37}
The following weighting certifies that the value of $LP(T_{37},T_{26})$ is $7$.
\[w (1,2,3,1,3,6,7,7)  =  4/35,  \qquad w (1,6,7,1,7,2,1,1)  =  3/35,  \qquad w (1,6,7,7,1,6,7,7)  =  1/20, \]
\[w (2,1,6,6,6,1,6,2)  =  1/35,  \qquad w (2,3,5,4,4,3,4,5)  =  2/7,  \qquad w (2,3,5,5,4,1,2,2)  =  1/7, \]
\[w (2,3,5,5,5,3,5,5)  =  1/35,  \qquad w (6,1,6,6,6,1,6,6)  =  9/140,  \qquad w (6,7,8,8,8,7,8,8)  =  1/5.\]
\end{const}


\begin{const}
\label{primal:26-40}
The following weighting certifies that the value of $LP(T_{40},T_{26})$ is $7$.
\[w (2,3,4,4,3,5,5,1)  =  1/2,  \qquad w (6,1,2,2,1,2,6,7)  =  1/6,  \qquad w (6,7,8,8,1,6,6,7)  =  1/6, \]
\[w (6,7,8,8,7,8,8,7)  =  1/6.\]
\end{const}



\begin{const}
\label{primal:27-37}
The following weighting certifies that the value of $LP(T_{37},T_{27})$ is $7$.
\[w (1,2,1,1,1,2,1,1)  =  3/50,  \qquad w (1,2,3,3,3,2,3,3)  =  1/25,  \qquad w (1,6,7,7,7,6,7,7)  =  3/20, \]
\[w (1,6,8,8,8,6,8,7)  =  1/4,  \qquad w (2,1,2,2,2,3,5,5)  =  1/10,  \qquad w (2,3,4,4,4,3,4,4)  =  9/50, \]
\[w (2,3,5,5,4,1,6,6)  =  1/10,  \qquad w (2,3,5,5,5,3,5,5)  =  3/25.\]
\end{const}


\begin{const}
\label{primal:27-40}
The following weighting certifies that the value of $LP(T_{40},T_{27})$ is $7$.
\[w (1,6,8,7,6,7,8,2)  =  1/2,  \qquad w (2,3,5,4,3,4,5,1)  =  1/2.\]
\end{const}

\begin{const}
\label{primal:28-32}
The following weighting certifies that the value of $LP(T_{32},T_{28})$ is $7$.
\[w (1,6,7,8,2,5,5,3)  =  1/2,  \qquad w (2,1,6,7,3,4,4,4)  =  1/3,  \qquad w (6,1,2,3,7,8,8,8)  =  1/6.\]
\end{const}


\begin{const}
\label{primal:28-37}
The following weighting certifies that the value of $LP(T_{37},T_{28})$ is $7$.
\[w (1,2,5,5,1,2,5,5)  =  1/16,  \qquad w (1,2,5,5,5,6,7,7)  =  1/4,  \qquad w (1,6,1,1,1,6,1,1)  =  1/40, \]
\[w (1,6,7,7,7,2,3,3)  =  1/40,  \qquad w (2,3,4,4,4,3,4,4)  =  1/5,  \qquad w (6,1,6,6,2,7,8,8)  =  1/8, \]
\[w (6,1,6,6,6,1,6,6)  =  1/40,  \qquad w (6,7,8,8,8,7,8,8)  =  3/20.\]
\end{const}



\begin{const}
\label{primal:28-40}
The following weighting certifies that the value of $LP(T_{40},T_{28})$ is $7$.
\[w (1,2,5,3,2,5,3,6)  =  7/32,  \qquad w (1,2,5,5,2,5,5,6)  =  1/16,  \qquad w (1,6,1,1,6,1,1,6)  =  1/32, \]
\[w (1,6,7,1,2,1,1,6)  =  1/8,  \qquad w (1,6,7,7,2,3,5,6)  =  1/16,  \qquad w (2,3,4,4,3,4,4,5)  =  1/4, \]
\[w (6,7,8,8,7,8,8,7)  =  1/4.\]
\end{const}


\begin{const}
\label{primal:28-44}
The following weighting certifies that the value of $LP(T_{44},T_{28})$ is $7$.
\[w (1,2,5,2,5,2,5,6)  =  1/6,  \qquad w (2,3,4,3,4,1,6,5)  =  1/2,  \qquad w (6,7,8,7,8,7,8,1)  =  1/3.\]
\end{const}


\begin{const}
\label{primal:29-37}
The following weighting certifies that the value of $LP(T_{37},T_{29})$ is $7$.
\[w (1,2,1,5,3,6,8,8)  =  1/12,  \qquad w (1,2,3,5,5,6,7,7)  =  5/12,  \qquad w (1,6,8,7,8,2,1,3)  =  1/24, \]
\[w (1,6,8,8,7,6,8,8)  =  1/8,  \qquad w (1,6,8,8,8,2,3,5)  =  1/12,  \qquad w (2,3,4,4,4,1,6,2)  =  1/8, \]
\[w (2,3,4,4,4,3,4,4)  =  1/8.\]
\end{const}

\begin{const}
\label{primal:29-43}
The following weighting certifies that the value of $LP(T_{43},T_{29})$ is $7$.
\[w (2,3,4,4,3,3,3,5)  =  1/6,  \qquad w (2,3,4,4,5,1,5,1)  =  1/4,  \qquad w (2,3,4,4,5,5,5,5)  =  1/12, \]
\[w (6,1,2,2,7,8,7,8)  =  1/4,  \qquad w (6,1,6,6,7,7,7,7)  =  1/8,  \qquad w (6,1,6,6,8,8,8,8)  =  1/8.\]
\end{const}

\begin{const}
\label{primal:30-35}
The following weighting certifies that the value of $LP(T_{35},T_{30})$ is $7$.
\[w (1,2,3,4,6,7,8,8)  =  1/2,  \qquad w (2,1,6,7,3,4,5,5)  =  1/2.\]
\end{const}


\begin{const}
\label{primal:30-37}
The following weighting certifies that the value of $LP(T_{37},T_{30})$ is $7$.
\[w (1,2,5,5,3,6,7,7)  =  1/2,  \qquad w (2,1,6,8,8,3,4,4)  =  1/2.\]
\end{const}






\begin{const}
\label{primal:30-44}
The following weighting certifies that the value of $LP(T_{44},T_{30})$ is $7$.
\[w (1,2,5,6,7,6,7,8)  =  1/2,  \qquad w (2,1,8,1,8,3,4,5)  =  1/4,  \qquad w (2,3,4,3,4,3,4,5)  =  1/4.\]
\end{const}

\begin{const}
\label{primal:30-45}
The following weighting certifies that the value of $LP(T_{45},T_{30})$ is $7$.
\[w (1,6,7,6,7,8,2,8)  =  1/2,  \qquad w (2,3,4,3,4,5,5,1)  =  1/2.\]
\end{const}

\begin{const}
\label{primal:31-29}
The following weighting certifies that the value of $LP(T_{29},T_{31})$ is $7$.
\[w (2,1,2,3,8,3,4,4)  =  1/2,  \qquad w (5,1,5,6,8,6,7,7)  =  1/2.\]
\end{const}




\begin{const}
\label{primal:31-41}
The following weighting certifies that the value of $LP(T_{41},T_{31})$ is $7$.
\[w (2,1,8,5,3,4,3,4)  =  1/2,  \qquad w (5,1,2,2,6,7,6,7)  =  1/4,  \qquad w (5,1,8,8,6,7,6,7)  =  1/4.\]
\end{const}


\begin{const}
\label{primal:33-37}
The following weighting certifies that the value of $LP(T_{37},T_{33})$ is $7$.
\[w (1,2,3,3,3,2,3,3)  =  1/18,  \qquad w (1,5,6,7,7,5,6,6)  =  1/3,  \qquad w (1,5,7,7,7,2,3,3)  =  1/9, \]
\[w (2,1,8,2,8,3,4,4)  =  5/18,  \qquad w (2,1,8,8,5,3,4,4)  =  2/9.\]
\end{const}


\begin{const}
\label{primal:33-39}
The following weighting certifies that the value of $LP(T_{39},T_{33})$ is $7$.
\[w (1,5,6,6,7,5,8,8)  =  1/6,  \qquad w (1,5,7,7,6,8,8,2)  =  1/3,  \qquad w (2,3,4,4,4,1,1,1)  =  1/9, \]
\[w (3,2,3,3,3,4,4,4)  =  2/9,  \qquad w (5,1,2,2,5,6,6,7)  =  1/6.\]
\end{const}

\begin{const}
\label{primal:33-40}
The following weighting certifies that the value of $LP(T_{40},T_{33})$ is $7$.
\[w (1,5,6,6,5,6,6,2)  =  1/4,  \qquad w (1,5,7,7,5,7,7,2)  =  1/4,  \qquad w (2,1,8,8,1,8,8,3)  =  1/4, \]
\[w (2,3,4,4,3,4,4,3)  =  1/4.\]
\end{const}

\begin{const}
\label{primal:34-33}
The following weighting certifies that the value of $LP(T_{33},T_{34})$ is $7$.
\[w (1,2,3,4,5,6,6,5)  =  1/2,  \qquad w (1,2,3,4,7,8,8,7)  =  1/2.\]
\end{const}


\begin{const}
\label{primal:34-41}
The following weighting certifies that the value of $LP(T_{41},T_{34})$ is $7$.
\[w (1,5,6,6,5,6,5,6)  =  1/4,  \qquad w (1,7,8,8,7,8,7,8)  =  1/4,  \qquad w (2,1,2,5,3,4,3,4)  =  1/4, \]
\[w (2,1,7,2,3,4,3,4)  =  1/4.\]
\end{const}


\begin{const}
\label{primal:34-44}
The following weighting certifies that the value of $LP(T_{44},T_{34})$ is $7$.
\[w (1,5,6,7,8,5,6,2)  =  1/3,  \qquad w (1,7,8,5,6,7,8,2)  =  1/3,  \qquad w (2,3,4,3,4,3,4,1)  =  1/3.\]
\end{const}

\begin{const}
\label{primal:35-36}
The following weighting certifies that the value of $LP(T_{36},T_{35})$ is $7$.
\[w (1,5,6,6,6,6,2,3)  =  1/4,  \qquad w (2,1,5,5,5,7,3,4)  =  1/12,  \qquad w (2,1,7,7,7,7,3,4)  =  1/12, \]
\[w (2,1,8,8,7,5,3,4)  =  1/2,  \qquad w (2,3,4,4,4,4,1,7)  =  1/12.\]
\end{const}

\begin{const}
\label{primal:35-43}
The following weighting certifies that the value of $LP(T_{43},T_{35})$ is $7$.
\[w (1,2,1,1,2,2,2,2)  =  5/64,  \qquad w (1,2,1,1,5,5,5,5)  =  3/64,  \qquad w (1,2,3,3,5,7,5,2)  =  1/8, \]
\[w (1,5,6,6,2,8,8,5)  =  1/16,  \qquad w (1,5,6,6,8,7,7,8)  =  7/16,  \qquad w (3,2,3,3,4,4,4,4)  =  1/4.\]
\end{const}


\begin{const}
\label{primal:38-43}
The following weighting certifies that the value of $LP(T_{43},T_{38})$ is $7$.
\[w (2,1,2,2,4,4,4,4)  =  1/8,  \qquad w (2,1,6,6,4,5,5,4)  =  1/8,  \qquad w (2,1,8,8,3,5,3,5)  =  3/8, \]
\[w (2,1,8,8,4,4,3,3)  =  1/8,  \qquad w (6,1,6,6,7,7,7,7)  =  1/4.\]
\end{const}


\begin{const}
\label{primal:40-38}
The following weighting certifies that the value of $LP(T_{38},T_{40})$ is $7$.
\[w (1,2,3,4,4,2,3,8)  =  1/2,  \qquad w (1,5,7,6,6,5,7,8)  =  1/2.\]
\end{const}



\begin{const}
\label{primal:41-40}
The following weighting certifies that the value of $LP(T_{40},T_{41})$ is $7$.
\[w (1,2,3,4,2,4,3,7)  =  1/4,  \qquad w (1,2,4,4,5,6,6,5)  =  1/4,  \qquad w (1,5,6,6,2,3,3,5)  =  1/4, \]
\[w (1,7,8,8,7,8,8,7)  =  1/4.\]
\end{const}


\begin{const}
\label{primal:42-39}
The following weighting certifies that the value of $LP(T_{39},T_{42})$ is $7$.
\[w (1,2,3,4,3,8,7,7)  =  1/3,  \qquad w (2,1,2,8,8,3,4,4)  =  1/3,  \qquad w (5,1,5,5,5,6,6,6)  =  2/9, \]
\[w (5,1,7,7,7,6,6,6)  =  1/9.\]
\end{const}



\begin{const}
\label{primal:44-42}
The following weighting certifies that the value of $LP(T_{42},T_{44})$ is $7$.
\[w (1,2,3,3,6,7,6,4)  =  2/9,  \qquad w (1,4,5,5,2,3,4,2)  =  1/9,  \qquad w (1,4,5,5,2,3,8,8)  =  1/9, \]
\[w (1,4,5,5,4,5,2,6)  =  1/9,  \qquad w (1,4,5,5,6,7,4,8)  =  1/9,  \qquad w (1,6,7,7,2,3,8,8)  =  1/3.\]
\end{const}


\begin{const}
\label{primal:45-43}
The following weighting certifies that the value of $LP(T_{43},T_{45})$ is $7$.
\[w (1,2,3,3,2,4,4,8)  =  5/24,  \qquad w (1,2,3,3,2,6,8,6)  =  1/8,  \qquad w (1,2,3,3,7,8,8,7)  =  1/6, \]
\[w (1,4,5,5,2,8,2,8)  =  1/12,  \qquad w (1,4,5,5,4,8,6,8)  =  1/12,  \qquad w (1,4,5,5,7,7,6,6)  =  1/3.\]
\end{const}
}

\section{Dual Certificates}
\label{app:Dualcertificates}

We provide certificates for the value of $DLP(H,T)$ being less than $e(T)$ for various trees $H$ and $T$ on at most 8 vertices. Throughout, the vertices of all such trees are labelled as in Appendix~\ref{app:smallTrees}. Whenever we specify the weighting of $V(T)\cup E(T)$, any vertex or edge that is not listed is assumed to have weight zero.

\footnotesize{
\begin{const}
\label{dual:6-9}
The following weighting certifies that the value of $DLP(T_{9},T_{6})$ is less than $4$.
\[y (23)  =  5/6,  \qquad y (1)  =  5/3.\]
\end{const}


\begin{const}
\label{dual:6-15}
The following weighting certifies that the value of $DLP(T_{15},T_{6})$ is less than $4$.
\[y (23)  =  6/5,  \qquad y (1)  =  12/5.\]
\end{const}

\begin{const}
\label{dual:6-16}
The following weighting certifies that the value of $DLP(T_{16},T_{6})$ is less than $4$.
\[y (23)  =  6/5,  \qquad y (1)  =  12/5.\]
\end{const}

\begin{const}
\label{dual:6-19}
The following weighting certifies that the value of $DLP(T_{19},T_{6})$ is less than $4$.
\[y (12)  =  2/5,  \qquad y (23)  =  6/5,  \qquad y (1)  =  2.\]
\end{const}

\begin{const}
\label{dual:6-26}
The following weighting certifies that the value of $DLP(T_{26},T_{6})$ is less than $4$.
\[y (23)  =  7/6,  \qquad y (1)  =  7/3.\]
\end{const}


\begin{const}
\label{dual:6-41}
The following weighting certifies that the value of $DLP(T_{41},T_{6})$ is less than $4$.
\[y (23)  =  7/6,  \qquad y (1)  =  7/3.\]
\end{const}





\begin{const}
\label{dual:11-17}
The following weighting certifies that the value of $DLP(T_{17},T_{11})$ is less than $5$.
\[y (16)  =  6/11,  \qquad y (23)  =  9/11,  \qquad y (45)  =  9/11,  \qquad y (1)  =  18/11,  \qquad y (2)  =  6/11, \]
\[y (4)  =  6/11.\]
\end{const}

\begin{const}
\label{dual:11-19}
The following weighting certifies that the value of $DLP(T_{19},T_{11})$ is less than $5$.
\[y (23)  =  6/5,  \qquad y (45)  =  6/5,  \qquad y (1)  =  12/5.\]
\end{const}

\begin{const}
\label{dual:11-41}
The following weighting certifies that the value of $DLP(T_{41},T_{11})$ is less than $5$.
\[y (23)  =  7/6,  \qquad y (45)  =  7/6,  \qquad y (1)  =  7/3.\]
\end{const}

\begin{const}
\label{dual:12-36}
The following weighting certifies that the value of $DLP(T_{36},T_{12})$ is less than $5$.
\[y (23)  =  35/47,  \qquad y (1)  =  91/47,  \qquad y (2)  =  42/47.\]
\end{const}

\begin{const}
\label{dual:14-31}
The following weighting certifies that the value of $DLP(T_{31},T_{14})$ is less than $6$.
\[y (23)  =  7/27,  \qquad y (34)  =  7/9,  \qquad y (56)  =  7/27,  \qquad y (67)  =  7/9,  \qquad y (1)  =  7/9, \]
\[y (2)  =  28/27,  \qquad y (3)  =  14/27,  \qquad y (5)  =  28/27,  \qquad y (6)  =  14/27.\]
\end{const}

\begin{const}
\label{dual:14-35}
The following weighting certifies that the value of $DLP(T_{35},T_{14})$ is less than $6$.
\[y (34)  =  35/47,  \qquad y (67)  =  35/47,  \qquad y (1)  =  42/47,  \qquad y (2)  =  42/47,  \qquad y (3)  =  28/47, \]
\[y (5)  =  42/47,  \qquad y (6)  =  28/47.\]
\end{const}

\begin{const}
\label{dual:14-45}
The following weighting certifies that the value of $DLP(T_{45},T_{14})$ is less than $6$.
\[y (34)  =  49/67,  \qquad y (67)  =  49/67,  \qquad y (1)  =  63/67,  \qquad y (2)  =  56/67,  \qquad y (3)  =  42/67, \]
\[y (5)  =  56/67,  \qquad y (6)  =  42/67.\]
\end{const}

\begin{const}
\label{dual:15-20}
The following weighting certifies that the value of $DLP(T_{20},T_{15})$ is less than $6$.
\[y (12)  =  4/5,  \qquad y (23)  =  2/5,  \qquad y (34)  =  4/5,  \qquad y (2)  =  6/5,  \qquad y (3)  =  2/5, \]
\[y (5)  =  2.\]
\end{const}

\begin{const}
\label{dual:15-26}
The following weighting certifies that the value of $DLP(T_{26},T_{15})$ is less than $6$.
\[y (34)  =  14/17,  \qquad y (2)  =  28/17,  \qquad y (3)  =  7/17,  \qquad y 62/34  =  63/34.\]
\end{const}

\begin{const}
\label{dual:15-29}
The following weighting certifies that the value of $DLP(T_{29},T_{15})$ is less than $6$.
\[y (15)  =  7/27,  \qquad y (34)  =  7/9,  \qquad y (56)  =  14/27,  \qquad y (1)  =  14/27,  \qquad y (5)  =  35/27.\]
\end{const}

\begin{const}
\label{dual:15-41}
The following weighting certifies that the value of $DLP(T_{41},T_{15})$ is less than $6$.
\[y (12)  =  7/6,  \qquad y (23)  =  7/8,  \qquad y (34)  =  7/8,  \qquad y (3)  =  7/24,  \qquad y (5)  =  7/3.\]
\end{const}


\begin{const}
\label{dual:16-18}
The following weighting certifies that the value of $DLP(T_{18},T_{16})$ is less than $6$.
\[y (17)  =  6/11,  \qquad y (23)  =  2/3,  \qquad y (34)  =  31/33,  \qquad y (56)  =  9/11,  \qquad y (1)  =  18/11, \]
\[y (2)  =  6/11,  \qquad y (3)  =  2/11,  \qquad y (5)  =  6/11.\]
\end{const}

\begin{const}
\label{dual:16-19}
The following weighting certifies that the value of $DLP(T_{19},T_{16})$ is less than $6$.
\[y (15)  =  2/5,  \qquad y (56)  =  6/5,  \qquad y (1)  =  2,  \qquad y (3)  =  2.\]
\end{const}


\begin{const}
\label{dual:16-29}
The following weighting certifies that the value of $DLP(T_{29},T_{16})$ is less than $6$.
\[y (23)  =  56/69,  \qquad y (34)  =  70/69,  \qquad y (56)  =  35/46,  \qquad y (1)  =  119/69.\]
\end{const}

\begin{const}
\label{dual:16-32}
The following weighting certifies that the value of $DLP(T_{32},T_{16})$ is less than $6$.
\[y (15)  =  14/41,  \qquad y (23)  =  21/41,  \qquad y (34)  =  35/41,  \qquad y (56)  =  35/41,  \qquad y (1)  =  56/41, \]
\[y (2)  =  28/41,  \qquad y (5)  =  14/41.\]
\end{const}

\begin{const}
\label{dual:16-41}
The following weighting certifies that the value of $DLP(T_{41},T_{16})$ is less than $6$.
\[y (23)  =  7/6,  \qquad y (34)  =  7/6,  \qquad y (56)  =  7/6,  \qquad y (1)  =  7/3.\]
\end{const}

\begin{const}
\label{dual:17-37}
The following weighting certifies that the value of $DLP(T_{37},T_{17})$ is less than $6$.
\[y (16)  =  14/15,  \qquad y (67)  =  7/6,  \qquad y (2)  =  14/5.\]
\end{const}

\begin{const}
\label{dual:21-36}
The following weighting certifies that the value of $DLP(T_{36},T_{21})$ is less than $6$.
\[y (23)  =  14/17,  \qquad y (45)  =  14/17,  \qquad y (67)  =  14/17,  \qquad y (1)  =  28/17.\]
\end{const}

\begin{const}
\label{dual:21-41}
The following weighting certifies that the value of $DLP(T_{41},T_{21})$ is less than $6$.
\[y (23)  =  7/6,  \qquad y (45)  =  7/6,  \qquad y (67)  =  7/6,  \qquad y (1)  =  7/3.\]
\end{const}

\begin{const}
\label{dual:22-36}
The following weighting certifies that the value of $DLP(T_{36},T_{22})$ is less than $6$.
\[y (23)  =  14/17,  \qquad y (45)  =  14/17,  \qquad y (1)  =  28/17.\]
\end{const}

\begin{const}
\label{dual:22-38}
The following weighting certifies that the value of $DLP(T_{38},T_{22})$ is less than $6$.
\[y (23)  =  7/5,  \qquad y (45)  =  7/5,  \qquad y (1)  =  14/5.\]
\end{const}

\begin{const}
\label{dual:25-28}
The following weighting certifies that the value of $DLP(T_{28},T_{25})$ is less than $7$.
\[y (34)  =  2/19,  \qquad y (78)  =  10/19,  \qquad y (1)  =  12/19,  \qquad y (2)  =  24/19,  \qquad y (4)  =  27/19, \]
\[y (6)  =  20/19,  \qquad y (7)  =  1.\]
\end{const}

\begin{const}
\label{dual:25-35}
The following weighting certifies that the value of $DLP(T_{35},T_{25})$ is less than $7$.
\[y (34)  =  14/47,  \qquad y (45)  =  35/47,  \qquad y (67)  =  14/47,  \qquad y (78)  =  35/47,  \qquad y (1)  =  42/47, \]
\[y (2)  =  42/47,  \qquad y (3)  =  42/47,  \qquad y (4)  =  28/47,  \qquad y (6)  =  42/47,  \qquad y (7)  =  28/47.\]
\end{const}

\begin{const}
\label{dual:26-29}
The following weighting certifies that the value of $DLP(T_{29},T_{26})$ is less than $7$.
\[y (12)  =  2/5,  \qquad y (16)  =  2/5,  \qquad y (34)  =  3/10,  \qquad y (35)  =  3/10,  \qquad y (78)  =  4/5, \]
\[y (2)  =  3/5,  \qquad y (3)  =  2,  \qquad y (6)  =  3/5.\]
\end{const}

\begin{const}
\label{dual:26-41}
The following weighting certifies that the value of $DLP(T_{41},T_{26})$ is less than $7$.
\[y (12)  =  14/15,  \qquad y (3)  =  7/3.\]
\end{const}




\begin{const}
\label{dual:27-39}
The following weighting certifies that the value of $DLP(T_{39},T_{27})$ is less than $7$.
\[y (12)  =  7/4,  \qquad y (3)  =  7/4,  \qquad y (6)  =  7/4.\]
\end{const}



\begin{const}
\label{dual:28-26}
The following weighting certifies that the value of $DLP(T_{26},T_{28})$ is less than $7$.
\[y (25)  =  7/17,  \qquad y (34)  =  14/17,  \qquad y (78)  =  14/17,  \qquad y (1)  =  7/17,  \qquad y (2)  =  28/17, \]
\[y (3)  =  7/17,  \qquad y (6)  =  28/17,  \qquad y (7)  =  7/17.\]
\end{const}

\begin{const}
\label{dual:28-27}
The following weighting certifies that the value of $DLP(T_{27},T_{28})$ is less than $7$.
\[y (16)  =  7/15,  \qquad y (34)  =  7/6,  \qquad y (67)  =  7/15,  \qquad y (78)  =  7/5,  \qquad y (2)  =  7/3, \]
\[y (6)  =  14/15.\]
\end{const}

\begin{const}
\label{dual:28-29}
The following weighting certifies that the value of $DLP(T_{29},T_{28})$ is less than $7$.
\[y (16)  =  7/30,  \qquad y (25)  =  7/10,  \qquad y (34)  =  7/10,  \qquad y (67)  =  7/20,  \qquad y (1)  =  7/10, \]
\[y (2)  =  7/5,  \qquad y (3)  =  7/10,  \qquad y (6)  =  7/10,  \qquad y (8)  =  7/30.\]
\end{const}

\begin{const}
\label{dual:28-33}
The following weighting certifies that the value of $DLP(T_{33},T_{28})$ is less than $7$.
\[y (16)  =  7/8,  \qquad y (34)  =  7/6,  \qquad y (67)  =  7/12,  \qquad y (78)  =  7/8,  \qquad y (2)  =  7/3, \]
\[y (6)  =  7/12,  \qquad y (7)  =  7/24.\]
\end{const}

\begin{const}
\label{dual:28-35}
The following weighting certifies that the value of $DLP(T_{35},T_{28})$ is less than $7$.
\[y (23)  =  14/47,  \qquad y (25)  =  42/47,  \qquad y (34)  =  35/47,  \qquad y (67)  =  14/47,  \qquad y (78)  =  35/47, \]
\[y (1)  =  42/47,  \qquad y (2)  =  42/47,  \qquad y (3)  =  28/47,  \qquad y (6)  =  42/47,  \qquad y (7)  =  28/47.\]
\end{const}

\begin{const}
\label{dual:28-41}
The following weighting certifies that the value of $DLP(T_{41},T_{28})$ is less than $7$.
\[y (16)  =  7/6,  \qquad y (34)  =  7/6,  \qquad y (67)  =  7/8,  \qquad y (78)  =  7/8,  \qquad y (2)  =  7/3, \]
\[y (7)  =  7/24.\]
\end{const}


\begin{const}
\label{dual:29-38}
The following weighting certifies that the value of $DLP(T_{38},T_{29})$ is less than $7$.
\[y (23)  =  3/13,  \qquad y (25)  =  3/13,  \qquad y (34)  =  17/13,  \qquad y (67)  =  1/13,  \qquad y (68)  =  1/13, \]
\[y (1)  =  3/13,  \qquad y (2)  =  28/13,  \qquad y (6)  =  28/13.\]
\end{const}

\begin{const}
\label{dual:30-32}
The following weighting certifies that the value of $DLP(T_{32},T_{30})$ is less than $7$.
\[y (23)  =  7/22,  \qquad y (25)  =  49/66,  \qquad y (34)  =  28/33,  \qquad y (67)  =  28/33,  \qquad y (1)  =  35/33, \]
\[y (2)  =  35/33,  \qquad y (3)  =  14/33,  \qquad y (6)  =  14/33,  \qquad y (8)  =  14/33.\]
\end{const}


\begin{const}
\label{dual:31-32}
The following weighting certifies that the value of $DLP(T_{32},T_{31})$ is less than $7$.
\[y (18)  =  28/41,  \qquad y (23)  =  21/41,  \qquad y (34)  =  35/41,  \qquad y (56)  =  21/41,  \qquad y (67)  =  35/41, \]
\[y (1)  =  56/41,  \qquad y (2)  =  28/41,  \qquad y (3)  =  14/41,  \qquad y (5)  =  28/41,  \qquad y (6)  =  14/41.\]
\end{const}

\begin{const}
\label{dual:32-37}
The following weighting certifies that the value of $DLP(T_{37},T_{32})$ is less than $7$.
\[y (12)  =  49/53,  \qquad y (23)  =  49/53,  \qquad y (34)  =  49/53,  \qquad y (56)  =  7/53,  \qquad y (57)  =  7/53, \]
\[y (58)  =  7/53,  \qquad y (1)  =  14/53,  \qquad y (3)  =  14/53,  \qquad y (5)  =  161/53.\]
\end{const}


\begin{const}
\label{dual:32-43}
The following weighting certifies that the value of $DLP(T_{43},T_{32})$ is less than $7$.
\[y (12)  =  7/9,  \qquad y (23)  =  7/27,  \qquad y (34)  =  7/9,  \qquad y (2)  =  14/9,  \qquad y (3)  =  14/27, \]
\[y (5)  =  7/3.\]
\end{const}



\begin{const}
\label{dual:33-42}
The following weighting certifies that the value of $DLP(T_{42},T_{33})$ is less than $7$.
\[y (18)  =  28/43,  \qquad y (34)  =  7/43,  \qquad y (1)  =  63/43,  \qquad y (2)  =  28/43,  \qquad y (3)  =  63/43, \]
\[y (5)  =  91/43.\]
\end{const}

\begin{const}
\label{dual:34-29}
The following weighting certifies that the value of $DLP(T_{29},T_{34})$ is less than $7$.
\[y (23)  =  56/69,  \qquad y (34)  =  70/69,  \qquad y (56)  =  35/46,  \qquad y (78)  =  35/46,  \qquad y (1)  =  119/69.\]
\end{const}


\begin{const}
\label{dual:34-32}
The following weighting certifies that the value of $DLP(T_{32},T_{34})$ is less than $7$.
\[y (23)  =  2/5,  \qquad y (34)  =  4/5,  \qquad y (56)  =  4/5,  \qquad y (78)  =  4/5,  \qquad y (1)  =  8/5, \]
\[y (2)  =  3/5,  \qquad y (5)  =  3/5.\]
\end{const}

\begin{const}
\label{dual:35-37}
The following weighting certifies that the value of $DLP(T_{37},T_{35})$ is less than $7$.
\[y (17)  =  7/30,  \qquad y (18)  =  7/30,  \qquad y (23)  =  14/15,  \qquad y (34)  =  7/6,  \qquad y (56)  =  14/15, \]
\[y (1)  =  14/5,  \qquad y (2)  =  7/30,  \qquad y (5)  =  7/30.\]
\end{const}

\begin{const}
\label{dual:35-38}
The following weighting certifies that the value of $DLP(T_{38},T_{35})$ is less than $7$.
\[y (23)  =  7/25,  \qquad y (34)  =  7/25,  \qquad y (56)  =  7/5,  \qquad y (1)  =  14/5,  \qquad y (3)  =  42/25.\]
\end{const}




\begin{const}
\label{dual:41-36}
The following weighting certifies that the value of $DLP(T_{36},T_{41})$ is less than $7$.
\[y (23)  =  14/17,  \qquad y (24)  =  14/17,  \qquad y (56)  =  14/17,  \qquad y (78)  =  14/17,  \qquad y (1)  =  28/17, \]
\[y (2)  =  21/34,  \qquad y (5)  =  21/34,  \qquad y (7)  =  21/34.\]
\end{const}



\begin{const}
\label{dual:44-36}
The following weighting certifies that the value of $DLP(T_{36},T_{44})$ is less than $7$.
\[y (18)  =  21/34,  \qquad y (23)  =  14/17,  \qquad y (45)  =  14/17,  \qquad y (67)  =  14/17,  \qquad y (1)  =  28/17, \]
\[y (2)  =  21/34,  \qquad y (4)  =  21/34,  \qquad y (6)  =  21/34.\]
\end{const}


\begin{const}
\label{dual:44-38}
The following weighting certifies that the value of $DLP(T_{38},T_{44})$ is less than $7$.
\[y (18)  =  7/15,  \qquad y (23)  =  14/15,  \qquad y (45)  =  14/15,  \qquad y (67)  =  14/15,  \qquad y (1)  =  28/15, \]
\[y (2)  =  7/15,  \qquad y (4)  =  7/15,  \qquad y (6)  =  7/15.\]
\end{const}

}

\section{A List of All Small Trees}
\label{app:smallTrees}


In this appendix, we include diagrams of all non-empty trees on at most 8 vertices up to isomorphism with an arbitrary labelling of their vertices. These labellings of the vertices are used in the previous appendices to describe homomorphisms between trees as vectors. There are $47$ such trees, labelled $T_1,\dots,T_{47}$. The trees are listed in non-decreasing order in terms of their number of vertices.
The list of trees was generated using the \texttt{TreeIterator} class in SageMath~\cite{sagemath} and the pictures were generated using Graphviz~\cite{graphViz}. 


\begin{figure}[htbp]
\begin{tabular}{ccccccccccccccc}
\begin{minipage}{0.43333333333333335cm}
\centering
\includegraphics[scale=0.2]{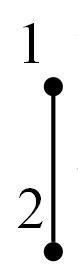}

$T_{1}$
\end{minipage}
&
\begin{minipage}{0.9388888888888889cm}
\centering
\includegraphics[scale=0.2]{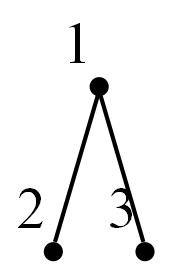}

$T_{2}$
\end{minipage}
&
\begin{minipage}{0.9388888888888889cm}
\centering
\includegraphics[scale=0.2]{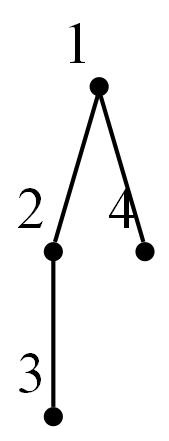}

$T_{3}$
\end{minipage}
&
\begin{minipage}{1.45cm}
\centering
\includegraphics[scale=0.2]{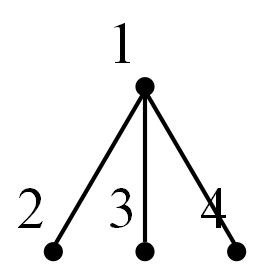}

$T_{4}$
\end{minipage}
&
\begin{minipage}{0.9388888888888889cm}
\centering
\includegraphics[scale=0.2]{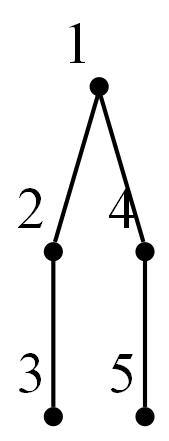}

$T_{5}$
\end{minipage}
&
\begin{minipage}{1.45cm}
\centering
\includegraphics[scale=0.2]{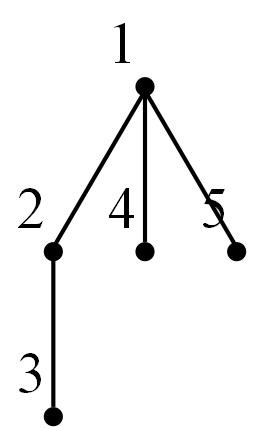}

$T_{6}$
\end{minipage}
&
\begin{minipage}{1.9611111111111112cm}
\centering
\includegraphics[scale=0.2]{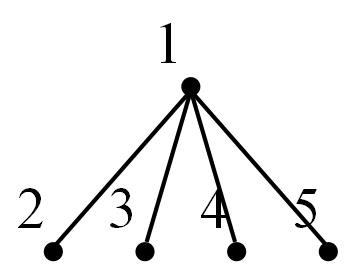}

$T_{7}$
\end{minipage}
&
\begin{minipage}{0.9388888888888889cm}
\centering
\includegraphics[scale=0.2]{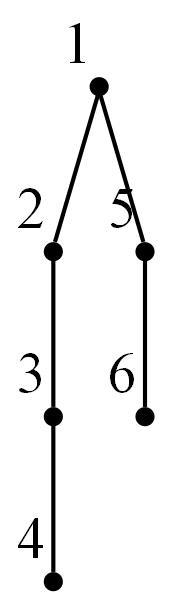}

$T_{8}$
\end{minipage}
\\
\begin{minipage}{1.45cm}
\centering
\includegraphics[scale=0.2]{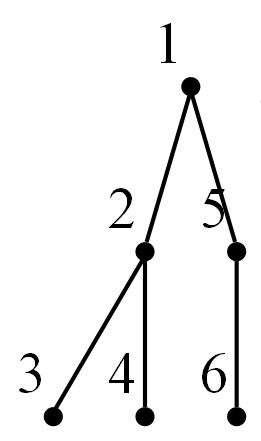}

$T_{9}$
\end{minipage}
&
\begin{minipage}{1.7055555555555557cm}
\centering
\includegraphics[scale=0.2]{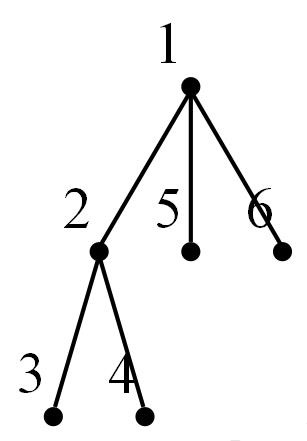}

$T_{10}$
\end{minipage}
&
\begin{minipage}{1.45cm}
\centering
\includegraphics[scale=0.2]{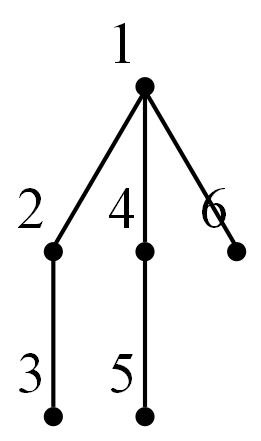}

$T_{11}$
\end{minipage}
&
\begin{minipage}{1.9611111111111112cm}
\centering
\includegraphics[scale=0.2]{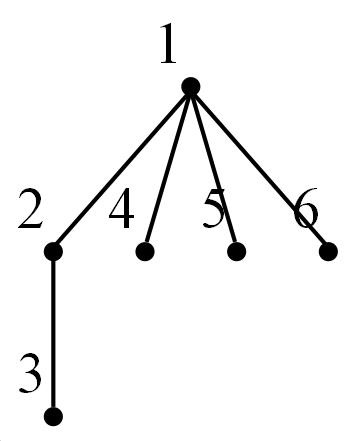}

$T_{12}$
\end{minipage}
&
\begin{minipage}{2.466666666666667cm}
\centering
\includegraphics[scale=0.2]{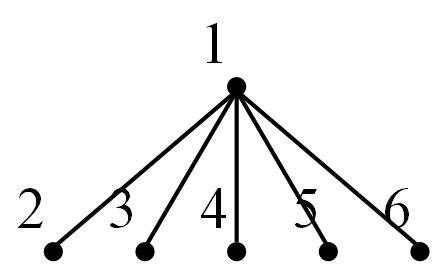}

$T_{13}$
\end{minipage}

\end{tabular}
\caption{The trees $T_1,\dots,T_{13}$. All non-empty trees on at most $6$ vertices, up to isomorphism.}
\label{fig:2to6}
\end{figure}
\begin{figure}[htbp]

\begin{tabular}{ccccccccccccccc}

\begin{minipage}{0.9388888888888889cm}
\centering
\includegraphics[scale=0.2]{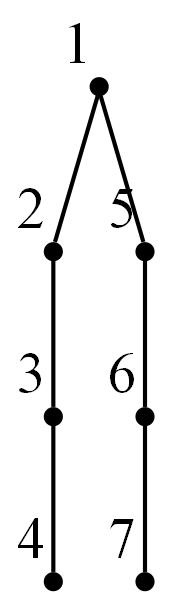}

$T_{14}$
\end{minipage}
&
\begin{minipage}{1.45cm}
\centering
\includegraphics[scale=0.2]{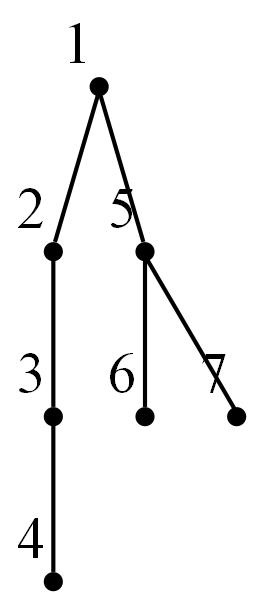}

$T_{15}$
\end{minipage}
&
\begin{minipage}{1.45cm}
\centering
\includegraphics[scale=0.2]{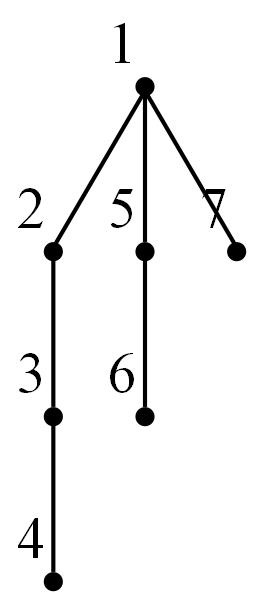}

$T_{16}$
\end{minipage}
&
\begin{minipage}{1.9611111111111112cm}
\centering
\includegraphics[scale=0.2]{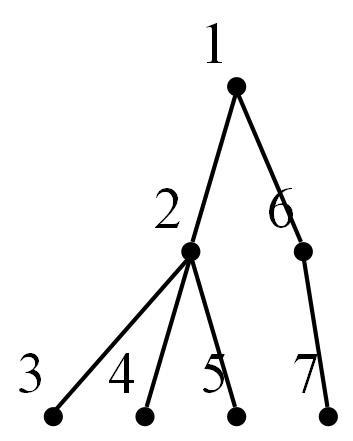}

$T_{17}$
\end{minipage}
&
\begin{minipage}{1.9611111111111112cm}
\centering
\includegraphics[scale=0.2]{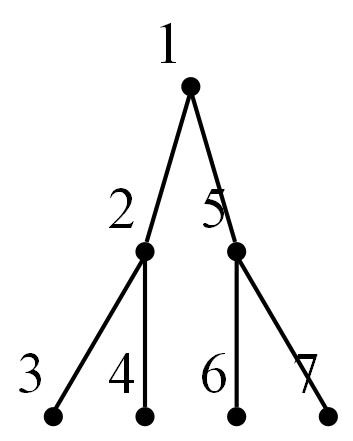}

$T_{18}$
\end{minipage}
&
\begin{minipage}{1.9611111111111112cm}
\centering
\includegraphics[scale=0.2]{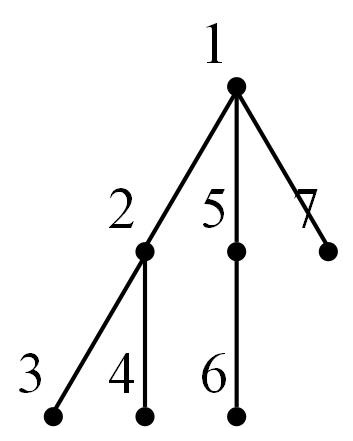}

$T_{19}$
\end{minipage}
\\
\begin{minipage}{2.2111111111111112cm}
\centering
\includegraphics[scale=0.2]{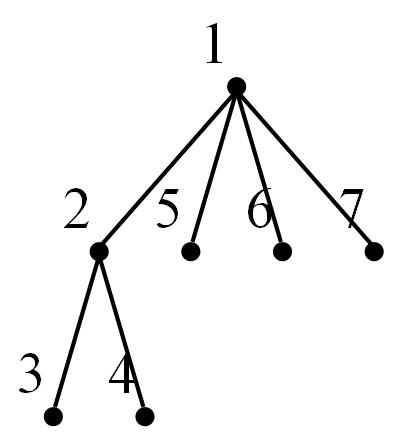}

$T_{20}$
\end{minipage}
&
\begin{minipage}{1.45cm}
\centering
\includegraphics[scale=0.2]{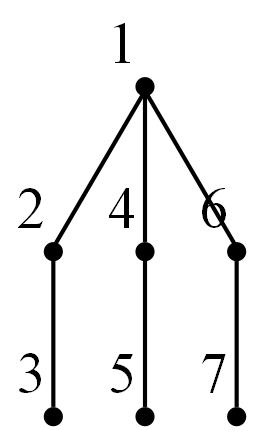}

$T_{21}$
\end{minipage}
&

\begin{minipage}{1.9611111111111112cm}
\centering
\includegraphics[scale=0.2]{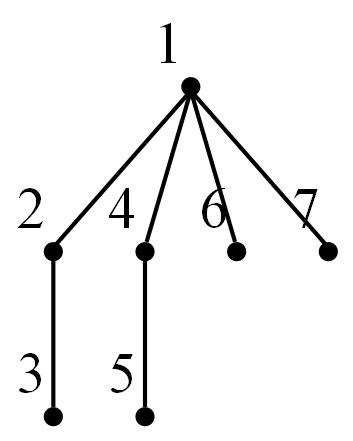}

$T_{22}$
\end{minipage}
&
\begin{minipage}{2.466666666666667cm}
\centering
\includegraphics[scale=0.2]{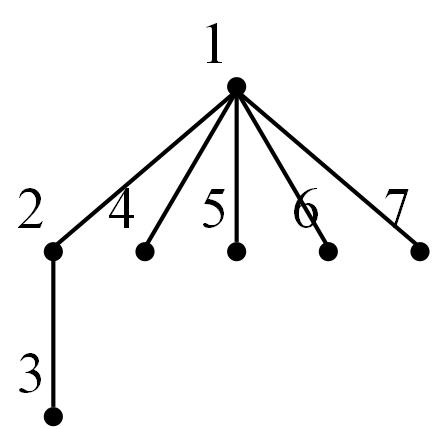}

$T_{23}$
\end{minipage}
&
\begin{minipage}{2.977777777777778cm}
\centering
\includegraphics[scale=0.2]{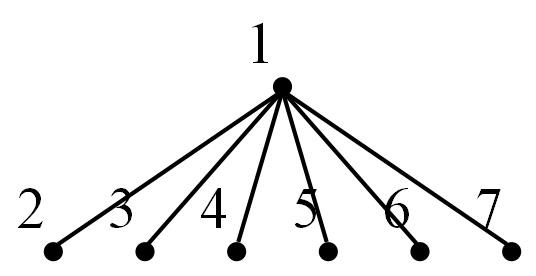}

$T_{24}$
\end{minipage}
\end{tabular}
\caption{The trees $T_{14},\dots,T_{24}$. All trees on 7 vertices, up to isomorphism. }
\label{fig:7}
\end{figure}

\begin{figure}[htbp]
\begin{tabular}{ccccccccccccc}

\begin{minipage}{0.9388888888888889cm}
\centering
\includegraphics[scale=0.2]{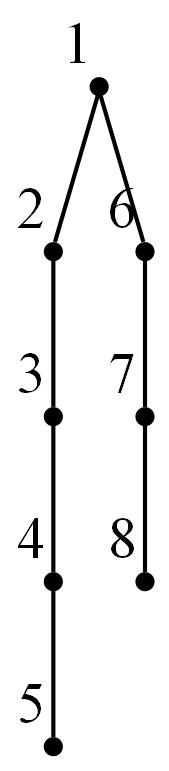}

$T_{25}$
\end{minipage}
&
\begin{minipage}{1.45cm}
\centering
\includegraphics[scale=0.2]{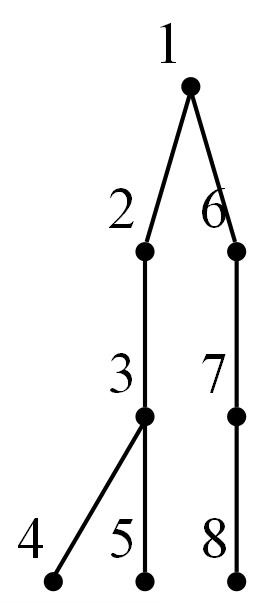}

$T_{26}$
\end{minipage}
&
\begin{minipage}{1.7055555555555557cm}
\centering
\includegraphics[scale=0.2]{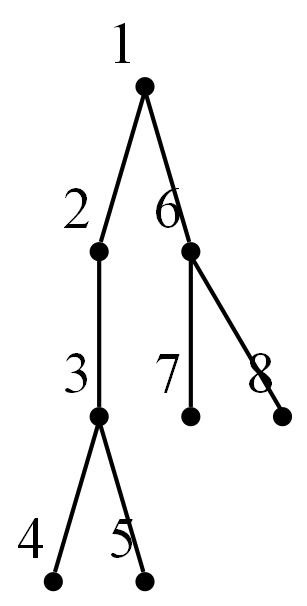}

$T_{27}$
\end{minipage}
&
\begin{minipage}{1.45cm}
\centering
\includegraphics[scale=0.2]{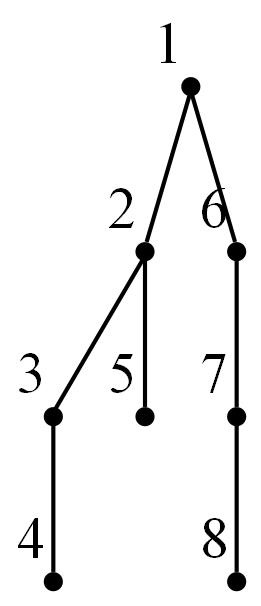}

$T_{28}$
\end{minipage}
&
\begin{minipage}{1.9611111111111112cm}
\centering
\includegraphics[scale=0.2]{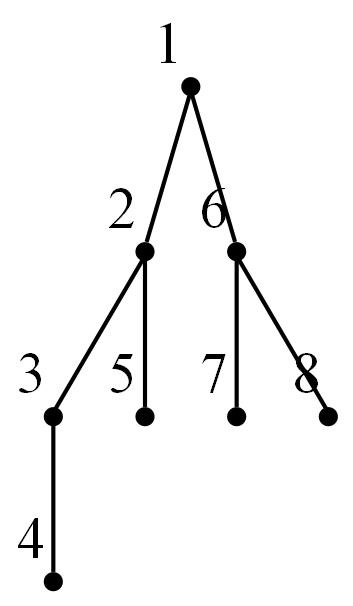}

$T_{29}$
\end{minipage}
\\
\begin{minipage}{1.9611111111111112cm}
\centering
\includegraphics[scale=0.2]{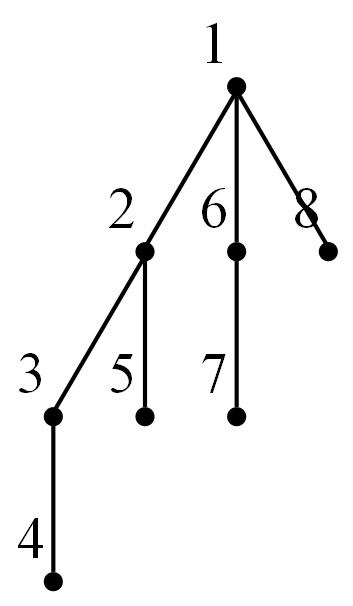}

$T_{30}$
\end{minipage}
&
\begin{minipage}{1.45cm}
\centering
\includegraphics[scale=0.2]{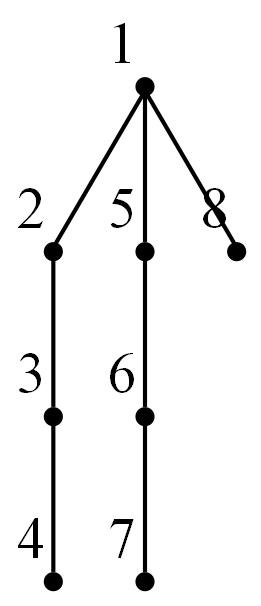}

$T_{31}$
\end{minipage}
&
\begin{minipage}{1.9611111111111112cm}
\centering
\includegraphics[scale=0.2]{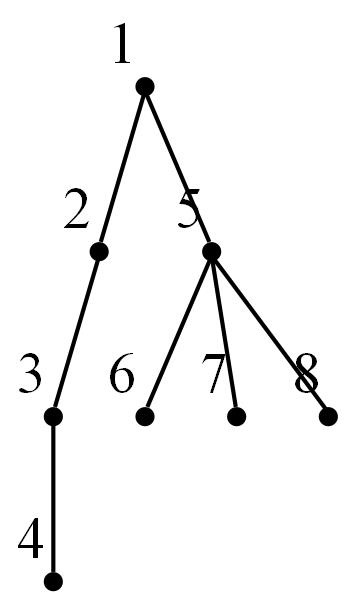}

$T_{32}$
\end{minipage}
&
\begin{minipage}{1.45cm}
\centering
\includegraphics[scale=0.2]{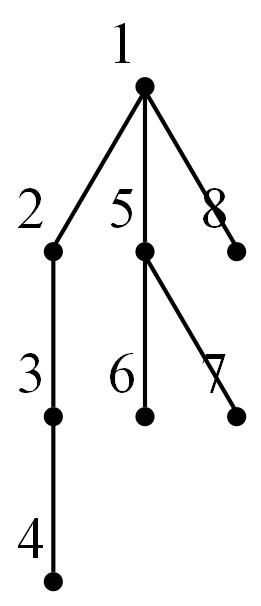}

$T_{33}$
\end{minipage}
&
\begin{minipage}{1.45cm}
\centering
\includegraphics[scale=0.2]{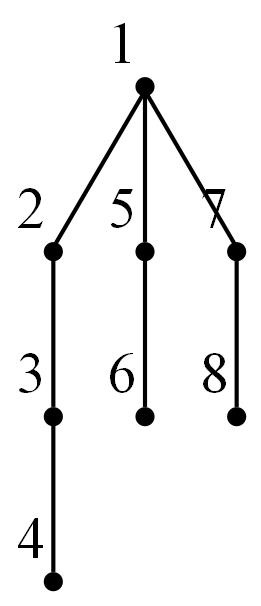}

$T_{34}$
\end{minipage}
\\
\begin{minipage}{1.9611111111111112cm}
\centering
\includegraphics[scale=0.2]{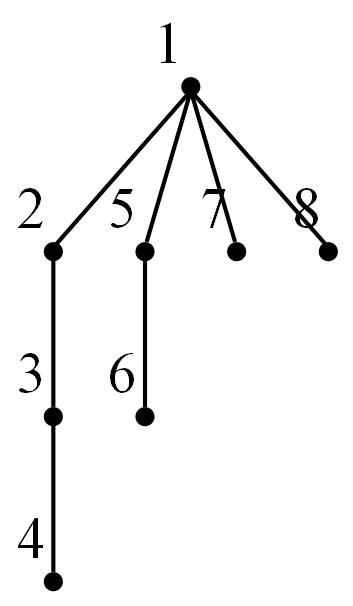}

$T_{35}$
\end{minipage}
&
\begin{minipage}{2.466666666666667cm}
\centering
\includegraphics[scale=0.2]{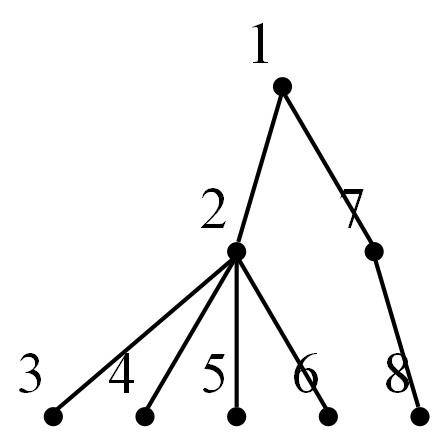}

$T_{36}$
\end{minipage}
&
\begin{minipage}{2.466666666666667cm}
\centering
\includegraphics[scale=0.2]{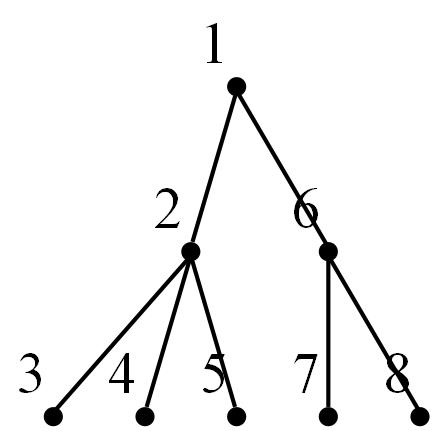}

$T_{37}$
\end{minipage}
&
\begin{minipage}{2.327777777777778cm}
\centering
\includegraphics[scale=0.2]{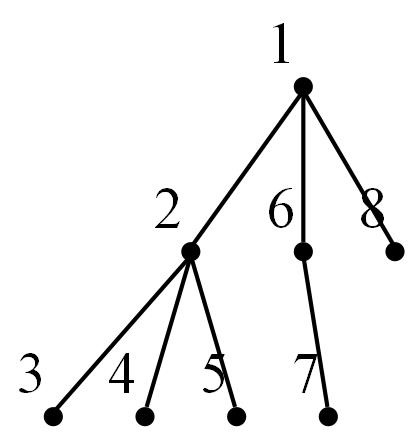}
$T_{38}$
\end{minipage}
&
\begin{minipage}{2.466666666666667cm}
\centering
\includegraphics[scale=0.2]{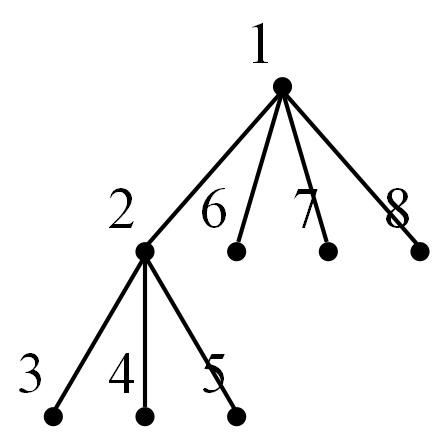}
$T_{39}$
\end{minipage}
\\
\begin{minipage}{1.9611111111111112cm}
\centering
\includegraphics[scale=0.2]{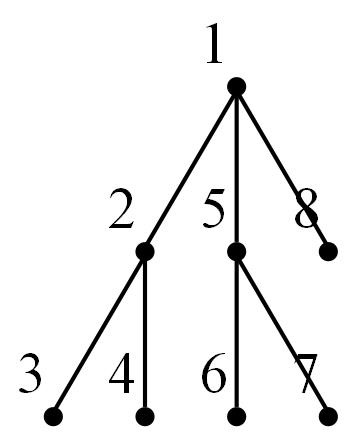}

$T_{40}$
\end{minipage}
&
\begin{minipage}{1.9611111111111112cm}
\centering
\includegraphics[scale=0.2]{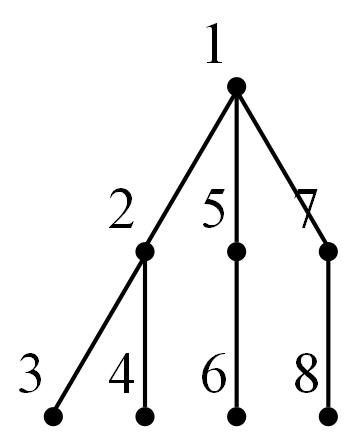}

$T_{41}$
\end{minipage}
&
\begin{minipage}{2.466666666666667cm}
\centering
\includegraphics[scale=0.2]{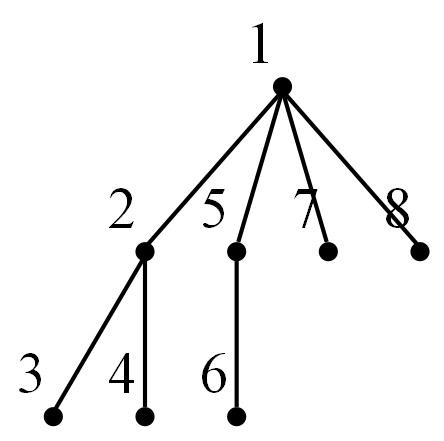}

$T_{42}$
\end{minipage}
&
\begin{minipage}{2.7222222222222223cm}
\centering
\includegraphics[scale=0.2]{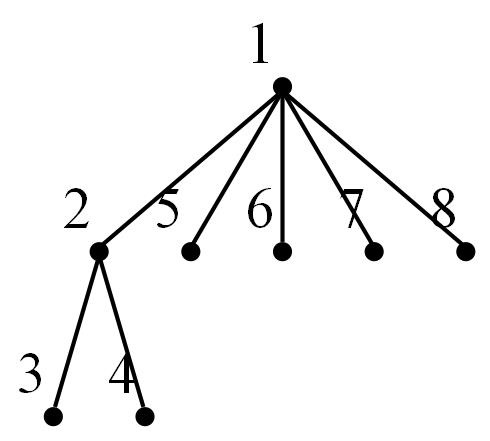}
$T_{43}$
\end{minipage}
&
\begin{minipage}{1.9611111111111112cm}
\centering
\includegraphics[scale=0.2]{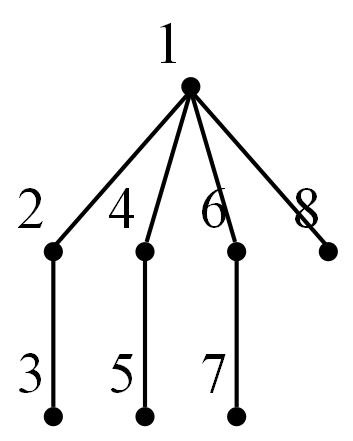}

$T_{44}$
\end{minipage}
\\
\begin{minipage}{2.466666666666667cm}
\centering
\includegraphics[scale=0.2]{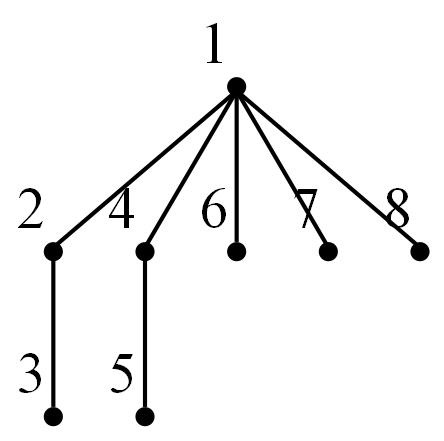}

$T_{45}$
\end{minipage}
&
\begin{minipage}{2.977777777777778cm}
\centering
\includegraphics[scale=0.2]{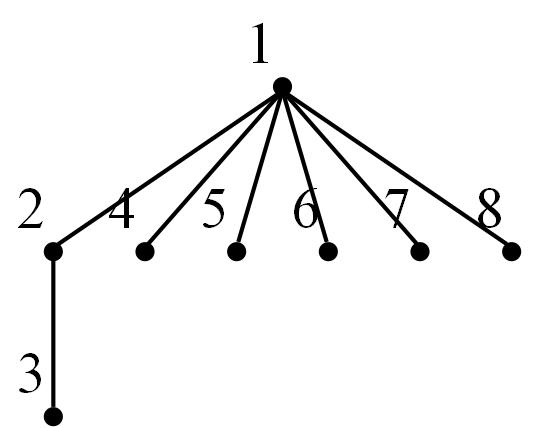}

$T_{46}$
\end{minipage}
&
\begin{minipage}{3.488888888888889cm}
\centering
\includegraphics[scale=0.2]{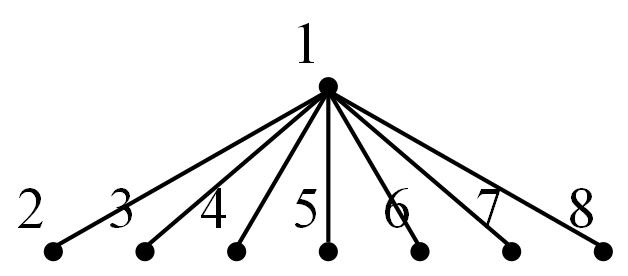}

$T_{47}$
\end{minipage}
\end{tabular}
\caption{The trees $T_{25},\dots,T_{47}$. All trees on 8 vertices up to isomorphism.}
\label{fig:8}
\end{figure}

\end{document}